\documentclass[11pt,twoside,leqno]{article}
\usepackage{amssymb}
\usepackage{amsmath}
\usepackage{mathrsfs}
\usepackage{amsthm}
\usepackage{amsfonts}
\usepackage{txfonts}
\usepackage{hyperref}
\usepackage{latexsym,amssymb}
\usepackage{color}
\usepackage{stmaryrd}
\usepackage{makecell}
\usepackage{cleveref}
\usepackage{enumitem}

\usepackage{indentfirst}
\usepackage{graphicx}
\usepackage{float}
\usepackage{pgf,tikz}
\usepackage{mathrsfs}
\usetikzlibrary{arrows}

\definecolor{ttttff}{rgb}{0.2,0.2,1.}
\definecolor{ttffcc}{rgb}{0.2,1.,0.8}
\definecolor{qqqqff}{rgb}{0.,0.,1.}

\definecolor{zzttqq}{rgb}{0.6,0.2,0.}
\definecolor{qqqqff}{rgb}{0.,0.,1.}

\allowdisplaybreaks

\usepackage{hyperref}
\hypersetup{
        colorlinks   = true,
        citecolor    = blue,
        linkcolor    = blue,
        urlcolor     = blue
}

\def\rr{{\mathbb R}}
\def\rn{{\mathbb{R}^n}}

\def\zz{{\mathbb Z}}

\def\nn{{\mathbb N}}

\def\cd{{\mathcal D}}

\def\ls{\lesssim}
\def\gs{\gtrsim}

\def\gfz{\genfrac{}{}{0pt}{}}

\def\loc{{\mathop\mathrm{\,loc\,}}}
\def\supp{\mathop\mathrm{\,supp\,}}

\newtheorem{theorem}{Theorem}[section]
\newtheorem{lemma}[theorem]{Lemma}
\newtheorem{Question}[theorem]{Problem}
\newtheorem{corollary}[theorem]{Corollary}
\newtheorem{proposition}[theorem]{Proposition}

\theoremstyle{definition}
\newtheorem{remark}[theorem]{Remark}
\newtheorem{definition}[theorem]{Definition}
\renewcommand{\appendix}{\par
   \setcounter{section}{0}%
   \setcounter{subsection}{0}%
   \setcounter{subsubsection}{0}%
   \gdef\thesection{\@Alph\c@section}%
   \gdef\thesubsection{\@Alph\c@section.\@arabic\c@subsection}%
   \gdef\theHsection{\@Alph\c@section.}%
   \gdef\theHsubsection{\@Alph\c@section.\@arabic\c@subsection}%
   \csname appendixmore\endcsname
 }

\numberwithin{equation}{section}

\definecolor{qqqqff}{rgb}{0,0,1}
\definecolor{ffqqqq}{rgb}{1,0,0}

\newcounter{rea}
\usepackage{makeidx}
\makeindex

\allowdisplaybreaks

\begin{document}

\arraycolsep=1pt

\title{\vspace{-1.8cm}
A Sharp Capacity Gauge Solution to the Open Problem of Quasisymmetric
Composition on $\mathcal{Q}_\alpha(\mathbb R^n)$}
\footnotetext{\hspace{-0.35cm}
\endgraf {\it Key words and phrases}. $\mathcal Q_\alpha(\mathbb R^n)$ space, quasisymmetric mapping, capacity gauge, composition operator,
dyadic decomposition, mean oscillation, wavelets.
\endgraf L. Liu is supported by the National Natural Science Foundation of China
(\# 12371102); J. Xiao is supported by NSERC of Canada (\# 202979).}
\author{Liguang Liu and Jie Xiao}
\date{}
\maketitle
% 42B35: Function spaces arising in harmonic analysis.
% 30C65: Quasiconformal mappings in $\mathbb{R}^n$, quasiregular mappings, and other generalizations.
% 47B33: Composition operators on function spaces.
% 31B15: Capacities and other potential-theoretic quantities in $\mathbb{R}^n$.
% 42C40: Wavelets and other multiscale systems.
%
%A Sharp Capacity Gauge Solution to the Quasisymmetric
%Pullback  Open Problem for
%$\mathcal{Q}_\alpha(\mathbb R^n)$
\begin{center}
\begin{minipage}{13.8cm}
{\small {\bf Abstract.}
 The spaces $\mathcal Q_\alpha(\mathbb R^n)$ for
 the critical index range $\alpha\in(0,\min\{1,n/2\})$
form a scale-invariant family  lying strictly between
the space of constant functions and
 $\mathrm{BMO}(\mathbb R^n)$.  A long-standing open problem,
 posed by Ess\'en, Janson, Peng and Xiao in 2000,
  asks to characterize the quasisymmetric mappings $\varphi:\ \mathbb R^n\to\mathbb R^n$
 for which the composition operator $\mathcal C_\varphi(f)=f\circ\varphi^{-1}$
 is bounded on $\mathcal Q_\alpha(\mathbb R^n)$.
This paper establishes the full intrinsic characterization
by introducing a new geometric quantity,
the capacity gauge $\|\varphi\|_{\mathscr G_\beta}$,
which measures the distortion of $\varphi$ through  pullback volume ratios 
along the dyadic tree.
We prove that for any quasisymmetric mapping $\varphi$
(with a mild Muckenhoupt \(A_\infty(\mathbb R)\) assumption on the Jacobian determinants of \(\varphi\) and \(\varphi^{-1}\) when \(n=1\)),
the composition operator  $\mathcal C_\varphi$ is bounded on $\mathcal Q_\alpha(\mathbb R^n)$ if and only if
$\|\varphi\|_{\mathscr G_{1-\frac{2\alpha}{n}}}<\infty,$
with quantitative equivalence
$$
\|\mathcal C_\varphi\|_{\mathcal Q_\alpha(\mathbb R^n)\to\mathcal Q_\alpha(\mathbb R^n)}^2 \simeq \|\varphi\|_{\mathscr G_{1-\frac{2\alpha}{n}}}.
$$
This sharp, necessary and sufficient characterization
refines earlier sufficient criteria of Koskela, Xiao, Zhang and Zhou in 2017,
which were formulated in terms of local or global self-similar Minkowski dimension
of the exceptional sets of the Jacobian.
As applications, we obtain
the composition stability of the finite capacity gauge, the invariance of $\mathcal Q_\alpha$-removability under quasisymmetric mappings with finite capacity gauge, and the propagation  of $\mathcal Q_\alpha(\mathbb R^n)$-regularity and initial-data stability for transport equations driven by quasisymmetric flows.
}
\end{minipage}
\end{center}

\tableofcontents

\section{Introduction}\label{S1}

\subsection{The real-variable Euclidean $\mathcal Q_\alpha(\mathbb R^n)$ spaces}\label{ss1.1}

Aulaskari, Xiao and Zhao \cite{AulaskarXiaoZhao1995Analysis} introduced
the M\"obius-invariant analytic $Q_p$ spaces on the unit disk $\mathbb D$,
with $\mathrm{BMOA}$ and the Bloch space as endpoint cases.
Ess\'en and Xiao \cite{EssenXiao1997Crelle} later provided a foundational study
of these spaces for the critical range $p\in(0,1)$, establishing them
as proper subspaces of $\mathrm{BMOA}$.
The real-variable Euclidean
$\mathcal Q_\alpha(\mathbb R^n)$ spaces were subsequently introduced by
Ess\'en, Janson, Peng and Xiao \cite{EssenJansonPengXiao2000IUMJ} and
developed in the twin papers of Dafni and Xiao
\cite{DafniXiao2004JFA,DafniXiao2005TMJ}, as the all-dimensional version
of the real counterpart of the analytic $Q_p$ spaces. We now recall the definition of $\mathcal Q_\alpha(\mathbb R^n)$.

\begin{definition}\label{def-Qspace}
Let $\alpha \in \mathbb{R}$.
The space $\mathcal Q_\alpha(\mathbb{R}^n)$ is defined to be the collection of all
 $f \in L^2_{\text{loc}}(\mathbb{R}^n)$ such that
\begin{align}\label{eq-Q-norm}
\|f\|_{\mathcal Q_\alpha(\rn)} :=\sup_{\text{cubes}\, Q\subset\rn} \left(|Q|^{\frac{2\alpha}{n}-1} \int_{Q} \int_{
Q} \frac{|f(x) - f(y)|^2}{|x - y|^{n + 2\alpha}} \, dx \, dy\right)^\frac 12<\infty,
\end{align}
where the supremum is taken over all cubes $Q \subset \mathbb{R}^n$
with sides parallel to the coordinate axes, and $|Q|$ denotes the Lebesgue measure of $Q$.
\end{definition}

Since every cube \(Q\) with centre \(c_Q\) and side length \(\ell(Q)\) satisfies the two-sided ball inclusion
$$
B\left(c_Q,\, \ell(Q)/2\right) \subset Q \subset B\left(c_Q,\, \sqrt{n}\,\ell(Q)/2\right),
$$
the cube \(Q\) in \eqref{eq-Q-norm} can be equivalently replaced by a ball, yielding an equivalent seminorm on the same space \(\mathcal Q_\alpha(\mathbb R^n)\). Modulo constants, the quotient space \(\mathcal Q_\alpha(\mathbb R^n)/\mathbb R\) is a Banach space.

According to \cite{EssenJansonPengXiao2000IUMJ},
the parameter range of interest for $\mathcal Q_\alpha(\rn)$ is
\begin{align}\label{eq-alpha-range}
0<\alpha <\min\left\{1,\ \frac n 2\right\}.
\end{align}
Indeed, for $\alpha\in(-\infty, 0)$, one has the identification
$$\mathcal{Q}_\alpha(\mathbb R^n) = \mathrm{BMO}(\mathbb{R}^n),$$
where $\mathrm{BMO}(\mathbb{R}^n)$ is the classical \emph{space of functions of bounded mean oscillation} introduced by John and Nirenberg \cite{JohnNirenberg1961CPAM}. This space consists of all $f \in L^1_{\text{loc}}(\mathbb{R}^n)$ for which
$$
\|f\|_{\mathrm{BMO}(\rn)} :=\sup_{\text{cubes}\, Q\subset\rn} \frac 1{|Q|} \int_{Q}\left|f(x)-f_Q\right| \, dx<\infty,
$$
where $f_Q:=\frac1{|Q|}\int_Q f(x)\,dx$.
In contrast, for $n\ge 2$ and $\alpha\in[1,\infty)$, or for $n=1$ and $\alpha\in(\tfrac12,\infty)$, one has
$$
\mathcal{Q}_\alpha(\rn)=\{\text{constants}\}.
$$
For $n=1$ and $\alpha=\frac 12$, the space $\mathcal Q_{\frac 12}(\rr)$ is non-trivial and
coincides with the Besov space $\dot \Lambda^{\frac 12}_{2,\, 2}(\rr)$, where for $s\in(0,1)$ and $p, q\in [1,\infty)$, the \emph{Besov space} $\dot \Lambda^{s}_{p,\, q}(\rn)$ is defined via the seminorm
$$
\|f\|_{\dot \Lambda^{s}_{p,\, q}(\rn)}:=\left(\int_{\rn}\left(\int_{\rn}|f(x+h)-f(x)|^p\,dx\right)^\frac qp\,\frac{dh} {|h|^{n+sq}}\right)^\frac 1q.
$$
Furthermore, for $0<\gamma<\alpha<1$, we have the strict chain of inclusions
\begin{align}\label{eq-Q-chain}
\{\text{constants}\}\subsetneqq \dot \Lambda^{\alpha}_{n/\alpha,\,2}(\rn)\subsetneqq \mathcal Q_{\alpha}(\rn)\subsetneqq \mathcal  Q_{\gamma}(\rn)\subsetneqq \mathcal Q_{0}(\rn) \subsetneqq \mathrm{BMO}(\rn).
\end{align}
Note that all spaces in \eqref{eq-Q-chain} are scale-invariant.
For more elementary properties of $\mathcal Q_\alpha(\rn)$ spaces, we refer the reader to \cite{EssenJansonPengXiao2000IUMJ, DafniXiao2004JFA,DafniXiao2005TMJ, XiaoQbook2019, YueDafni2009JMAA}.

As can be seen from \eqref{eq-Q-chain}, the family $\mathcal Q_\alpha(\mathbb R^n)$ with $\alpha$ satisfying \eqref{eq-alpha-range} serves as a fractional oscillation bridge between $\mathrm{BMO}(\mathbb R^n)$ and the Besov space $\dot \Lambda^{\alpha}_{n/\alpha,\,2}(\mathbb R^n)$.
The spaces $\mathcal Q_\alpha(\mathbb R^n)$, Besov spaces and $\mathrm{BMO}(\rn)$ all fall within the same framework of Besov--Triebel--Lizorkin-type spaces (see \cite{YangYuan2008JFA, YangYuan2010MZ, YuanSickelYang2010LNM}). However, the study of $\mathcal Q_\alpha(\mathbb R^n)$ spaces inherently requires a capacity-theoretic perspective that is not essential for Besov or $\mathrm{BMO}$ spaces.
For instance, the predual of $\mathcal Q_\alpha(\mathbb R^n)$ (see \cite{DafniXiao2004JFA}) is characterized in terms of Hausdorff capacity, revealing the deep connections between $\mathcal Q_\alpha(\mathbb R^n)$ spaces and capacity theory. This capacity-theoretic nature is precisely what distinguishes $\mathcal Q_\alpha$ spaces from the other two classes, and it constitutes the central theme of the present article.

Beyond these structural considerations,  $\mathcal Q_\alpha(\rn)$ spaces have been applied to a number of PDE problems, particularly in contexts where classical Sobolev or Lebesgue spaces are not directly applicable due to the oscillatory or rough nature of the data. Typical examples include the transport equation \cite{Xiao2019JDE, XiaoQbook2019}, the Navier--Stokes and related fluid systems \cite{LiZhai2010JFA, XiaoYangZhangZhou2019ADE, XiaoZhang2019JGA}, as well as harmonic maps and liquid crystal models \cite{XiaoZhang2019DPDE}.
In such settings, the scale-invariant structure and the fractional parameter \(\alpha\) serve as flexible tools for quantifying oscillation at the scaling threshold.

\subsection{An open problem on quasisymmetric 
$\mathcal{Q}_\alpha(\mathbb R^n)$-composition}\label{ss1.2}

The spaces $\mathcal Q_\alpha(\rn)$ are natural
candidates for studying function behavior under quasisymmetric
or quasiconformal mappings, owing to their scale-invariant structure.

\begin{definition}\label{def-QS}
Let $n\in\nn$.
A homeomorphism $\varphi:\ \mathbb{R}^n \to \mathbb{R}^n$ is called \emph{$\eta$-quasisymmetric} if there exists an increasing homeomorphism $\eta:\ [0,\infty) \to [0,\infty)$ with $\eta(0)=0$ such that for all distinct $x, y, z \in \mathbb{R}^n$,
\begin{align}\label{eq-QSM}
\frac{|\varphi(x) - \varphi(y)|}{|\varphi(x) - \varphi(z)|} \leq \eta\left( \frac{|x - y|}{|x - z|} \right).
\end{align}
The function \(\eta\) is called the \emph{distortion} function of \(\varphi\).
\end{definition}

By replacing $\eta$ with $\tilde\eta(t):=\sup_{0<s\le t}\eta(s)$ if necessary, one may always assume that $\eta$ is increasing.
Some known facts on quasisymmetric mappings are recalled in Section \ref{sec2},
with particular emphasis on their close relationship with classical Muckenhoupt weights on \(\mathbb R^n\).

For $n\ge 2$, it is classical \cite{Heinonen2001book, Koskela2009note}
that quasisymmetric mappings are exactly quasiconformal mappings.
Recall that a homeomorphism $\varphi$ on $\mathbb R^n$ is called \emph{$K$-quasiconformal}
for some constant $K\in[1,\infty)$ if
$$
|D\varphi(x)|^n \le K J_\varphi(x)\quad \text{for a.e. } x\in\mathbb R^n,
$$
where $J_\varphi$ is the \emph{Jacobian} of $\varphi$, and $$|D\varphi(x)|:=\sup_{|h|\le 1} |D\varphi(x)h|$$ denotes the operator norm of the Jacobian matrix.
To include the case $n=1$ (where the quasiconformal theory is less standard), we adopt the notion of quasisymmetric mappings throughout the paper.

Let $\varphi:\ \mathbb R^n \to \mathbb R^n$ be quasisymmetric. Consider the pullback composition operator $\mathcal C_\varphi$, defined by
$$
\mathcal C_\varphi(f) := f \circ \varphi^{-1}.
$$
For dimension $n\ge 2$, the following open problem was posed in \cite[Problem 8.4]{EssenJansonPengXiao2000IUMJ}:
\begin{center}
\begin{minipage}{13cm}
\emph{Let $\alpha\in(0,1)$. Prove or disprove that $\mathcal C_\varphi$
is bounded on $\mathcal Q_\alpha(\rn)$.
}
\end{minipage}
\end{center}
Invoking the case of dimension $n=1$, this naturally suggests the following refined version of this open problem.

\begin{Question}\label{openProb}
Let $\alpha\in(0,\, \min\{1,\, n/2\})$. What are the necessary and sufficient conditions on the quasisymmetric mapping $\varphi$ for $\mathcal C_\varphi$ to be bounded on $\mathcal Q_\alpha(\rn)$?
\end{Question}

\begin{remark} Let us explain the reason for the parameter restriction $\alpha\in(0,\min\{1,n/2\})$ in Problem \ref{openProb}.
For parameters outside this range, we have the identifications (see Subsection \ref{ss1.1})
$$
\mathcal Q_\alpha(\rn)=
\begin{cases}\mathrm{BMO}(\rn) &\quad\text{for } \alpha\in(-\infty,0);\\
\{\text{constants}\} &\quad\text{for }\, \alpha\in(\min\{1,\,n/2\},\,\infty)\, \text{ or }\, \alpha=1.
\end{cases}
$$
Clearly, nothing needs to be studied when $\mathcal Q_\alpha(\rn)=\{\text{constants}\}$.
For the case $\mathcal Q_\alpha(\rn)=\mathrm{BMO}(\rn)$, the boundedness of $\mathcal C_\varphi$ is already fully understood:
\begin{enumerate}[label=\textup{(\roman*)}]
  \item for $n\ge 2$, it follows from \cite{Reimann1974CMH} that $\mathcal C_\varphi$ is bounded on $\mathrm{BMO}(\rn)$ if and only if $\varphi$ is quasisymmetric;
  \item for $n=1$, by \cite[Theorem]{Jones1983ArkMat}, if $\varphi$ is increasing on $\mathbb R$, then $\mathcal C_\varphi$ is bounded on $\mathrm{BMO}(\rr)$ precisely when $\varphi'$ belongs to the Muckenhoupt class $A_\infty(\mathbb R)$.
\end{enumerate}
Note that the boundedness of \(\mathcal C_\varphi\) on the Besov space \(\dot\Lambda^\alpha_{n/\alpha,\,2}(\mathbb R^n)\) (see \eqref{eq-Q-chain}) is encompassed by the complete study of quasiconformally invariant Besov spaces \(\dot\Lambda^{s}_{n/s,\,q}(\mathbb R^n)\) and Triebel--Lizorkin spaces \(\dot F^{s}_{n/s,\,q}(\mathbb R^n)\) conducted in \cite{KoskelaYangZhou2011AdvMath} (see also \cite{KochKoskelaSaksmanSoto2014JFA}), where \(s\in(0,1)\) and \(q\in(0,\infty)\).
In view of these facts, we restrict $\alpha\in(0,\,\min\{1,\,n/2\})$ in Problem \ref{openProb}, as given in \eqref{eq-alpha-range}.
\end{remark}

Next, we recall the seminal work of Koskela, Xiao, Zhang and Zhou \cite{KoskelaXiaoZhangZhou2017JEMS} on \cite[Problem 8.4]{EssenJansonPengXiao2000IUMJ} or Problem \ref{openProb}, which gives a sufficient sharp condition for  boundedness of $\mathcal C_\varphi$ on $\mathcal Q_\alpha(\rn)$. To state their result, we recall that a nonnegative function $w$ on $\mathbb R^n$ is said to belong to the \emph{local Muckenhoupt class $A_1(\mathbb R^n;\, E)$} for some closed set $E\subset\mathbb R^n$ if  for any ball $B=B(x_0,r)\subset\mathbb R^n$ with $d(x_0,E)>2r$,
$$
\frac{1}{|B|} \int_B w(x)\,dx \le C \,\mathop{\operatorname{ess\,inf}}_{B} w,
$$
where $C$ is a positive constant independent of the balls $B$.

\begin{theorem}[\cite{KoskelaXiaoZhangZhou2017JEMS}]\label{thm-JEMS}
Let $n\ge 2$, $\alpha\in(0,1)$ and $\varphi:\ \mathbb R^n \to \mathbb R^n$ be quasisymmetric.
Then $\mathcal C_\varphi$ is bounded on $\mathcal Q_\alpha(\mathbb R^n)$
provided that $J_{\varphi^{-1}}\in A_1(\mathbb R^n;\, E)$ and
\begin{align}\label{eq-con-dimE}
n-2\alpha>
\begin{cases}
\overline{\dim}_L E & \text{for bounded } E ;\\[2pt]
\overline{\dim}_{LG} E & \text{for unbounded } E.
\end{cases}
\end{align}
Moreover, condition \eqref{eq-con-dimE} is sharp in the following sense: there exist quasisymmetric mappings $\varphi$ and closed sets $E$ such that $J_{\varphi^{-1}}\in A_1(\mathbb R^n;\, E)$ and
\begin{align}\label{eq1-con-dimE}
n-2\alpha<
\begin{cases}
\overline{\dim}_L E & \text{for bounded } E ;\\[2pt]
\overline{\dim}_{LG} E & \text{for unbounded  } E ,
\end{cases}
\end{align}
but \(\mathcal C_\varphi\) is not bounded on \(\mathcal Q_\alpha(\mathbb R^n)\).
\end{theorem}

The symbols $\overline{\dim}_L E$ and $\overline{\dim}_{LG} E$ denote, respectively,
the \emph{local self-similar Minkowski dimension} and the \emph{global self-similar Minkowski dimension} of the closed set $E$.
Their precise definitions can be found in \cite{KoskelaXiaoZhangZhou2017JEMS}; we shall not reproduce them here.
Conditions \eqref{eq-con-dimE} and \eqref{eq1-con-dimE} show that the boundedness of $\mathcal C_\varphi$ is governed by the size of the exceptional set $E$ of $J_{\varphi^{-1}}$, as quantified by these dimensions; the exponent $n-2\alpha$ serves as the critical threshold for such exceptional sets.

Still with $n\ge 2$, Xiao and Zhou \cite[Theorems 1.2, 1.4]{XiaoZhou2019ArkMat} (see also \cite{Xiao2019JDE}) further investigated whether boundedness of composition operators can conversely imply the quasisymmetry of a homeomorphism $\varphi$ on $\mathbb R^n$.
 Indeed, for $\alpha\in(0,1/2)$, if both $\mathcal C_\varphi$ and $\mathcal C_{\varphi^{-1}}$ are bounded on $\mathcal Q_\alpha(\rn)$, then $\varphi$ is quasisymmetric, provided that $\varphi$ is ACL (absolutely continuous on almost all lines parallel to the coordinate axes of $\mathbb R^n$) and differentiable a.e. on $\mathbb R^n$. For $\alpha\in(1/2,1)$, boundedness of $\mathcal C_\varphi$ on $\mathcal Q_\alpha(\mathbb R^2)$ alone is sufficient for $\varphi$ to be quasisymmetric.

The complex analogue of Problem \ref{openProb} was first posed by Xiao in his foundational monograph \cite[p.~22]{XiaoQbook2001} and later highlighted by Zhao \cite{Zhao2009Qquestion} as a central open problem in the field. Recently, Hu and Zhou \cite{HuZhou2026Arxiv1} have fully resolved this complex version by characterizing those analytic self-maps $\varphi:\ \mathbb D \to \mathbb D$ for which $\mathcal C_\varphi$ is bounded on the analytic $Q_p(\mathbb D)$ space for all $p\in(0,1)$. Their approach employs a novel dyadic trace formulation over Carleson tents, expressed in terms of generalized $p$-Nevanlinna counting functions.

In view of the classical fact that the behavior of $Q_p(\mathbb D)$ on the unit circle is governed by the $2\pi$-periodic analogue of $\mathcal Q_\alpha(\mathbb R)$ with $\alpha\in(0,1/2)$, the work of Hu and Zhou \cite{HuZhou2026Arxiv1} provides the key insight that enables us to completely settle the original Problem \ref{openProb} for the full range $\alpha\in(0,\min\{1,n/2\})$. The precise statement of our main result will be given in the next subsection.

\subsection{Statement of the main theorem}\label{ss1.3}

Let $\mathcal{D}(\mathbb R^n)$ denote the standard dyadic grid on $\mathbb R^n$:
\begin{align}\label{eq-dyadic-rn}
\mathcal{D}(\rn) := \left\{ Q_{k,m} = 2^{-k}(m + [0,1)^n) :\  k \in \mathbb{Z},\; m \in \mathbb{Z}^n \right\}.
\end{align}
For each $k \in \mathbb{Z}$, the family
$\mathcal{D}_k(\rn) := \{ Q_{k,m} :\  m \in \zz^n \}$  is called the \emph{$k$-th generation of dyadic cubes};
each $Q \in \mathcal{D}_k(\mathbb R^n)$ has side length $\ell(Q) = 2^{-k}$.

For an arbitrary cube $Q \subset \mathbb R^n$ (open, closed, or half-open, with sides parallel to the coordinate axes), we define a relative dyadic system as follows.
Set $\mathcal D_0(Q):=\{Q\}.$
For $k\in\mathbb N$, let $\mathcal D_k(Q)$ be the family of $2^{kn}$ subcubes of $Q$
obtained by $k$ successive bisections of each side of $Q$; each such subcube has side length $2^{-k}\ell(Q)$.
Define
\begin{align*}
\mathcal D(Q):=\bigcup_{k=0}^\infty \mathcal D_k(Q).
\end{align*}
The family $\mathcal D_k(Q)$ is called the \emph{$k$-th generation dyadic subfamily of $Q$}.

For any cube $Q \subset \mathbb R^n$,
the subcubes in $\mathcal D(Q)$ need not belong to the standard dyadic grid $\mathcal D(\mathbb R^n)$.
However, if $Q \in \mathcal D(\mathbb R^n)$, then $\mathcal D(Q) \subset \mathcal D(\mathbb R^n)$.

\begin{definition}\label{def-gauge}
Let $\beta\in \mathbb R$ and
$\varphi:\ \mathbb{R}^n \to \mathbb{R}^n$ be a homeomorphism.
For any cube $Q\subset\rn$, define
\begin{align}\label{eq0-gauge}
\mathscr T_\beta(Q):=
\left\{
\{\lambda_I\}_{I\in\cd(Q)}\subset [0,\infty):\
\sum_{\gfz{I\in \cd(Q)}{I\supset J}} \lambda_I \left(
\frac{|I|}{|J|}\cdot\frac{|\varphi^{-1}(J)|}{|\varphi^{-1}(I)|}\right)^{\beta}
\ge 1 \
\text{for all}\ J\in \cd(Q)\right\}
\end{align}
and
\begin{align}\label{eq1-gauge}
\|\varphi\|_{\mathscr G_\beta(Q)}
:=\inf_{\{\lambda_I\}\in \mathscr T_\beta(Q)}\left\{
\sum_{I\in \cd(Q)} \lambda_I \left(\frac{|I|}{|Q|}\right)^{\beta}
\right\}.
\end{align}
Define the \emph{$\beta$-capacity gauge} of $\varphi$ by
\begin{align}\label{eq2-gauge}
\|\varphi\|_{\mathscr G_\beta}
:=\displaystyle\sup_{\text{cubes}\, Q\subset\rn} \|\varphi\|_{\mathscr G_\beta(Q)}.
\end{align}
\end{definition}

A detailed study of $\|\varphi\|_{\mathscr G_\beta}$ is presented in Section \ref{sec3} below.
In particular, two equivalent characterizations of \(\|\varphi\|_{\mathscr G_\beta}\)
are established in Theorems \ref{thm:upward-duality} and \ref{thm-dyadic-gauge}:
Theorem \ref{thm:upward-duality} gives a dual form representation
of \(\|\varphi\|_{\mathscr G_\beta(Q)}\)
via linear programming duality;
Theorem \ref{thm-dyadic-gauge} shows that the supremum in
\eqref{eq2-gauge} can be restricted to standard
dyadic cubes of the form \(2^{-k}(m + [0,1)^n)\)
with \(k \in \mathbb Z\) and \(m \in \mathbb Z^n\).
In addition,
we provide an explicit \emph{quantitative} upper bound for \(\|\varphi\|_{\mathscr G_\beta(Q)}\)
in terms of the pullback-volume ratios \(|\varphi^{-1}(J)|/|\varphi^{-1}(I)|\)
for dyadic cubes \(J\) and their dyadic ancestors \(I\supset J\); see Theorem \ref{thm-gauge-calc} below.

\begin{remark}\label{rem-gauge}
The terminology ``capacity gauge'' is adopted from \cite{HuZhou2026Arxiv1}. To clarify the geometric meaning of the \(\beta\)-capacity gauge in Definition~\ref{def-gauge}, we make four comments below.

\begin{enumerate}[label=\textup{(\roman*)}]
\item Note that $\mathscr T_\beta(Q)\neq\emptyset$
(for instance, by taking $\lambda_I\equiv 1$).
Since $\varphi$ is a homeomorphism, it maps open sets to open sets and compact sets to compact sets.
Thus, for any cube $I\subset\mathbb R^n$, the set $\varphi^{-1}(I)$ is bounded and has nonempty interior,
which implies $|\varphi^{-1}(I)|\in(0,\infty)$. Hence, in \eqref{eq0-gauge}, the ratio
$$
\frac{|I|}{|J|}\cdot \frac{|\varphi^{-1}(J)|}{|\varphi^{-1}(I)|}
$$
is well-defined and quantifies the scale distortion of $\varphi$ across dyadic levels.
Moreover, the condition
$$
\sum_{\gfz{I\in \cd(Q)}{I\supset J}} \lambda_I \left(
\frac{|I|}{|J|}\cdot\frac{|\varphi^{-1}(J)|}{|\varphi^{-1}(I)|}\right)^{\beta}
\ge 1
\quad \text{for all } J\in\mathcal D(Q)
$$
requires that every dyadic cube \(J\subset Q\) receive enough mass from its dyadic ancestors \(I\supset J\). The coefficients \(\lambda_I\) may therefore be viewed as weights assigned to dyadic cubes, forming a sufficiently strong covering of every smaller dyadic cube \(J\). The quantity \(\|\varphi\|_{\mathscr G_\beta(Q)}\) in \eqref{eq1-gauge} is then the least cost of an admissible weighted dyadic covering of \(Q\), where assigning weight \(\lambda_I\) to a cube \(I\) incurs cost
$$
\lambda_I\left(\frac{|I|}{|Q|}\right)^\beta.
$$
In other words, among all admissible distributions of mass,
\(\|\varphi\|_{\mathscr G_\beta(Q)}\) selects the one with minimal total energy.

\item Now, we give a precise formulation of $\|\varphi\|_{\mathscr G_\beta(Q)}$ as a \emph{tree capacity}. For \(\lambda=\{\lambda_I\}_{I\in\mathcal D(Q)}\) with each \(\lambda_I\ge 0\), define the \emph{weighted energy}
$$
\mathcal E_\beta(\lambda;\,Q):=\sum_{I\in\mathcal D(Q)}
\lambda_I\left(\frac{|I|}{|Q|}\right)^\beta,
$$
and the \emph{distortion potential kernel}
$$
K_\varphi(I,\,J):=
\left(
\frac{|I|}{|J|}\cdot
\frac{|\varphi^{-1}(J)|}{|\varphi^{-1}(I)|}
\right)^\beta
\quad\text{for } \, I,J\in\mathcal D(Q),\ \, I\supset J.
$$
Then
$$
\|\varphi\|_{\mathscr G_\beta(Q)}
=
\inf\left\{\mathcal E_\beta(\lambda;\,Q):\
\lambda_I\ge 0,\
\sum_{\gfz{I\in\mathcal D(Q)}{I\supset J}}\lambda_I K_\varphi(I,\,J)\ge 1
\ \text{for all } J\in\mathcal D(Q)
\right\}.
$$
This is a discrete tree capacity.
When \(\varphi=\mathrm{id}\), we have \(K_{\mathrm{id}}\equiv1\), and the gauge reduces to the purely dyadic model where the potential at \(J\) is the total mass on its ancestors.

\item
The gauge \(\|\varphi\|_{\mathscr G_\beta(Q)}\) can be viewed heuristically
as a generalized weighted dyadic Hausdorff content,
induced by the pull-back geometry of \(\varphi\).
To see the analogy, normalize \(|Q|=1\) and set
$$
\gamma_I := \lambda_I \left( \frac{|I|}{|\varphi^{-1}(I)|} \right)^\beta \ge 0.
$$
Then the constraint in \eqref{eq0-gauge} becomes
$$
\sum_{\gfz{I\in \cd(Q)}{I\supset J}}\gamma_I \ge \left( \frac{|J|}{|\varphi^{-1}(J)|} \right)^\beta
\qquad \text{for all }\, J\in\mathcal D(Q),
$$
and the capacity gauge in \eqref{eq1-gauge} reduces to
$$
\|\varphi\|_{\mathscr G_\beta(Q)}
=
\inf\left\{
\sum_{I\in\mathcal D(Q)}\gamma_I |\varphi^{-1}(I)|^\beta:
\ \,
\sum_{\gfz{I\in \cd(Q)}{I\supset J}}\gamma_I \ge \left( \frac{|J|}{|\varphi^{-1}(J)|} \right)^\beta
\ \; \text{for all } J\in\mathcal D(Q)
\right\}.
$$
This is structurally analogous to the mass-distribution formulation of Hausdorff content.
Indeed, \(\gamma_I\) plays the role of the  mass
assigned to  \(I\),
the constraint requires that
every subcube \(J\) receives at least
\((|J|/|\varphi^{-1}(J)|)^\beta\) mass from its ancestors,
and the cost of assigning mass \(\gamma_I\) to \(I\) is \(\gamma_I |\varphi^{-1}(I)|^\beta\).
\end{enumerate}
\end{remark}

\begin{remark}\label{rem-gauge-eg}
To complement the abstract formulation in Remark \ref{rem-gauge}, we now provide concrete examples
 from Section \ref{sec3} showing when the gauge is finite or infinite.
  \begin{enumerate}[label=\textup{(\roman*)}]

\item  Bi-Lipschitz maps have finite capacity gauge (see Corollary \ref{cor-bilip}).  By
  Proposition \ref{prop-beta=0>1} below, we know that
  $$
  \begin{cases}
   \|\mathrm{id}\|_{\mathscr G_\beta}=1\quad &\text{for all}\ \,\beta\in\rr;\\
    \|\varphi\|_{\mathscr G_\beta}\ge 1 \quad &\text{for all}\ \, \beta\in \rr.
  \end{cases}
  $$

\item Again, by   Proposition \ref{prop-beta=0>1} below, we have
    $$
  \begin{cases}
   \|\varphi\|_{\mathscr G_0}=1\quad &\text{as}\ \, \beta=0;\\
    \|\varphi\|_{\mathscr G_\beta}<\infty \quad &\text{as}\ \, \beta\in (1,\infty).
  \end{cases}
  $$
  For \(\beta\in(-\infty,0)\), the gauge \(\|\varphi\|_{\mathscr G_\beta}\) may be finite or infinite.
  For example, the radial stretching
$$
\varphi_p(x):= x |x|^p\quad \text{for}\ \, x\in\mathbb R^n.
$$
satisfies that (see  Proposition~\ref{prop-gauge-finite-infity} below)
\begin{align*}
\begin{cases}
\|\varphi_p\|_{\mathscr G_\beta}<\infty\quad\
& \text{as}\ \ p\in(-1, 0] \ \,\text{and}\ \,  \beta\in (-\infty, 0);\\
 \|\varphi_p\|_{\mathscr G_\beta}=\infty\quad\
 &\text{as}\ \ p\in(0, \infty) \ \,\text{and}\  \,  \beta\in (-\infty, 0).
\end{cases}
\end{align*}
The range \(\beta\in((1-\frac 2n)_+,\,1)\) is the one relevant to our setting, as it corresponds to
$$
\beta=1-\frac{2\alpha}{n}
\qquad\text{for}\ \,\alpha\in\left(0, \, \min\left\{1,\,\frac n2\right\}\right).
$$

\item
By Proposition \ref{prop-A1toGauge} below,
      if $J_{\varphi^{-1}}\in A_1(\mathbb R^n)$,
      then $\|\varphi\|_{\mathscr G_\beta}<\infty$ for all $\beta\in(0,\infty)$;
      if $J_{\varphi}\in A_1(\mathbb R^n)$,
      then $\|\varphi\|_{\mathscr G_\beta}<\infty$ for all $\beta\in(-\infty,0)$.
  \end{enumerate}
\end{remark}

We now state the main result of this paper, which establishes an intrinsic geometric necessary and sufficient condition for the boundedness of quasisymmetric composition operators $\mathcal C_\varphi$ on the spaces $\mathcal Q_\alpha(\mathbb R^n)$.

\begin{theorem}\label{thm-trace-gauge}
 Let \(\alpha \in (0,\,\min\{1,\,n/2\})\)
 and \(\varphi :\ \mathbb R^n \to \mathbb R^n\)
 be  $\eta$-quasisymmetric,
 with the additional assumption that \(J_\varphi, J_{\varphi^{-1}} \in A_\infty(\mathbb R)\)
 when \(n = 1\). Then \(\mathcal C_\varphi\) is bounded on \(\mathcal Q_\alpha(\mathbb R^n)\)
 if and only if
 $$\|\varphi\|_{\mathscr G_{1-\frac{2\alpha}{n}}} < \infty.$$
 Moreover, there is the equivalence
\begin{align}\label{eq-norm-gauge}
 \|\mathcal C_\varphi\|_{\mathcal Q_\alpha(\mathbb R^n)\to \mathcal Q_\alpha(\mathbb R^n)}^2
 \simeq \|\varphi\|_{\mathscr G_{1-\frac{2\alpha}{n}}},
\end{align}
 where the implicit constants depend only on \(n,\alpha,\eta, [J_{\varphi}]_{A_\infty(\rn)}\) and \([J_{\varphi^{-1}}]_{A_\infty(\rn)}\).
\end{theorem}

Let \(\varphi :\ \mathbb R^n \to \mathbb R^n\) be quasisymmetric.
For \(n\ge 2\), Gehring \cite{Gehring1973Acta} proved that the Jacobian
\(J_\varphi\) satisfies the reverse H\"older inequality (see \eqref{eq-RHn} below).
Since \(\varphi^{-1}\) is also quasisymmetric, the same inequality holds
for \(J_{\varphi^{-1}}\).
Consequently, when \(n\ge 2\), we have \(J_\varphi, J_{\varphi^{-1}} \in A_\infty(\mathbb R^n)\).

For $n=1$, Gehring's lemma is not applicable,
and there is no automatic reverse H\"older inequality for \(J_\varphi\) or
\(J_{\varphi^{-1}}\).
We therefore impose the additional assumption
that \(J_\varphi, J_{\varphi^{-1}} \in A_\infty(\mathbb R)\),
which, by the classical theory of Muckenhoupt weights,
is equivalent to these Jacobians satisfying the reverse H\"older inequality.

Throughout the paper, we shall freely use, without further comment, that
\(J_\varphi, J_{\varphi^{-1}} \in A_\infty(\mathbb R^n)\)
and that both satisfy the reverse H\"older inequality.
The constants appearing in the reverse H\"older inequality for \(J_\varphi\) are determined by
\([J_\varphi]_{A_\infty(\mathbb R^n)}\), while those for \(J_{\varphi^{-1}}\) are determined by
\([J_{\varphi^{-1}}]_{A_\infty(\mathbb R^n)}\).

\begin{remark}\label{rem-mainthm}
Theorem \ref{thm-trace-gauge}  resolves the classical long-standing
Problem \ref{openProb}, posed in \cite{EssenJansonPengXiao2000IUMJ}
(see also \cite{XiaoQbook2001,Zhao2009Qquestion}),
of characterizing the quasisymmetric mappings $\varphi$ for
which the composition operator $f\mapsto f\circ\varphi^{-1}$
preserves the real $Q_\alpha(\mathbb R^n)$ space. It refines
the sharp sufficient criteria of \cite{KoskelaXiaoZhangZhou2017JEMS} ---
formulated in terms of the local and global self-similar Minkowski
dimension of the exceptional sets of the Jacobian --- with the intrinsic
 quantity $\|\varphi\|_{\mathscr G_\beta}$.
This capacity gauge $\|\varphi\|_{\mathscr G_\beta}$ is not merely technical;
it possesses a clear geometric interpretation (see Remarks \ref{rem-gauge}
and \ref{rem-gauge-eg}). Moreover, the  quantitative equivalence in \eqref{eq-norm-gauge}
captures the optimal dependence of the operator norm on the geometry of $\varphi$,
providing sharp control that is essential for applications in harmonic analysis and beyond.
\end{remark}

\begin{remark}
We now describe the main strategy behind the proof of Theorem \ref{thm-trace-gauge}.

\begin{enumerate}[label=\textup{(\roman*)}]
\item
The sufficiency part of Theorem \ref{thm-trace-gauge} is established in Section \ref{sec4},
where we prove that finiteness of the capacity gauge implies boundedness of $\mathcal C_\varphi$.
The argument uses mean oscillation characterizations of \(\mathcal Q_\alpha(\mathbb R^n)\)
from \cite{EssenJansonPengXiao2000IUMJ, KoskelaXiaoZhangZhou2017JEMS}.
The core estimates are carried out in Theorems \ref{thm2-varphi-EJPX} and \ref{thm1-varphi-EJPX},
where we respectively bound the \(\mathcal Q_\alpha(\mathbb R^n)\)
seminorm from above and below
via mean oscillation estimates adapted to \(\varphi\).

\item
The necessity direction of Theorem \ref{thm-trace-gauge} is
established in Section \ref{sec5}, where we show that boundedness of $\mathcal C_\varphi$
forces finiteness of the capacity gauge.
This part relies on the wavelet characterization
of $\mathcal Q_\alpha(\mathbb R^n)$ from \cite{EssenJansonPengXiao2000IUMJ}
and the dyadic reduction of the dual form of the capacity gauge
from Theorems \ref{thm:upward-duality} and \ref{thm-dyadic-gauge}.
The main difficulty and novelty lie in Theorem \ref{thm-upperf},
where wavelets are used to construct an explicit sequence $\{f_N\}_{N\in\nn}$
of functions in
$\mathcal Q_\alpha(\mathbb R^n)$
such that $\|f_N\|_{\mathcal Q_\alpha(\mathbb R^n)}$
is uniformly bounded below by the weighted cost
defining the dual gauge (see \eqref{eq-dual-gauge}),
while $\|f_N\circ\varphi\|_{\mathcal Q_\alpha(\mathbb R^n)}$
is uniformly controlled above via the dual constraint (see \eqref{eq0-dual-gauge}).
This construction is the technical heart of the necessity proof.
\end{enumerate}
\end{remark}

As an immediate application of Theorem \ref{thm-trace-gauge},
we combine it with the earlier result of Koskela, Xiao, Zhang and Zhou \cite{KoskelaXiaoZhangZhou2017JEMS}
(Theorem \ref{thm-JEMS}) to obtain the following dimension criterion
for the capacity gauge.

\begin{corollary}\label{cor-JEMS-gauge}
Let \(n\ge 2\) and \(\alpha\in(0,1)\).
\begin{enumerate}[label=\textup{(\roman*)}]
    \item If \(\varphi:\mathbb R^n \to \mathbb R^n\) is quasisymmetric and \(J_{\varphi^{-1}}\in A_1(\mathbb R^n; E)\) for some closed set \(E\subset\mathbb R^n\) whose local or global self-similar Minkowski dimension satisfies \eqref{eq-con-dimE}, then
$$
\|\varphi\|_{\mathscr G_{1-\frac{2\alpha}{n}}}<\infty.
$$

\item Conversely, there exist a quasisymmetric mapping \(\varphi:\mathbb R^n \to \mathbb R^n\) and a closed set \(E\subset\mathbb R^n\) such that \(J_{\varphi^{-1}}\in A_1(\mathbb R^n; E)\) and the local or global self-similar Minkowski dimension of \(E\) satisfies \eqref{eq1-con-dimE}. For any such \(\varphi\),
$$
\|\varphi\|_{\mathscr G_{1-\frac{2\alpha}{n}}}=\infty.
$$
\end{enumerate}
\end{corollary}

This dimension criterion is one instance of how Theorem \ref{thm-trace-gauge} can be applied.
Further applications are presented in Section \ref{sec6}, where we consider three contexts:
composition stability, removability of sets, and transport equations.
In Subsection \ref{ss6.1}, we prove that finiteness of the capacity gauge
is stable under composition of quasisymmetric mappings.
Subsection \ref{ss6.2} shows that finite capacity gauge
preserves \(\mathcal Q_\alpha\)-removability under quasisymmetric mappings.
Subsection \ref{ss6.3} proves that, for quasisymmetric flows with finite gauge,
the transport equation propagates the \(\mathcal Q_\alpha\)-regularity of initial data,
and the corresponding estimates are uniform in time.

\bigskip

\noindent{\bf Notation.}
Throughout the paper, we adopt the following notation.

\begin{itemize}

\item Let \(\mathbb N:=\{1,2,\dots\}\), \(\mathbb Z_+:=\{0\}\cup\mathbb N\) and \(\mathbb Z:=\{0,\pm1,\pm2,\dots\}\).

\item For any $a,b\in\rr$,  set $a\vee b:=\max\{a,b\}$, $a\wedge b:=\min\{a,b\}$, $a_+:=\max\{a,0\}$ and $a_-:=-\min\{a,0\}$.

\item By a cube \(Q \subset \mathbb R^n\), we always mean one whose sides are parallel to the coordinate axes. Denote by \(\ell(Q)\) its side-length and by \(c_Q\) its center. The symbols \(I,J,R,Q\) are used generically for cubes.

\item \(C_c^\infty(\mathbb R^n)\) denotes the space of infinitely differentiable functions on \(\mathbb R^n\) with compact support.

\item For a set \(E\subset\mathbb R^n\), we write \(|E|\) for its \(n\)-dimensional Lebesgue measure and \(\mathbf 1_E\) for its characteristic function.
     For locally integrable \(f\) on \(\mathbb R^n\), set
$$
\fint_E f(x)\,dx:=\frac1{|E|}\int_E f(x)\,dx.
$$

\item The notation \(U\lesssim V\) or \(V\gtrsim U\) means that \(U\le C V\) for some positive constant \(C\). We occasionally write \(C(\alpha,\beta,\dots)\) to indicate that the constant depends on the parameters \(\alpha,\beta,\dots\). The symbol \(U\simeq V\) means that
    \(U\lesssim V\lesssim U\).
\end{itemize}

\section{Quasisymmetric mappings}\label{sec2}

\subsection{Basic properties}\label{ss2.1}

The following basic properties of quasisymmetric mappings can be found in
 \cite{Heinonen2001book, Koskela2009note, Vaisala1971LNM}.

\begin{proposition}\label{prop-dia-doubl}
Let $\varphi :\ \rn \to \rn$ be $\eta$-quasisymmetric.

\begin{enumerate}[label=\textup{(\roman*)}]
  \item  The inverse $\varphi^{-1}$ of $\varphi$ is $\eta'$-quasisymmetric with
\begin{align}\label{eq-eta'}
\eta'(t):=\frac{1}{\eta^{-1}(t^{-1})}\qquad \text{for all} \ t\in(0,\infty).
\end{align}
   \item  If $E \subset F \subset \rn$ are such that $0 < \operatorname{diam} E \leq \operatorname{diam} F < \infty$, then $\operatorname{diam} \varphi(F)$ is finite and
\begin{align}\label{eq-diamEF}
\frac{1}{2}\left[\eta\left(\frac{\operatorname{diam} F}{\operatorname{diam} E}\right)\right]^{-1}
\leq
\frac{\operatorname{diam} \varphi(E)}{\operatorname{diam} \varphi(F)}
\leq
\eta\left(\frac{2\operatorname{diam} E}{\operatorname{diam} F}\right).
\end{align}
 \item  Given any \(c\in(0,1)\), let \(c_{\eta}:=2^{-1}(\eta')^{-1}(c)\) and \(c_{\eta,n}:=c_{\eta}(2\eta(\sqrt n))^{-1}\).
 Then, for any ball \(B\subset\mathbb R^n\) with center \(c_B\),  we have the two-sided ball inclusions
\begin{align}\label{eq-ball-varphi}
  B\big(\varphi(c_B),\, c_{\eta} \operatorname{diam}\varphi(B)\big)
  \subset \varphi(B)\subset
  B\big(\varphi(c_B),\, \operatorname{diam}\varphi(B)\big),
\end{align}
and, similarly, for any cube \(I\subset\mathbb R^n\) with center \(c_I\),
\begin{align}\label{eq-cube-varphi}
  B\big(\varphi(c_I),\, c_{\eta,n} \operatorname{diam}\varphi(I)\big)
  \subset \varphi(I)\subset
  B\big(\varphi(c_I),\, \operatorname{diam}\varphi(I)\big).
\end{align}
Moreover, it holds  that
\begin{align}\label{eq3-ball-varphi}
|\varphi(B)| \simeq [\operatorname{diam}\varphi(B)]^n
\quad\text{and}\quad
|\varphi(I)| \simeq [\operatorname{diam}\varphi(I)]^n,
\end{align}
with implicit constants independent of \(B\) and \(I\).

\end{enumerate}
\end{proposition}

\begin{proof}
Items (i) and (ii) are from \cite[Proposition~10.6 and 10.8]{Heinonen2001book}, respectively.

Item (iii) is also a well-known property of quasisymmetric mappings. We now provide some details.
The second inclusions in \eqref{eq-ball-varphi} and \eqref{eq-cube-varphi} are obvious.

By (ii), we have $\operatorname{diam}\varphi(B)\in(0,\infty)$ for all balls $B$, and similarly for cubes $I$.
For the first inclusion in \eqref{eq-ball-varphi}, we set
 $$L_\varphi(B):=\sup\left\{|\varphi(z)-\varphi(c_B)|:\,  |z-c_B|\le r\right\},$$
 where $r$ denotes the radius of the ball $B$.
 By the triangle inequality, we have
 $$
 \operatorname{diam}\varphi(B)\le 2 L_\varphi(B).
 $$
Suppose that
$
|u-\varphi(c_B)|< c_\eta \operatorname{diam}\varphi(B),
$
which gives
$$
|u-\varphi(c_B)|< 2c_\eta L_\varphi(B).
$$
By the supremum in defining $L_\varphi(B)$,
there exists some point \(z\) such that $|z-c_B|\le r$ and
$$
|u-\varphi(c_B)|< (\eta')^{-1}(c)|\varphi(y)-\varphi(c_B)|.
$$
Since \(\varphi\) is a homeomorphism on \(\rn\), there is some \(y\in\rn\) with \(u=\varphi(y)\). By the increasing property of \(\eta\) (and hence of \(\eta'\)), we obtain
$$
\frac{|y-c_B|}{|z-c_B|}
\le \eta'\left(\frac{|u-\varphi(c_B)|}{|\varphi(z)-\varphi(c_B)|}\right)
\le \eta'\left((\eta')^{-1}(c)\right)=c.
$$
Thus \(|y-c_B|\le c|z-c_B|\le cr<r\), which shows that $y\in B$
and, hence, \(u=\varphi(y)\in\varphi(B)\). This proves the first inclusion in \eqref{eq-ball-varphi}, as desired.

For the first inclusion in \eqref{eq-cube-varphi}, consider the ball \(B_I:=B(c_I,\, \ell(I)/2)\). Note that \(B_I\subset I\). From the first inequality in \eqref{eq-diamEF}, it follows that
$$
\frac{\operatorname{diam}\varphi(B_I)}{\operatorname{diam}\varphi(I)}
\ge \frac{1}{2\eta(\sqrt n)}.
$$
Combining this with the first inclusion in \eqref{eq-ball-varphi} yields
$$
\varphi(I)\supset \varphi(B_I)
\supset B\bigl(\varphi(c_I),\, c_{\eta}\operatorname{diam}\varphi(B_I)\bigr)
\supset
B\left(\varphi(c_I),\, \frac{c_{\eta}}{2\eta(\sqrt n)}\operatorname{diam}\varphi(I)\right).
$$
Thus the first inclusion in \eqref{eq-cube-varphi} follows with \(c_{\eta,n}=c_{\eta}(2\eta(\sqrt n))^{-1}\).

Finally, the two equivalences in \eqref{eq3-ball-varphi} follow directly from \eqref{eq-ball-varphi} and \eqref{eq-cube-varphi}.
\end{proof}

\subsection{Muckenhoupt weights}\label{ss2.2}

By a \emph{weight} $w$ on $\mathbb R^n$, we mean a nonnegative locally integrable function.
We say that a weight $w$ belongs to the \emph{Muckenhoupt class $A_p(\mathbb R^n)$} with $1<p<\infty$ if$$
[w]_{A_p(\mathbb{R}^n)} = \sup_{\text{cubes } Q \subset \mathbb{R}^n} \left( \frac{1}{|Q|} \int_Q w(x) \, dx \right) \left( \frac{1}{|Q|} \int_Q (w(x))^{\frac{1}{1-p}} \, dx \right)^{p-1} < \infty,
$$
and that \(w \in A_1(\mathbb{R}^n)\) if
$$
[w]_{A_1(\mathbb{R}^n)} = \sup_{\text{cubes } Q \subset \mathbb{R}^n} \left( \frac{1}{|Q|} \int_Q w(x) \, dx \right) \left( \mathop{\mathrm{ess\,inf}\,}_{Q} w \right)^{-1} < \infty.
$$
Define
$$
A_\infty(\mathbb{R}^n) := \bigcup_{1 \le p < \infty} A_p(\mathbb{R}^n).
$$
According to Grafakos \cite[Section~7]{Grafakos-CFAbook}, \(w \in A_\infty(\mathbb{R}^n)\) if and only if
$$
[w]_{A_\infty(\mathbb{R}^n)}
:= \sup_{\text{cubes } Q \subset \mathbb{R}^n} \left( \frac{1}{|Q|} \int_Q w(x) \, dx \right)
 \exp \left( -\frac{1}{|Q|} \int_Q \log w(x) \, dx \right) < \infty,
$$
and for any $p\in[1,\infty)$ we have the comparison
$$[w]_{A_\infty(\rn)}\le [w]_{A_p(\rn)}.$$

Given a weight $w$ on $\mathbb R^n$ and a measurable set $E\subset\mathbb R^n$, we denote $w(E):=\int_E w(x)\,dx$.
Let $w\in A_p(\mathbb R^n)$ for some $p\in[1,\infty)$. Then, by  \cite[(7.2.1)]{Grafakos-CFAbook},
for any cube $Q\subset\mathbb{R}^n$ and any measurable subset $S\subset Q$,
\begin{align}\label{eq-Apweight}
 \frac{w(S)}{w(Q)}\ge \frac1{[w]_{A_p(\rn)}}\left(\frac{|S|}{|Q|}\right)^p.
\end{align}
In particular,  for any cube \(Q\subset\mathbb{R}^n\) and $\lambda\in(1,\infty)$,
\begin{align}\label{eq-vd}
w(\lambda Q)\le [w]_{A_p(\rn)}\lambda^{np} w(Q).
\end{align}
In other words, the measure $w(x)\,dx$ satisfies the \emph{volume doubling } property.

There are two well-known equivalent characterizations of Muckenhoupt $A_\infty(\rn)$ weights
(see, for example, Grafakos
\cite[Theorem~7.3.3]{Grafakos-CFAbook}). On the one hand,
\(w\in A_\infty(\rn)\) if and only if there exist constants \(C,\delta>0\) such that for every cube \(Q\) and measurable \(S\subset Q\),
\begin{align}\label{eq-Afzweight}
    \frac{w(S)}{w(Q)} \le C \left(\frac{|S|}{|Q|}\right)^\delta.
\end{align}
 On the other hand,  \(w \in A_\infty(\mathbb{R}^n)\) if and only if it satisfies the
 the following \emph{reverse H\"older inequality}:
 there exist  constants \(C\in(0,\infty)\) and $r_w\in(1,\infty)$ such that for every cube \(Q \subset \mathbb R^n\),
\begin{align}\label{eq-RH}
\left(\fint_Q w(x)^{r_w} \, dx\right)^{\frac{1}{r_w}}
\le C \fint_Q w(x) \, dx.
\end{align}
The number $r_w$ is called the \emph{reverse H\"older exponent}. Define
 $$[w]_{\mathrm{RH}_{r_w}}=\inf\left\{C:\
 C\ \text{ satisfies }\ \eqref{eq-RH}\right\}.$$
The value $[w]_{\mathrm{RH}_{r_w}}$ depend on $[w]_{A_\infty(\rn)}$ and vice versa.

\begin{proposition}\label{prop-varphi-weight}
Let \(\varphi :\ \mathbb R^n \to \mathbb R^n\) be an \(\eta\)-quasisymmetric mapping.

\begin{enumerate}[label=\textup{(\roman*)}]
  \item  For \(n \ge 2\), the Jacobian \(J_\varphi\in L_\loc^1(\rn)\) and satisfies the reverse H\"older inequality:
  there exist constants \(r_\varphi \in (1,\infty)\) and \(C \in (1,\infty)\), depending only on \(\eta\) and \(n\), such that for every cube \(Q \subset \mathbb R^n\),
  \begin{align}\label{eq-RHn}
  \left(\fint_Q J_\varphi(x)^{r_\varphi} \, dx\right)^{\frac{1}{r_\varphi}}
  \le C \fint_Q J_\varphi(x) \, dx.
  \end{align}
  In particular, \(J_\varphi \in A_\infty(\mathbb R^n)\) whenever \(n \ge 2\).

  \item For \(n = 1\), the reverse H\"older inequality \eqref{eq-RHn} remains valid under
  the additional assumption \(J_\varphi  \in A_\infty(\mathbb R)\).

  \item  If \(J_\varphi \in A_\infty(\mathbb{R}^n)\), then for every cube \(Q \subset \mathbb{R}^n\)
  and measurable subset \(S \subset Q\),
\begin{align}\label{eq-varphi-weight}
C^{-1}\left(\frac{|S|}{|Q|}\right)^{\beta_1}\le \frac{|\varphi(S)|}{|\varphi(Q)|}
\le C\left(\frac{|S|}{|Q|}\right)^{\beta_2},
\end{align}
where \(C \in (1,\infty)\), \(\beta_1 \in [1,\infty)\), \(\beta_2 \in (0,1]\),
  and these constants depend only on \(n\) and   \([J_\varphi]_{A_\infty(\mathbb{R}^n)}\).

  \item If \(J_\varphi \in A_\infty(\mathbb{R}^n)\), then there exists a constant
  \(C = C(n, \, [J_\varphi]_{A_\infty(\mathbb{R}^n)})\) such that
for any cube $Q\subset\mathbb{R}^n$,
$$
|\varphi(2Q)|\le C|\varphi(Q)|.
$$

\item Items (i) through (iv) also hold for the inverse mapping \(\varphi^{-1}\),
  with the analogous additional assumption that \(J_{\varphi^{-1}} \in A_\infty(\mathbb R)\)
  whenever $n=1$.
\end{enumerate}
\end{proposition}

The reverse H\"older inequality in Proposition \ref{prop-varphi-weight}(i)
for quasisymmetric mappings in dimension $n\ge 2$ is a classical result of Gehring \cite{Gehring1973Acta}.
The remaining properties in Proposition \ref{prop-varphi-weight}
follow from the standard theory of $A_\infty(\mathbb R^n)$ weights.

\section{Capacity gauge}\label{sec3}

\subsection{Dual form of the $\beta$-capacity gauge}\label{ss3.1}

In this subsection, we establish a dual representation for the $\beta$-capacity gauge
defined in Definition~\ref{def-gauge}.
Our proof follows the standard finite-dimensional linear programming duality framework,
combined with a limiting argument
over finite dyadic truncations, as used in \cite[Lemma~2.1]{HuZhou2026Arxiv1}.

\begin{theorem}
\label{thm:upward-duality}
Let $\beta\in\rr$ and $\varphi:\ \mathbb{R}^n \to \mathbb{R}^n$ be a homeomorphism. Then,
\begin{align}\label{eq-dual-gauge}
\|\varphi\|_{{\mathscr G}_\beta(Q)}
=\sup_{\{\mu_J\}\in \mathscr T_\beta^*(Q)}\left\{
\sum_{J\in \cd(Q) } \mu_J \left(\frac{|J|}{|Q|}\right)^{\beta }
\right\},
\end{align}
where the dual test sequence space $\mathscr T_\beta^*(Q)$ is  given by
\begin{align}\label{eq0-dual-gauge}
\mathscr T_\beta^*(Q):=
\left\{
\{\mu_J\}_{J\in\cd(Q)}\subset [0,\infty):\
\sum_{\gfz{J\in \cd(Q)}{J\subset I}} \mu_J \left(
\frac{|\varphi^{-1}(J)|}{|\varphi^{-1}(I)|}\right)^{\beta }
\le 1 \ \,
\text{for all}\ I\in \cd(Q)
\right\}.
\end{align}
\end{theorem}

\begin{proof}
To simplify the notation, for any $I,J\in\cd(Q)$, we define
$$
b_I:=\left(\frac{|I|}{|Q|}\right)^{\beta }
$$
and the weight matrix
$$
A_{I,J}:=\begin{cases}
\displaystyle\left(\frac{|I|}{|J|}\cdot \frac{|\varphi^{-1}(J)|}{|\varphi^{-1}(I)|}\right)^{\beta }
&\quad\text{if}\ \, J\subset I\subset Q;\vspace{0.1cm}\\
0 &\quad\text{otherwise}.
\end{cases}
$$
For $N\in\nn$, consider the finite dyadic tree
$$
\cd_{\le N}(Q):=\left\{I\subset Q:\ I\in\cd_k(Q)\ \, \text{for some}\ \, k=0,1,\dots, N\right\}.
$$
As $N\to\infty$, the finite trees $\mathcal D_{\le N}(Q)$ increase to the full tree $\mathcal D(Q)$.
 We split the arguments into four steps.

\medskip

{\bf Step 1:\, Applying finite-dimensional linear programming duality on a finite dyadic tree.}
For any $N\in\nn$, we define the truncated $\beta$-capacity gauge on $Q$ by
\begin{align}\label{eq-gaugeN}
\|\varphi\|_{\mathscr G_\beta^{\le N}(Q)}
&:=\inf\Bigg\{
\sum_{I\in \cd_{\le N}(Q)} \lambda_I b_I:\ \  \{\lambda_I\}_{I\in\cd(Q)}\subset [0,\infty)\ \text{such that} \\
&\hspace{3cm}\ \sum_{\gfz{I\in \cd_{\le N}(Q)}{I\supset J}} \lambda_I A_{I,J}
\ge 1 \
\text{for all}\ J\in \cd_{\le N}(Q)
\Bigg\}.\notag
\end{align}
Set $\lambda:=\{\lambda_I\}_{I\in \cd_{\le N}(Q)}$, which is a finite dimensional vector.
Consider the following finite-dimensional linear programming primal problem:
\begin{align}
\begin{cases}
&\text{minimize}\ {\mathcal J}_N\left( \lambda\right):=
\displaystyle\sum_{I\in \cd_{\le N}(Q)} \lambda_I b_I \vspace{0.1cm} \\
&\text{subject to}\
\eta_J(\lambda):= 1-\displaystyle\sum_{\gfz{I\in \cd_{\le N}(Q)}{I\supset J}} \lambda_I A_{I,\,J} \le 0\ \ \text{for all}\ J\in\cd_{\le N}(Q)
\end{cases}. \tag{$P_N$}
\end{align}
By the definition of $\|\varphi\|_{\mathscr G_\beta^{\le N}(Q)}$, the optimal value of $(P_N)$, denoted by $\mathrm{val}(P_N)$, is precisely
\begin{align}\label{eq-vanPN}
\mathrm{val}(P_N)=\|\varphi\|_{\mathscr G_\beta^{\le N}(Q)}.
\end{align}

Now, we introduce Lagrange multipliers $\mu_J\ge 0$ for the constraints indexed by $J\in\mathcal D_{\le N}(Q)$.
The associated Lagrangian function is given by
$$
L(\lambda,\mu)= {\mathcal J}_N\left( \lambda\right) + \sum_{J\in\mathcal D_{\le N}(Q)} \mu_J \eta_J(\lambda).
$$
A direct rearrangement gives
$$
\begin{aligned}
L(\lambda,\mu)
&=\sum_{I\in\mathcal D_{\le N}(Q)}\lambda_I  b_I
+\sum_{J\in\mathcal D_{\le N}(Q)}\mu_J\left(1-\sum_{\gfz{I\in\mathcal D_{\le N}(Q)}{I\supset J}}A_{I,\,J}\lambda_I\right)\\
&=\sum_{J\in\mathcal D_{\le N}(Q)}\mu_J
+\sum_{I\in\mathcal D_{\le N}(Q)}\lambda_I\left(b_I-
\sum_{\gfz{J\in\mathcal D_{\le N}(Q)}{J\subset I}}A_{I,\,J}\mu_J\right).
\end{aligned}
$$
For a fixed nonnegative sequence $\mu$, by taking the infimum over all $\lambda_I\ge 0$, we observe that $\inf_{\lambda\ge0}L(\lambda,\mu)$ is finite if and only if
$$
\sum_{\gfz{J\subset I}{J\in\mathcal D_{\le N}(Q)}}A_{I,\,J}\mu_J \le b_I\qquad \text{for all}\ \, I\in\mathcal D_{\le N}(Q).
$$
Under this latter condition, we have
$$
\inf_{\lambda\ge 0} L(\lambda,\mu)=\sum_{J\in\mathcal D_{\le N}(Q)} \mu_J.
$$
Thus, the dual problem of $(P_N)$ is given by
\begin{align}
\begin{cases}
&\text{maximize}\ {\mathcal K}_N\left( \mu\right):= \displaystyle\sum_{J\in \cd_{\le N}(Q)} \mu_J \vspace{0.1cm} \\
&\text{subject to}\
\zeta_I(\mu):= b_I-\displaystyle\sum_{\gfz{J\in \cd_{\le N}(Q)}{J\subset I}} A_{I,\,J} \mu_J \ge 0\ \ \text{for all}\ I\in\cd_{\le N}(Q)
\end{cases}. \tag{$D_N$}
\end{align}

Note that $(P_N)$ is a finite-dimensional linear program.
It is feasible. Indeed, taking $\lambda_I=1$ for all $I\in\cd(Q)$ gives that, for every $J\in\cd(Q)$,
$$
\sum_{I\supset J} \lambda_I A_{I,J}
\ge A_{J,J}
=
\left(
\frac{|J|}{|J|}\cdot
\frac{|\varphi^{-1}(J)|}{|\varphi^{-1}(J)|}
\right)^{\beta }
=1,
$$
so all constraints in  $(P_N)$ are satisfied. Moreover, the objective is bounded below by $0$.

Therefore, by the finite-dimensional strong duality theorem for linear programming
(see, e.g., \cite[Theorem~5.2]{Vanderbei2020book}), the optimal value of $(P_N)$ equals the optimal value of its dual $(D_N)$; that is,
\begin{align}\label{eq-vanPNDN}
\mathrm{val}(P_N)=\mathrm{val}(D_N).
\end{align}

\medskip

{\bf Step 2:\, Passing the primal problem $(P_N)$ to infinite dyadic tree.}
The aim of this step is to prove
\begin{align}\label{eq-limPN}
 \lim_{N\to\infty}\mathrm{val}(P_N)=\|\varphi\|_{\mathscr G_\beta(Q)}.
\end{align}

For every \(J\in\cd_{\le N}(Q)\), if \(I\in\cd(Q)\) satisfies \(I\supset J\),
then necessarily \(I\in \cd_{\le N}(Q)\), which implies
$$
\sum_{\gfz{I\in \cd_{\le N}(Q)}{I\supset J}} \lambda_I A_{I,J}
=
\sum_{\gfz{I\in \cd(Q)}{I\supset J}} \lambda_I A_{I,J}.
$$
Consequently, any feasible sequence \(\lambda=\{\lambda_I\}_{I\in\cd(Q)}\subset [0,\infty)\) for \((P_{N+1})\) is also feasible for \((P_N)\). Furthermore, comparing the objective functions in \((P_N)\) and \((P_{N+1})\), we get
$$
\sum_{I\in \cd_{\le N}(Q)} \lambda_I b_I
\le \sum_{I\in \cd_{\le N+1}(Q)} \lambda_I b_I.
$$
Thus, we obtain
$$
\mathrm{val}(P_N)\le \mathrm{val}(P_{N+1}).
$$
In particular, the sequence \(\{\mathrm{val}(P_N)\}_{N\in\nn}\) is nondecreasing and, hence,
$$
\lim_{N\to\infty}\mathrm{val}(P_N)=\sup_{N\in\nn}\mathrm{val}(P_N).
$$

In a similar manner, the preceding argument also yields
\begin{align*}
\|\varphi\|_{\mathscr G_\beta^{\le N}(Q)}
&=\inf\Bigg\{
\sum_{I\in \cd_{\le N}(Q)} \lambda_I b_I:\ \  \{\lambda_I\}_{I\in\cd(Q)}\subset [0,\infty)\ \text{such that} \\
&\hspace{3cm}\ \sum_{\gfz{I\in \cd(Q)}{I\supset J}} \lambda_I A_{I,J}
\ge 1 \
\text{for all}\ J\in \cd_{\le N}(Q)
\Bigg\}\\
&\le \inf\Bigg\{
\sum_{I\in \cd(Q)} \lambda_I b_I:\ \  \{\lambda_I\}_{I\in\cd(Q)}\subset [0,\infty)\ \text{such that} \\
&\hspace{3cm}\ \sum_{\gfz{I\in \cd(Q)}{I\supset J}} \lambda_I A_{I,J}
\ge 1 \
\text{for all}\ J\in \cd(Q)
\Bigg\}\\
&=\|\varphi\|_{\mathscr G_\beta(Q)}.
\end{align*}
Invoking \eqref{eq-vanPN}, we therefore obtain
\begin{align}\label{eq-PN<=}
\lim_{N\to\infty}\mathrm{val}(P_N)=\lim_{N\to\infty}\|\varphi\|_{\mathscr G_\beta^{\le N}(Q)}
\le \|\varphi\|_{\mathscr G_\beta(Q)}.
\end{align}

If \(\lim_{N\to\infty}\mathrm{val}(P_N)=\infty\), then \eqref{eq-PN<=} forces \(\|\varphi\|_{\mathscr G_\beta(Q)}=\infty\), so \eqref{eq-limPN} holds automatically. Assume henceforth that
$$
L:=\lim_{N\to\infty}\mathrm{val}(P_N)=\sup_{N\in\nn}\mathrm{val}(P_N)<\infty.
$$
For each \(N\in\nn\), choose a feasible sequence
$
\lambda^{(N)}
=
\{\lambda_I^{(N)}\}_{I\in\cd_{\le N}(Q)}
$
for the primal problem \((P_N)\) such that
\begin{align}\label{eq-lambdaI}
\sum_{I\in \cd_{\le N}(Q)} \lambda_I^{(N)} b_I
\le \mathrm{val}(P_N)+\frac1N.
\end{align}
Extend \(\lambda^{(N)}\) to the entire dyadic tree by declaring
$$
\lambda_I^{(N)}:=0
\qquad
\text{for }\ I\notin\cd_{\le N}(Q).
$$
Now, fix \(I\in\mathcal D(Q)\). For every \(N\in\nn\), since \(\mathrm{val}(P_N)+\frac1N\le L+1\), it follows from \eqref{eq-lambdaI} that
$$
0\le \lambda_I^{(N)}
\le \frac{L+1}{b_I}.
$$
As the dyadic tree is countable, a standard diagonal argument allows us to extract a subsequence, denoted by \(\{\lambda^{(N_j)}\}_{j\in\nn}\), and obtain a nonnegative sequence \(\lambda^*=\{\lambda_I^*\}_{I\in\mathcal D(Q)}\) such that
\begin{align}\label{eq-lambdaI-lim}
\lim_{j\to\infty}\lambda_I^{(N_j)}=\lambda_I^*
\qquad
\text{for every fixed }I\in\mathcal D(Q).
\end{align}

Fix \(J\in\mathcal D(Q)\). Without loss of generality,
assume that \(J\in \mathcal D_{\le N_{j_0}}(Q)\) for some fixed \(j_0\in\nn\).
For any \(j\ge j_0\), since \(\lambda^{(N_j)}\) is feasible for \((P_{N_j})\)
and \(J\in
 \mathcal D_{\le N_j}(Q)\), we have
$$
\sum_{\gfz{I\in \cd(Q)}{I\supset J}} \lambda_I^{(N_j)} A_{I,J}
=\sum_{\gfz{I\in \cd_{\le N_j}(Q)}{I\supset J}} \lambda_I^{(N_j)} A_{I,J}
\ge 1.
$$
Note that the set of dyadic ancestors
$
\{I\in\mathcal D(Q):\ I\supset J\}
$
is finite. Hence, by \eqref{eq-lambdaI-lim},
$$
\sum_{\gfz{I\in \cd(Q)}{I\supset J}}
\lambda_I^* A_{I,J}
=\lim_{j\to\infty} \sum_{\gfz{I\in \cd(Q)}{I\supset J}}
\lambda_I^{(N_j)} A_{I,J}
\ge1.
$$
In view of the definition of \(A_{I,J}\), this implies that
\(\lambda^*=\{\lambda_I^*\}_{I\in\mathcal D(Q)}\) belongs to \(\mathscr T_\beta(Q)\)
as defined in \eqref{eq0-gauge}.
Moreover, applying \eqref{eq-lambdaI-lim},
the Fatou lemma for nonnegative series, and \eqref{eq-lambdaI}, we obtain
$$
\sum_{I\in\mathcal D(Q)}b_I\lambda_I^*
\le
\liminf_{j\to\infty}
\sum_{I\in \mathcal D(Q)}
b_I\lambda_I^{(N_j)}
\le
\liminf_{j\to\infty} \left(\mathrm{val}(P_{N_j})+\frac1{N_j}\right)\le L.
$$
This, combined with \eqref{eq1-gauge} and the definition of \(L\), gives
\begin{align}\label{eq-PN>=}
\|\varphi\|_{\mathscr G_\beta(Q)}
\le L =\lim_{N\to\infty}\mathrm{val}(P_N).
\end{align}
Combining \eqref{eq-PN<=} and \eqref{eq-PN>=} yields \eqref{eq-limPN}.

\medskip

{\bf Step 3:\, Passing the dual problem $(D_N)$ to infinite dyadic tree.}
Now, we show
\begin{align}\label{eq-limDN}
 \lim_{N\to\infty}\mathrm{val}(D_N)
 &=\sup\Bigg\{
\sum_{J\in\mathcal D(Q)}\mu_J:\ \  \{\mu_J\}_{J\in\cd(Q)}\subset [0,\infty)\ \,
 \text{such that}\ \\
&\hspace{1.5cm}
\sum_{\gfz{J\in \cd(Q)}{J\subset I}}A_{I, J} \mu_J
\le b_I \
\text{for all}\ I\in \cd(Q)
\Bigg\}.\notag
\end{align}

From $(D_N)$ we have
\begin{align}\label{eq-gaudeN-dual}
\mathrm{val}(D_N)
&=\sup\Bigg\{
\sum_{J\in\mathcal D_{\le N}(Q)}\mu_J:\ \  \{\mu_J\}_{J\in\cd_{\le N}(Q)}\subset [0,\infty)\ \text{such that}\ \\
&\hspace{2.5cm}
\sum_{\gfz{J\in \cd_{\le N}(Q)}{J\subset I}}A_{I,J} \mu_J
\le b_I \
\text{for all}\ I\in \cd_{\le N}(Q)
\Bigg\}.\notag
\end{align}
Suppose that \(\{\mu_J\}_{J\in\cd_{\le N}(Q)}\) satisfies
the constraint condition in the set on the right-hand side of \eqref{eq-gaudeN-dual}.
Extend this sequence to the full dyadic tree by setting \(\mu_J=0\) for \(J\notin\cd_{\le N}(Q)\).
For this extended sequence, we have
\begin{align}\label{eq-ext}
\sum_{\gfz{J\in \cd(Q)}{J\subset I}}A_{I,J} \mu_J
\le b_I \quad
\text{for all } I\in \cd(Q).
\end{align}
Indeed, if \(I\notin \cd_{\le N}(Q)\), then every \(J\in\cd(Q)\) with \(J\subset I\) also satisfies \(J\notin \cd_{\le N}(Q)\); hence \(\mu_J=0\) for all such \(J\), and therefore
$$
\sum_{\gfz{J\in \cd(Q)}{J\subset I}}A_{I,J} \mu_J=0\le b_I,
$$
which verifies \eqref{eq-ext}. Consequently, it follows from \eqref{eq-ext} and \eqref{eq-gaudeN-dual} that
\begin{align}\label{eq-gaudeN-dual2}
\mathrm{val}(D_N)
&\le \sup\Bigg\{
\sum_{J\in\mathcal D(Q)}\mu_J:\ \ \{\mu_J\}_{J\in\cd(Q)}\subset [0,\infty)\ \,
 \text{such that}\ \\
&\hspace{2.5cm}
\sum_{\gfz{J\in \cd(Q)}{J\subset I}}A_{I,J} \mu_J
\le b_I \
\text{for all}\ I\in \cd(Q)
\Bigg\}.\notag
\end{align}
Letting \(N\to\infty\) yields the \(\le\) inequality in \eqref{eq-limDN}.

Conversely, suppose that \(\{\mu_J\}_{J\in\cd(Q)}\) satisfies the constraint condition in the set on the right-hand side of \eqref{eq-gaudeN-dual2}. Define its truncation \(\mu^{(N)}\) by
$$
\mu_J^{(N)} :=
\begin{cases}
\mu_J & \text{if } J\in\cd_{\le N}(Q);\\
0 & \text{otherwise}.
\end{cases}
$$
Clearly, for any \(I\in \cd_{\le N}(Q)\),
$$
\sum_{\gfz{J\in \cd_{\le N}(Q)}{J\subset I}}A_{I,J} \mu_J^{(N)}
\le \sum_{\gfz{J\in \cd(Q)}{J\subset I}}A_{I,J} \mu_J
\le b_I.
$$
Hence \(\{\mu_J^{(N)}\}_{J\in\cd_{\le N}(Q)}\) satisfies the constraint condition in the set on the right-hand side of \eqref{eq-gaudeN-dual}. By this, together with \eqref{eq-gaudeN-dual} and the monotone convergence theorem,
$$
\sum_{J\in\mathcal D(Q)}\mu_J
=\lim_{N\to\infty}\sum_{J\in\mathcal D(Q)}\mu_J^{(N)}
\le \lim_{N\to\infty}\mathrm{val}(D_N).
$$
Taking the supremum over all such \(\{\mu_J\}_{J\in\cd(Q)}\),
 we obtain
\begin{align}\label{eq-gaudeN-dual3}
& \sup\Bigg\{
\sum_{J\in\mathcal D(Q)}\mu_J:\ \ \{\mu_J\}_{J\in\cd(Q)}\subset [0,\infty)\ \,
 \text{such that}\ \\
&\hspace{2.5cm}
\sum_{\gfz{J\in \cd(Q)}{J\subset I}}A_{I,J} \mu_J
\le b_I \
\text{for all}\ I\in \cd(Q)
\Bigg\}\le \lim_{N\to\infty}\mathrm{val}(D_N).\notag
\end{align}
Combining \eqref{eq-gaudeN-dual2} and \eqref{eq-gaudeN-dual3} gives the desired identity \eqref{eq-limDN}.

\medskip

{\bf Step 4: Passing to the limit as \(N\to\infty\) in both sides of \eqref{eq-vanPNDN}.}
Having established \eqref{eq-limPN} and \eqref{eq-limDN}, we pass to the limit in \eqref{eq-vanPNDN}
and substitute the expressions for \(A_{I,J}\) and \(b_I\) to obtain
\begin{align*}
\|\varphi\|_{\mathscr G_\beta(Q)}
 &=\sup\Bigg\{
\sum_{J\in\mathcal D(Q)}\mu_J:\ \  \{\mu_J\}_{J\in\cd(Q)}\subset [0,\infty)\ \,
 \text{such that}\ \\
&\hspace{1.5cm}   \sum_{\gfz{J\in \cd(Q)}{J\subset I}} \mu_J \left(\frac{|I|}{|J|}\cdot \frac{|\varphi^{-1}(J)|}{|\varphi^{-1}(I)|}\right)^{\beta }
\le \left(\frac{|I|}{|Q|}\right)^{\beta } \
\text{for all}\ I\in \cd(Q)
\Bigg\}.\notag
\end{align*}
Now, making the change of variables
$$
{\mu}_J := \widetilde\mu_J \left(\frac{|J|}{|Q|}\right)^{\beta },
$$
and substituting this into the preceding equality yields
$$
\|\varphi\|_{\mathscr G_\beta(Q)}
=
\sup_{\{\widetilde{\mu}_J\}_J\in \mathscr T_\beta^*(Q)}
\sum_{J\subset Q} \widetilde{\mu}_J \left(\frac{|J|}{|Q|}\right)^{\beta },
$$
where \(\mathscr T_\beta^*(Q)\) is defined in \eqref{eq0-dual-gauge}.
This is precisely the desired duality formula \eqref{eq-dual-gauge}.
\end{proof}

\subsection{A dyadic reduction characterization of the $\beta$-capacity gauge}\label{ss3.2}

The main aim of this subsection is to establish the following identity for the capacity gauge, which considers the supremum of the gauge $\|\varphi\|_{{\mathscr G}_\beta(Q)}$ over the family of standard dyadic cubes $Q \in \mathcal D(\rn)$.

\begin{theorem}
\label{thm-dyadic-gauge}
Let $\beta\in\rr$ and $\varphi:\ \mathbb{R}^n \to \mathbb{R}^n$ be a homeomorphism
such that $J_{\varphi^{-1}}\in A_\infty(\rn)$.
Then, there exists a constant $C=C(n, \beta,\, [J_{\varphi^{-1}}]_{A_\infty(\rn)})$ such that
\begin{align}\label{eq-dyadic-gauge}
\sup_{Q\in\cd(\rn)} \|\varphi\|_{{\mathscr G}_\beta(Q)}\le \|\varphi\|_{{\mathscr G}_\beta}
\le C \sup_{Q\in\cd(\rn)} \|\varphi\|_{{\mathscr G}_\beta(Q)}.
\end{align}
\end{theorem}

To prove Theorem~\ref{thm-dyadic-gauge}, we introduce the notion of a \( (\beta,\,\varphi)\)-packing condition on \(Q\), which is exactly the constraint defining the test sequence space \(\mathscr T_\beta^*(Q)\) in \eqref{eq0-dual-gauge}.

\begin{definition}\label{def-packing}
Let $Q\subset\mathbb R^n$ be a cube. A nonnegative sequence $\{\mu_J\}_{J\in\mathcal D(Q)}$ is said to satisfy the \emph{$(\beta,\,\varphi)$-packing condition on $Q$} if
\begin{align}\label{eq-packing}
\sup_{R\in\mathcal D(Q)}\sum_{\gfz{J\in\cd(Q)}{J\subset R} } \mu_J\left(\frac{|\varphi^{-1}(J)|}{|\varphi^{-1}(R)|}\right)^{\beta} \le 1.
\end{align}
\end{definition}

\begin{remark}\label{rem-packing}
If \(\{\mu_J\}_{J\in\mathcal D(Q)}\) satisfies the \((\beta,\varphi)\)-packing condition on \(Q\),
then taking only the term \(J=R\) in the sum of \eqref{eq-packing} yields
$$
0\le \mu_R\le 1\quad\,\text{for all }\, R\in\mathcal D(Q).
$$
\end{remark}

The following lemma is the main ingredient in the proof of Theorem~\ref{thm-dyadic-gauge}.

\begin{lemma}\label{lem-dyadic-packing}
Let $\beta\in\mathbb R$ and  $\varphi:\ \mathbb{R}^n \to \mathbb{R}^n$ be a homeomorphism
such that $J_{\varphi^{-1}}\in A_\infty(\rn)$.
For any cube $Q\subset\mathbb R^n$, suppose that
$\{\mu_J\}_{J\in\mathcal D(Q)}$ is a nonnegative sequence satisfying the $(\beta,\,\varphi)$-packing condition on $Q$.
Then, the following statements hold:
\begin{enumerate}[label=\textup{(\roman*)}]
  \item There exist standard dyadic cubes $\{Q_1,\dots, Q_N\}\subset\cd(\rn)$ of the same side length, with $N\le 2^n$, such that for each $i\in\{1,2,\dots, N\}$,
      $$
      \ell(Q)< \ell(Q_i)\le 2\ell(Q),\qquad Q_i\cap Q\neq\emptyset,\qquad  Q\subset \bigcup_{i=1}^N Q_i.
      $$
  \item For any dyadic cube $I\in\cd(\rn)$, define
  \begin{align}\label{eq-dI}
d_I:=\sum_{\gfz{J\in\cd(Q),\ J\cap I\neq\emptyset}{ 2^{-1}\ell(I)\le \ell(J)< \ell(I)}} \mu_J.
\end{align}
Then,
\begin{align}\label{eq-dya-packing}
\sum_{J\in\cd(Q)} \mu_J\left(\frac{|J|}{|Q|}\right)^{\beta} \le  \sum_{i=1}^N
\sum_{I\in\mathcal D(Q_i)} d_I\left(\frac{|I|}{|Q_i|}\right)^{\beta}.
\end{align}

\item There exist a constant $C=C([J_{\varphi^{-1}}]_{A_\infty(\rn)},\, n, \beta)$, such that
 \begin{align}\label{eq-dya-packing-varphi}
\sup_{I\in\mathcal D(Q_i)}\sum_{I'\in\cd(I)} d_{I'}\left(\frac{|\varphi^{-1}(I')|}{|\varphi^{-1}(I)|}\right)^{\beta} \le C.
\end{align}
\end{enumerate}
\end{lemma}

\begin{proof}
Fix the cube $Q\subset \rn$. Then there exists a unique integer $k_0\in\zz$ such that
$2^{-k_0-1}\le \ell(Q)<2^{-k_0}.$
The whole proof is split into four steps.

\medskip

{\bf Step 1:\, proof of (i).\,}
Let $\{Q_1,\dots,Q_N\}$ be the standard dyadic cubes in $\mathcal D_{k_0}(\rn)$ that intersect $Q$. These cubes are mutually disjoint, each with side length $2^{-k_0}$. Since
$\ell(Q)<2^{-k_0}$,
it follows that $N\le 2^n$.
Moreover,  by construction, we have $Q\subset \cup_{i=1}^N Q_i$. This proves (i).

\medskip

{\bf Step 2: An alternating expression for $d_I$ in \eqref{eq-dI}.}
Since each $Q_i$ is a standard dyadic cube in $\cd(\rn)$, the same holds for $I$ whenever $I\in\cd(Q_i)$. Without loss of generality, assume that $I\in\cd_k(Q_i)$ for some $k\in\zz_+$. This implies that $I\in \cd_{k+k_0}(\rn)$ and
$$\ell(I)=2^{-k}\ell(Q_i)=2^{-(k+k_0)}.$$
For such $I$, we claim that the definition of $d_I$ in \eqref{eq-dI} can be rewritten as
\begin{align}\label{eq-claim-dI}
d_I
=\sum_{\gfz{J\in\cd_k(Q)}{J\cap I\neq\emptyset}} \mu_J.
\end{align}
From either \eqref{eq-dI} or \eqref{eq-claim-dI}, it is clear that $d_I=0$ whenever $I\cap Q=\emptyset$.

Let us prove \eqref{eq-claim-dI}.
Take any $J\in\cd(Q)$ such that $J\cap I\neq\emptyset$. Then $J\in\cd_m(Q)$ for some $m\in\zz_+$. The condition
$
2^{-1}\ell(I)\le \ell(J)< \ell(I)
$
is equivalent to
$$
2^{-(k+k_0)-1}\le \ell(J)< 2^{-(k+k_0)}.
$$
Recalling that $2^{-k_0-1}\le \ell(Q)<2^{-k_0}$ and using $\ell(J)=2^{-m}\ell(Q)$, we obtain
$$
2^{-k-1}\ell(Q)<2^{-(k+k_0)-1}\le \ell(J)=2^{-m}\ell(Q)
$$
and
$$
2^{-m}\ell(Q)=\ell(J)< 2^{-(k+k_0)}\le 2^{-k+1}\ell(Q).
$$
These inequalities yield
$2^{-k-1}<2^{-m}<2^{-k+1}$,
so the only possible value is $m=k$.

Conversely,  if $J\in\cd_k(Q)$ and $J\cap I\neq\emptyset$, then by $\ell(J)=2^{-k}\ell(Q)$ and $2^{-k_0-1}\le \ell(Q)<2^{-k_0}$, we have
$$
2^{-1}\ell(I)=2^{-(k+k_0)-1}\le \ell(J)<2^{-(k+k_0)}=\ell(I).
$$
Thus, \eqref{eq-claim-dI} is established.

\medskip

{\bf Step 3:\, Proof of \eqref{eq-dya-packing} in (ii).\,}
 Write
\begin{align*}
  \sum_{J\in\cd(Q)} \mu_J\left(\frac{|J|}{|Q|}\right)^{\beta}
  &= \sum_{k=0}^\infty 2^{-kn\beta}\left( \sum_{J\in\cd_k(Q)} \mu_J\right).
\end{align*}
For each $J\in\cd_k(Q)$, observe that $J$ must intersect some collection of standard dyadic cubes $I\in \cd_k(Q_i)$, where $i\in\{1,2,\dots, N\}$.
 Thus,
$$
\sum_{J\in\cd_k(Q)} \mu_J
\le  \sum_{i=1}^N\sum_{I\in\mathcal D_k(Q_i)}
\sum_{\gfz{J\in\cd_k(Q)}{J\cap I\neq\emptyset}} \mu_J
=\sum_{i=1}^N\sum_{I\in\mathcal D_k(Q_i)} d_I,
$$
where the equality is a consequence of \eqref{eq-claim-dI}. It follows that
$$
\sum_{J\in\cd(Q)} \mu_J\left(\frac{|J|}{|Q|}\right)^{\beta}
\le \sum_{k=0}^\infty 2^{-kn\beta}\left( \sum_{i=1}^N\sum_{I\in\mathcal D_k(Q_i)} d_I\right)
=
 \sum_{i=1}^N
\sum_{I\in\mathcal D(Q_i)} d_I\left(\frac{|I|}{|Q_i|}\right)^{\beta}.$$
This completes the proof of \eqref{eq-dya-packing}.

\medskip

{\bf Step 4:\, Proof of \eqref{eq-dya-packing-varphi} in (iii).\,}
Fix a standard dyadic cube $I\in\cd(Q_i)$ that intersects $Q$; otherwise, all $d_{I'}=0$ for dyadic cubes $I'\in\cd(I)$, and \eqref{eq-dya-packing-varphi} holds trivially.

Since $I\cap Q\neq\emptyset$, we still assume that $I\in\cd_k(Q_i)$ for some $k\in\zz_+$ and $i\in\{1,2,\dots, N\}$. This implies that $I\in\cd_{k+k_0}(\rn)$. Then, by \eqref{eq-claim-dI}, we may write
\begin{align}\label{eq1-dI}
\sum_{I'\in\cd(I)} d_{I'}\left(\frac{|\varphi^{-1}(I')|}{|\varphi^{-1}(I)|}\right)^{\beta}
   &=\sum_{j=0}^\infty \sum_{I'\in\cd_j(I)}d_{I'}\left(\frac{|\varphi^{-1}(I')|}{|\varphi^{-1}(I)|}\right)^{\beta}\\
   &=\sum_{j=0}^\infty\, \sum_{\gfz{I'\subset I}{I'\in\cd_{j+k+k_0}(\rn)}}d_{I'}\left(\frac{|\varphi^{-1}(I')|}{|\varphi^{-1}(I)|}\right)^{\beta}\notag\\
   &=\sum_{j=0}^\infty \sum_{\gfz{I'\subset I}{I'\in\cd_{j+k+k_0}(\rn)}}
   \left(
   \sum_{\gfz{J\in\cd_{k+j}(Q)}{J\cap I'\neq\emptyset}} \mu_{J}
   \right)\left(\frac{|\varphi^{-1}(I')|}{|\varphi^{-1}(I)|}\right)^{\beta}\notag\\
   &=
   \sum_{m=k}^\infty\sum_{\gfz{J\in\cd_{m}(Q)}{J\cap I\neq\emptyset}} \mu_{J}\left(
      \sum_{\gfz{I'\in\cd_{m+k_0}(\rn)}{I'\subset I,\ I'\cap J\neq\emptyset}}
   \left(\frac{|\varphi^{-1}(I')|}{|\varphi^{-1}(I)|}\right)^{\beta}\right).  \notag
   \end{align}

Since \(I\in\mathcal D_k(Q_i)\) for some \(k\in\mathbb Z_+\) and \(I\cap Q\neq\emptyset\), we have
$
\ell(I)=2^{-k}\ell(Q_i)=2^{-k-k_0}.
$
Note that each dyadic cube \(R\in\mathcal D_k(Q)\) has side length $\ell(R)$ satisfying
$$
2^{-1}\ell(I)=2^{-k-k_0-1}\le \ell(R)=2^{-k}\ell(Q)<2^{-k-k_0}=\ell(I).
$$
Thus, by an argument similar to that used in {\bf Step 1}, there exist dyadic cubes \(\{R_1,\dots,R_L\}\subset\mathcal D_k(Q)\), with \(L\le 4^n\), such that each \(R_\nu\) intersects \(I\) and
$$
I\subset \bigcup_{\nu=1}^L R_\nu.
$$
Using the facts that
$2^{-k-k_0-1}\le \ell(R_\nu)<2^{-k-k_0}$ and $\ell(I)=2^{-k-k_0}$, we deduce
that
$$
R_\nu\subset CI\quad \text{and}\quad I\subset CR_\nu
$$
for some dimensional constant $C$. Further, by the doubling property of the \(\varphi^{-1}\)-images (see Proposition \ref{prop-varphi-weight}), we obtain
\begin{align}\label{eq-IRnu}
|\varphi^{-1}(I)| \simeq |\varphi^{-1}(R_\nu)|.
\end{align}
Moreover, if $J\in\cd_{m}(Q)$ with $m\ge k$ and $J\cap I\neq\emptyset$,
then $J$ must intersect at least one of these $R_\nu$.

For any $J\in\cd_{m}(Q)$, we have $\ell(J)=2^{-m}\ell(Q)$ and, hence,
$2^{-m-k_0-1}\le\ell(J)<2^{-m-k_0}$. Recall that $I'\in \cd_{m+k_0}(\rn)$ means that $\ell(I')=2^{-(m+k_0)}$.
Thus, for those $I'$ satisfying $I'\cap J\neq\emptyset$, by an argument analogous to the proof of \eqref{eq-IRnu}, we have
\begin{align}\label{eq-I'J}
|\varphi^{-1}(I')| \simeq |\varphi^{-1}(J)|.
\end{align}
Meanwhile, we observe that
\begin{align}\label{eq-card}
\#(\{
I'\in\cd_{m+k_0}(\rn):\ I'\subset I,\ I'\cap J\neq\emptyset
\}) \le 2^n.
\end{align}
Combining the above estimates with \eqref{eq-IRnu}, \eqref{eq-I'J}, and \eqref{eq-card}, we obtain
\begin{align}\label{eq2-dI}
  \sum_{m=k}^\infty\sum_{\gfz{J\in\cd_{m}(Q)}{J\cap I\neq\emptyset}} \mu_{J}\left(
      \sum_{\gfz{I'\in\cd_{m+k_0}(\rn)}{I'\subset I,\ I'\cap J\neq\emptyset}}
   \left(\frac{|\varphi^{-1}(I')|}{|\varphi^{-1}(I)|}\right)^{\beta}\right)
   &\ls \sum_{\nu=1}^L \sum_{m=k}^\infty\sum_{\gfz{J\in\cd_{m}(Q)}{J\cap R_\nu\neq\emptyset}} \mu_{J}
   \left(\frac{|\varphi^{-1}(J)|}{|\varphi^{-1}(R_\nu)|}\right)^{\beta}\\
   &\simeq
   \sum_{\nu=1}^L \sum_{J\in\cd(R_\nu)}\mu_{J}
   \left(\frac{|\varphi^{-1}(J)|}{|\varphi^{-1}(R_\nu)|}\right)^{\beta}\notag\\
   &\ls \sup_{R\in\cd_k(Q)} \sum_{J\in\cd(R)} \mu_{J}
   \left(\frac{|\varphi^{-1}(J)|}{|\varphi^{-1}(R)|}\right)^{\beta}\notag\\
   &\ls 1.\notag
\end{align}
The desired inequality \eqref{eq-dya-packing-varphi} follows immediately by combining \eqref{eq1-dI} and \eqref{eq2-dI}.
\end{proof}

Applying Lemma \ref{lem-dyadic-packing}, we now prove Theorem \ref{thm-dyadic-gauge}.

\begin{proof}[Proof of Theorem \ref{thm-dyadic-gauge}]
By \eqref{eq2-gauge}, the first inequality in \eqref{eq-dyadic-gauge} is immediate.

Now fix a cube $Q\subset\rn$ and suppose that $\{\mu_J\}_{J\in\mathcal D(Q)}$ belongs to the test sequence space
$\mathscr T_\beta^*(Q)$ defined in \eqref{eq0-dual-gauge}. Equivalently,
$\{\mu_J\}_{J\in\mathcal D(Q)}$ is a nonnegative sequence satisfying the $(\beta,\,\varphi)$-packing condition on $Q$.
Applying Lemma \ref{lem-dyadic-packing}, we obtain a family of dyadic cubes
$\{Q_1,\dots, Q_N\}\subset\cd(\rn)$ satisfying \eqref{eq-dya-packing} and \eqref{eq-dya-packing-varphi}.
In particular, for the sequence $\{d_I\}_{I\in\cd(\rn)}$ defined in \eqref{eq-dI}, it follows from \eqref{eq-dya-packing-varphi} that
$$
\left\{C^{-1} d_I\right\}_{I\in\cd(Q_i)} \in \mathscr T_\beta^*(Q_i)
$$
for some constant $C\in(0,\infty)$.
Combining this with \eqref{eq-dya-packing} and \eqref{eq-dual-gauge}, we  obtain
\begin{align*}
\sum_{J\in\cd(Q)} \mu_J\left(\frac{|J|}{|Q|}\right)^{\beta}
&\le  \sum_{i=1}^N
\sum_{I\in\mathcal D(Q_i)} d_I\left(\frac{|I|}{|Q_i|}\right)^{\beta} \\
&\le C \sum_{i=1}^N \|\varphi\|_{\mathscr G_\beta(Q_i)} \\
&\le 2^n C\sup_{R\in \cd(\rn)}\|\varphi\|_{\mathscr G_\beta(R)}.
\end{align*}
Since $\{\mu_J\}_{J\in\mathcal D(Q)}$ is arbitrary, it follows from
\eqref{eq-dual-gauge} that
$$
\|\varphi\|_{\mathscr G_\beta(Q)}
\le C \sup_{R\in\cd(\rn)} \|\varphi\|_{\mathscr G_\beta(R)}.
$$
Taking supremum over all cubes \(Q\subset\mathbb R^n\) gives precisely the desired second inequality in \eqref{eq-dyadic-gauge}.
\end{proof}

\subsection{Finite gauge under volume comparability}\label{ss3.3}

As can be seen from the following proposition, the definition of the capacity gauge  depends only on
the pullback-volume ratios
$
|\varphi^{-1}(J)|/|\varphi^{-1}(I)|
$ for dyadic cubes $J\subset I$.

\begin{proposition}
\label{prop:gauge-pullback-measure}
Let \(\beta\in\mathbb R\) and
\(\varphi,\,\psi:\ \mathbb R^n\to\mathbb R^n\) be homeomorphisms. Suppose that there
exists a constant \(c\in(0,\infty)\) such that for every cube \(I\subset\mathbb R^n\),
$$
0<|\varphi^{-1}(I)|
=
c\,|\psi^{-1}(I)|<\infty.
$$
Then
$\|\varphi\|_{\mathscr G_\beta} = \|\psi\|_{\mathscr G_\beta}.$
\end{proposition}

\begin{proof}
For every pair of nested cubes \(J\subset I\), the assumption gives
$$
\frac{|\varphi^{-1}(J)|}{|\varphi^{-1}(I)|}
=
\frac{c|\psi^{-1}(J)|}{c|\psi^{-1}(I)|}
=
\frac{|\psi^{-1}(J)|}{|\psi^{-1}(I)|}.
$$
Thus, the test sequence spaces
\(\mathscr T_\beta(Q)\) in \eqref{eq0-gauge} associated with \(\varphi\) and \(\psi\) coincide. It follows directly from
\eqref{eq1-gauge} that
$$
\|\varphi\|_{\mathscr G_\beta(Q)}
=
\|\psi\|_{\mathscr G_\beta(Q)}.
$$
Taking the supremum over all cubes \(Q\) proves the global identity.
\end{proof}

Next, we show that a uniform two-sided relative-volume comparison with Lebesgue volume is sufficient for finiteness of the gauge.

\begin{proposition}
\label{prop-vol-comp}
Let \(\beta\in\mathbb R\) and
\(\varphi:\mathbb R^n\to\mathbb R^n\) be a homeomorphism.
Suppose there exists \(L\in(0,\infty)\) such that for any cube \(I\subset\mathbb R^n\) and \(J\in\mathcal D(I)\),
\begin{align}\label{eq-comp}
\begin{cases}
\dfrac{|\varphi^{-1}(J)|}{|\varphi^{-1}(I)|}
\le
L\dfrac{|J|}{|I|} &\quad\,\text{for }\, \beta\le 0;\\[2mm]
\dfrac{|\varphi^{-1}(J)|}{|\varphi^{-1}(I)|}
\ge
L^{-1}\dfrac{|J|}{|I|} &\quad\,\text{for }\, \beta\ge 0.
\end{cases}
\end{align}
Then $\|\varphi\|_{\mathscr G_\beta}
\le
L^{|\beta|}.$
\end{proposition}

\begin{proof}
Fix a cube \(Q\subset\mathbb R^n\). For  \(I,J\in\mathcal D(Q)\) with \(J\subset I\), the volume comparison condition \eqref{eq-comp} implies
$$
\left(
\frac{|I|}{|J|}
\frac{|\varphi^{-1}(J)|}{|\varphi^{-1}(I)|}
\right)^\beta
\ge
L^{-|\beta|}.
$$
Define \(\{\lambda_I\}_{I\in\mathcal D(Q)}\) by setting \(\lambda_Q=L^{|\beta|}\) and \(\lambda_I=0\) for all \(I\neq Q\). Then, for every \(J\in\mathcal D(Q)\), we have
$$
\sum_{\gfz{I\in\mathcal D(Q)}{I\supset J}}
\lambda_I
\left(
\frac{|I|}{|J|}
\frac{|\varphi^{-1}(J)|}{|\varphi^{-1}(I)|}
\right)^\beta
\ge
\lambda_Q L^{-|\beta|}
=1.
$$
This proves that \(\{\lambda_I\}_{I\in\mathcal D(Q)}\in \mathscr T_\beta(Q)\).
Applying \eqref{eq1-gauge} now yields
$$
\|\varphi\|_{\mathscr G_\beta(Q)}
\le
\sum_{I\in\mathcal D(Q)} \lambda_I \left(\frac{|I|}{|Q|}\right)^{\beta}
=
L^{|\beta|}.
$$
Taking the supremum over all cubes \(Q\) yields the desired estimate.
\end{proof}

The following corollary provides a sufficient condition ensuring \eqref{eq-comp}
under bi-Lipschitz control, whereas general quasisymmetry does not imply this volume comparability.

\begin{corollary}\label{cor-bilip}
Let \(\varphi:\mathbb R^n\to\mathbb R^n\) be \(L\)-bi-Lipschitz, that is, for all \(x,\, y\in\mathbb R^n\),
$$L^{-1}|x-y| \le |\varphi(x)-\varphi(y)| \le L|x-y|. $$
Then
$\|\varphi\|_{\mathscr G_\beta}
\le
L^{2n|\beta|}.$
\end{corollary}

\begin{proof}
For any Lebesgue measurable set \(E\subset\mathbb R^n\), the bi-Lipschitz change-of-variables yields
$$
L^{-n}|E|\le |\varphi^{-1}(E)|\le L^n|E|,
$$
 which implies the comparability condition \eqref{eq-comp}. An application of Proposition~\ref{prop-vol-comp} then yields the desired upper bound for the capacity gauge.
\end{proof}

\subsection{Finite gauge when   $\beta=0$ and $\beta\in(1,\infty)$}\label{ss3.4}
We begin with the following elementary identity for dyadic cubes, which is simple but useful.

\begin{lemma}\label{lem-sum-cube}
Let $p\in(1,\infty)$ and $Q\subset\rn$ be a cube. Then,
\begin{align}\label{eq-sum-cubes}
  \sum_{I\in \cd(Q)}  \left(\frac{|I|}{|Q|}\right)^p =\frac1{1-2^{-n(p-1)}}.
  \end{align}
\end{lemma}

\begin{proof}
For  any \(p > 1\), by a direct calculation, we obtain
 \begin{align*}
  \sum_{\gfz{I\in \cd(Q)}{I\subset Q}}  \left(\frac{|I|}{|Q|}\right)^p
  &=\sum_{k=0}^\infty \sum_{\gfz{I\in\cd(Q)}{\ell(I)=2^{-k}\ell(Q)}}  \left(\frac{|I|}{|Q|}\right)^p\\
  &= \sum_{k=0}^\infty \sum_{\gfz{I\in\cd(Q)}{\ell(I)=2^{-k}\ell(Q)}}  2^{-knp} \notag\\
  &=\sum_{k=0}^\infty 2^{kn}\cdot  2^{-knp} \notag\\
  &=\frac1{1-2^{-n(p-1)}}, \notag
  \end{align*}
as desired.
\end{proof}

Next, we prove the following universal estimate for finite capacity gauge.

\begin{proposition}\label{prop-beta=0>1}
Let \(\beta\in\mathbb R\) and
\(\varphi:\ \mathbb R^n\to\mathbb R^n\) be a homeomorphism.

\begin{enumerate}[label=\textup{(\roman*)}]
    \item      $    \|\varphi\|_{\mathscr G_\beta}\ge1. $

    \item If
    \(\varphi(x)=Ax+b\) for some \(A\in GL(n,\mathbb R)\) and \(b\in\mathbb R^n\), then
 $\|\varphi\|_{\mathscr G_\beta}=1$. In particular,
  $\|\operatorname{id}\|_{\mathscr G_\beta}=1.$

\item If $\beta=0$, then $\|\varphi\|_{\mathscr G_0}=1$.

\item If $\beta\in(1,\infty)$, then
$1\le \|\varphi\|_{\mathscr G_\beta} \le (1-2^{-n(\beta-1)})^{-1}.$
\end{enumerate}
\end{proposition}

\begin{proof}
To prove (i), fix a cube \(Q\subset\mathbb R^n\). If \(\mathscr T_\beta(Q)=\emptyset\), then \(\|\varphi\|_{\mathscr G_\beta(Q)}=\infty\) and, hence, \(\|\varphi\|_{\mathscr G_\beta}=\infty\ge 1\). Now assume \(\mathscr T_\beta(Q)\neq\emptyset\). Then, for any \(\{\lambda_I\}_{I\in\mathcal D(Q)}\in\mathscr T_\beta(Q)\), by \eqref{eq0-gauge}, we have
$$
\sum_{\gfz{I\in\mathcal D(Q)}{I\supset J}}
\lambda_I
\left(
\frac{|I|}{|J|}
\frac{|\varphi^{-1}(J)|}{|\varphi^{-1}(I)|}
\right)^\beta
\ge 1\quad\text{for all } J\in\mathcal D(Q).
$$
In particular, taking \(J=Q\), the sum reduces to the single term \(I=Q\), which gives \(\lambda_Q\ge 1\). Therefore,
$$
\sum_{\gfz{I\in\mathcal D(Q)}{I\subset Q}}
\lambda_I
\left(\frac{|I|}{|Q|}\right)^\beta
\ge
\lambda_Q
\ge 1.
$$
Taking the infimum over all \(\{\lambda_I\}_I\in\mathscr T_\beta(Q)\)
yields \(\|\varphi\|_{\mathscr G_\beta(Q)}\ge 1\).
Thus, \(\|\varphi\|_{\mathscr G_\beta}\ge 1\). This proves (i).

We now prove (ii). If \(\varphi(x)=Ax+b\), then for any cubes $I$ and $J$, we have
$$
\frac{|\varphi^{-1}(J)|}{|\varphi^{-1}(I)|}
=
\frac{|J|}{|I|},
$$
which implies that \(\varphi\) satisfies \eqref{eq-comp} with constant \(L=1\).
Applying Proposition \ref{prop-vol-comp} therefore gives $\|\varphi\|_{\mathscr G_\beta}\le 1$. Combined with (i), this yields \(\|\varphi\|_{\mathscr G_\beta}=1\). In particular, taking \(\varphi(x)=x\) gives
$
\|\mathrm{id}\|_{\mathscr G_\beta}=1.
$

To prove (iii), note that the lower bound \(\|\varphi\|_{\mathscr G_0}\ge 1\) follows from (i), while the upper bound \(\|\varphi\|_{\mathscr G_0}\le 1\) is a consequence of Theorem \ref{thm:upward-duality}. Thus, we obtain \(\|\varphi\|_{\mathscr G_0}=1\).

Finally, for (iv), we let $\beta\in(1,\infty)$ and fix a cube \(Q\subset\mathbb R^n\). For each \(I\in\mathcal D(Q)\), set \(\lambda_I:=1\). This sequence $\{\lambda_I\}_{I \in \mathcal{D}(Q)}$ belongs to the test sequence space $\mathscr T_\beta(Q)$. Indeed, for any $J \in \mathcal{D}(Q)$, we have
$$
\sum_{\gfz{I\in \cd(Q), I\subset Q}{I\supset J}} \lambda_I \left(
\frac{|I|}{|J|}\cdot\frac{|\varphi^{-1}(J)|}{|\varphi^{-1}(I)|}\right)^{\beta }
\ge \lambda_J \left(
\frac{|J|}{|J|}\cdot\frac{|\varphi^{-1}(J)|}{|\varphi^{-1}(J)|}\right)^{\beta }
=\lambda_J=1.
$$
By  \eqref{eq-sum-cubes} and \eqref{eq1-gauge}, together with the assumption $\beta\in(1,\infty)$, we obtain
that for any cube $Q\subset\rn$,
  \begin{align*}
\|\varphi\|_{\mathscr G_\beta(Q)}
&\le
\sum_{I\in \cd(Q)} \lambda_I \left(\frac{|I|}{|Q|}\right)^{\beta }
 =\sum_{I\in \cd(Q)}  \left(\frac{|I|}{|Q|}\right)^{\beta }
=\frac1{1-2^{-n(\beta-1)}}.
\end{align*}
This proves the desired estimate in (iv).
\end{proof}

The following result shows that the capacity gauge is invariant under two types of transformations: invertible affine changes of variables in the domain, and signed coordinate permutations combined with uniform scalings in the range.

\begin{proposition}\label{prop:gauge-affine-invariance}
Let \(\beta\in\mathbb R\) and
\(\varphi:\ \mathbb R^n\to\mathbb R^n\) be a homeomorphism.

\begin{enumerate}[label=\textup{(\roman*)}]
    \item If
    \(T(x)=Ax+b\) for some \(A\in GL(n,\mathbb R)\) and \(b\in\mathbb R^n\), then
    $$
    \|\varphi\circ T\|_{\mathscr G_\beta}
    =
    \|\varphi\|_{\mathscr G_\beta}.
    $$

    \item If
    \(S(x)=aUx+b\) for some \(a\in(0,\infty)\), \(b\in\mathbb R^n\), and a signed coordinate permutation \(U\), then
    $$
    \|\varphi\circ S\|_{\mathscr G_\beta}
    =
    \|\varphi\|_{\mathscr G_\beta}
    =
    \|S\circ\varphi\|_{\mathscr G_\beta}.
    $$
\end{enumerate}
\end{proposition}

\begin{proof}
For item~(i), observe that
$
(\varphi\circ T)^{-1}=T^{-1}\circ\varphi^{-1}.
$
Thus, for every measurable set \(E\subset\mathbb R^n\),
$$
|(\varphi\circ T)^{-1}(E)|
=
|T^{-1}(\varphi^{-1}(E))|
=
|\det A|^{-1}|\varphi^{-1}(E)|.
$$
Consequently, for any \(I,J\in\mathcal D(Q)\) with \(J\subset I\),
$$
\frac{|(\varphi\circ T)^{-1}(J)|}
{|(\varphi\circ T)^{-1}(I)|}
=
\frac{|\varphi^{-1}(J)|}
{|\varphi^{-1}(I)|}.
$$
It follows from \eqref{eq0-gauge} that the test sequence spaces for \(\varphi\) and \(\varphi\circ T\) coincide. Then, applying \eqref{eq1-gauge} and \eqref{eq2-gauge} yield the desired identity in (i).

For item~(ii), since \(S\) is an invertible affine map, the equality \(\|\varphi\circ S\|_{\mathscr G_\beta}=\|\varphi\|_{\mathscr G_\beta}\) follows directly from (i). It remains to prove \(\|S\circ\varphi\|_{\mathscr G_\beta}=\|\varphi\|_{\mathscr G_\beta}\). Note that \(S^{-1}\) sends axis-parallel cubes to axis-parallel cubes and induces a tree isomorphism
$
I\mapsto S^{-1}(I)
$
from \(\mathcal D(Q)\) onto \(\mathcal D(S^{-1}Q)\). Moreover,
$$
\frac{|S^{-1}(I)|}{|S^{-1}(Q)|}
=
\frac{|I|}{|Q|}.
$$
Since
$(S\circ\varphi)^{-1}=\varphi^{-1}\circ S^{-1}$,
we have
$$
|(S\circ\varphi)^{-1}(I)|
=
|\varphi^{-1}(S^{-1}(I))|.
$$
Thus, after the change of indices \(I\mapsto S^{-1}(I)\), the constraints and objective functionals for \(\varphi\) and \(S\circ\varphi\) coincide. Again, using \eqref{eq0-gauge} and \eqref{eq1-gauge} yields
$$
\|S\circ\varphi\|_{\mathscr G_\beta(Q)}
=
\|\varphi\|_{\mathscr G_\beta(S^{-1}(Q))}.
$$
Then taking the supremum over all cubes \(Q\subset\mathbb R^n\) yields \(\|S\circ\varphi\|_{\mathscr G_\beta}=\|\varphi\|_{\mathscr G_\beta}\). This proves (ii).
\end{proof}

\subsection{Finite gauge for $A_1(\mathbb R^n)$ weights}\label{ss3.5}

The following proposition provides a sufficient condition for \(\varphi\) to have finite gauge; as we shall see in Remark~\ref{rem-A1-gauge} below, this condition is not necessary.

\begin{proposition}\label{prop-A1toGauge}
  Let $\beta\in \mathbb R$ and $\varphi:\ \mathbb{R}^n \to \mathbb{R}^n$ be $\eta$-quasisymmetric.
Then
\begin{align}\label{eq-A1gauge}
\|\varphi\|_{\mathscr G_\beta} \le
\begin{cases}
(1-2^{-n\beta})^{-1}[J_{\varphi^{-1}}]_{A_1(\rn)}^{\beta }\quad
&\ \text{as}\ \,\beta> 0\ \,\text{and}\ \, J_{\varphi^{-1}}\in A_1(\rn);\vspace{0.15cm}\\
\left(C(\eta,n)[J_\varphi]_{A_1(\rn)}\right)^{-\beta}\quad
&\ \text{as}\ \,\beta\le 0\ \,\text{and}\ \, J_{\varphi}\in A_1(\rn),
\end{cases}
\end{align}
where $C(\eta,n)$ is a positive constant depending only on $n$ and $\eta$.
\end{proposition}

\begin{proof}
Consider first the case $\beta>0$ and $J_{\varphi^{-1}}\in A_1(\mathbb R^n)$.
For any cube $Q\subset\mathbb R^n$ and $I\in\mathcal D(Q)$, set
  $$\lambda_I:=\left([J_{\varphi^{-1}}]_{A_1(\rn)}\right)^{\beta }\,\frac{|I|}{|Q|}.$$
Now, we verify that such $\{\lambda_I\}_{I\in\mathcal D(Q)}\in\mathscr T_\beta(Q)$.
For any $J\in\mathcal D(Q)$, since $J_{\varphi^{-1}}\in A_1(\mathbb R^n)$, we have  $$
 \mathop{\mathrm{ess\,inf}}_J J_{\varphi^{-1}}\le \frac{|\varphi^{-1}(J)|}{|J|}
 =\fint_J J_{\varphi^{-1}}\,dx
 \le[J_{\varphi^{-1}}]_{A_1(\rn)}\mathop{\mathrm{ess\,inf}}_Q J_{\varphi^{-1}},
  $$
which,
combined with $\beta>0$, gives
   \begin{align*}
  \sum_{\gfz{I\in \cd(Q)}{J\subset I}} \lambda_I \left(
\frac{|I|}{|J|}\cdot\frac{|\varphi^{-1}(J)|}{|\varphi^{-1}(I)|}\right)^{\beta }
&\ge \lambda_Q \left(
\frac{|Q|}{|J|}\cdot\frac{|\varphi^{-1}(J)|}{|\varphi^{-1}(Q)|}\right)^{\beta }\\
&  \ge\lambda_Q
\left(\frac{\mathop{\mathrm{ess\,inf}\,}_{J}  J_{\varphi^{-1}}}
{[J_{\varphi^{-1}}]_{A_1(\rn)}\mathop{\mathrm{ess\,inf}\,}_{Q}  J_{\varphi^{-1}}}\right)^{\beta } \\
&    \ge\lambda_Q \left([J_{\varphi^{-1}}]_{A_1(\rn)}\right)^{-\beta }\\
&=1.
  \end{align*}
Thus $\{\lambda_I\}_{I\in\mathcal D(Q)}\in\mathscr T_\beta(Q)$.
Then, by \eqref{eq1-gauge} and \eqref{eq-sum-cubes}, together with $\beta>0$, we obtain
  \begin{align*}
    \|\varphi\|_{\mathscr G_\beta(Q)}
     \le \sum_{\gfz{I\in \cd(Q)}{I\subset Q}} \lambda_I \left(\frac{|I|}{|Q|}\right)^{\beta }
    = [J_{\varphi^{-1}}]_{A_1(\rn)}^{\beta } \sum_{\gfz{I\in \cd(Q)}{I\subset Q}}
    \left(\frac{|I|}{|Q|}\right)^{\beta+1}
   = (1-2^{-n\beta})^{-1}[J_{\varphi^{-1}}]_{A_1(\rn)}^{\beta }.
  \end{align*}
Taking the supremum over all cubes $Q$ yields the desired estimate in \eqref{eq-A1gauge}
 for the case $\beta>0$.

Now, we consider the case $\beta\le 0$ and \(J_\varphi\in A_1(\mathbb R^n)\).
Since $\varphi$ is $\eta$-quasisymmetric, it follows from \eqref{eq-cube-varphi} that
 there exists a constant $c=c(\eta,n)>0$ such that for any cube $I\subset\mathbb R^n$,
  $$
  B\left(\varphi^{-1}(c_I),\, c r_I\right)
  \subset \varphi^{-1}(I)\subset
  B\left(\varphi^{-1}(c_I),\,  r_I\right),
 $$
 where $c_I$ denotes the center of $I$ and $r_I:=\operatorname{diam}\varphi^{-1}(I)$.
 On the one hand, we have
\begin{align}\label{eq1-Iphi-1I}
  \frac{|I|}{|\varphi^{-1}(I)| }=
  \frac1{|\varphi^{-1}(I)|}\int_{\varphi^{-1}(I)} J_\varphi(y)\,dy\ge \mathop{\mathrm{ess\,inf}\,}_{\varphi^{-1}(I)}  J_\varphi.
\end{align}
On the other hand,
\begin{align}\label{eq2-Iphi-1I}
  \frac{|I|}{|\varphi^{-1}(I)| }
  &\le \frac1{|B(\varphi^{-1}(c_I),\, c r_I)|}\int_{B(\varphi^{-1}(c_I),\, r_I)} J_\varphi(y)\,dy\\
  &
  \le c^{-n} [J_\varphi]_{A_1(\rn)}\,\mathop{\mathrm{ess\,inf}\,}_{B(\varphi^{-1}(c_I),\,  r_I)}  J_\varphi\notag\\
  &
  \le c^{-n} [J_\varphi]_{A_1(\rn)}\,\mathop{\mathrm{ess\,inf}\,}_{\varphi^{-1}(I)}  J_\varphi. \notag
  \end{align}
We now use the dual formulation from Theorem \ref{thm:upward-duality}.
Fix a cube $Q\subset\rn$ and take $\{\mu_J\}_{J}\in \mathscr T_\beta^*(Q)$, that is,
$$
\sum_{\gfz{J\in \cd(Q)}{J\subset I}} \mu_J \left(
\frac{|\varphi^{-1}(J)|}{|\varphi^{-1}(I)|}\right)^{\beta }
\le 1 \quad\, \text{for all } \, I\in\mathcal D(Q).$$
 Applying \eqref{eq1-Iphi-1I}, \eqref{eq2-Iphi-1I}  and using  $\beta \le 0$, we obtain that for any $J\subset I$,
\begin{align*}
\left(
\frac{|\varphi^{-1}(J)|}{|\varphi^{-1}(I)|}\right)^{\beta }
&=
\left(\frac{|J|}{|I|}\right)^{\beta } \left(
\frac{\frac{|J|}{|\varphi^{-1}(J)|}}{\frac{|I|}{|\varphi^{-1}(I)|}}\right)
^{-\beta}\\
&\ge  c^{-n\beta}
\left(\frac 1{ [J_\varphi]_{A_1(\rn)}}\right)^{-\beta}
 \left(\frac{|J|}{|I|}\right)^{\beta }
\left(
\frac{\mathop{\mathrm{ess\,inf}\,}_{\varphi^{-1}(J)}  J_\varphi}{ \mathop{\mathrm{ess\,inf}\,}_{\varphi^{-1}(I)}  J_\varphi}\right)^{-\beta}\\
&\ge c^{-n\beta}
\left(\frac 1{ [J_\varphi]_{A_1(\rn)}}\right)^{-\beta}
 \left(\frac{|J|}{|I|}\right)^{\beta }.
\end{align*}
Thus, for every $I\in\mathcal D(Q)$,
$$
\sum_{\gfz{J\in \cd(Q)}{J\subset I}} \mu_J \left(
\frac{|J|}{|I|}\right)^{\beta }
\le c^{n\beta}
\left([J_\varphi]_{A_1(\rn)}\right)^{-\beta}.$$
Taking $I=Q$ and then the supremum over all $\{\mu_J\}\in\mathscr T_\beta^*(Q)$,
we obtain from \eqref{eq-dual-gauge} that
$$
\|\varphi\|_{\mathscr G_\beta(Q)}
\le
c^{n\beta}
\left([J_\varphi]_{A_1(\mathbb R^n)}\right)^{-\beta}.
$$
Taking  this with  \eqref{eq0-dual-gauge} yields that \eqref{eq-A1gauge} holds in the case $\beta\le0$.
\end{proof}

\subsection{The distortion-dependent case  $\beta\in(-\infty, 0)$}\label{ss3.6}

Let $n\in\nn$. Fix $p\in(-1,\infty)$ and consider the radial power-type function
\begin{align}\label{eq-varphip}
\varphi_p(x):= x |x|^p\quad \ \text{for }\ x\in\rn,
\end{align}
which is a homeomorphism of $\mathbb R^n$. Its inverse is given by
\begin{align}\label{eq-psip}
\psi_p(y):= \varphi_p^{-1}(y)=y |y|^{-\frac{p}{p+1}}\quad \text{for } y\in\mathbb R^n.
\end{align}
Furthermore, $\varphi_p$ is quasisymmetric, with distortion function
$$
\eta(t)= C\max\left\{t^{p+1},\, t^{\frac{1}{p+1}}\right\},\qquad t\in(0,\infty),
$$
where $C=C(p,n)$ is a constant depending only on $p$ and $n$. A direct calculation yields
$$
J_{\varphi_p}(x)= (p+1)|x|^{pn}
$$
and
$$
J_{\varphi_p^{-1}}(y)=J_{\psi_p}(y)= (p+1)^{-1}|y|^{-\frac{pn}{p+1}}.
$$
Now, we consider $\|\varphi_p\|_{\mathscr G_\beta}$ for $\beta\in(-\infty,0)$.

\begin{proposition}\label{prop-gauge-finite-infity}
Let $\beta\in(-\infty, 0)$.
Then, for the radial stretching function $\varphi_p$ defined in \eqref{eq-varphip}, we have
\begin{align}\label{eq-varphigaugefinite}
\|\varphi_p\|_{\mathscr G_\beta}<\infty\quad\  \text{for all}\ \ p\in (-1,0]
\end{align}
whereas
\begin{align}\label{eq-varphigauge-infinite}
\|\varphi_p\|_{\mathscr G_\beta}=\infty\quad\  \text{for all}\ \ p\in (0,\infty).
\end{align}
\end{proposition}

\begin{proof}
It is a classical fact from Muckenhoupt weight theory (see, e.g., \cite[Example~7.1.7]{Grafakos-CFAbook})
that
\begin{align}\label{eq-power-Afz}
|x|^a\in A_\infty(\mathbb R^n) \quad\Leftrightarrow\quad a>-n
\end{align}
and
\begin{align}\label{eq-power-A1}
|x|^a\in A_1(\mathbb R^n) \quad\Leftrightarrow\quad -n<a\le 0.
\end{align}
In view of the explicit forms of $J_{\varphi_p}$, we have
$$J_{\varphi_p}\in A_1(\mathbb R^n)\quad\Leftrightarrow\quad  p\in(-1,0]$$
Consequently, Proposition \ref{prop-A1toGauge} for the case $\beta\in(-\infty, 0)$ implies that \eqref{eq-varphigaugefinite}
holds.

It remains to prove \eqref{eq-varphigauge-infinite} in the case $p\in(0,\infty)$.
 We claim that for every cube $I\subset\mathbb R^n$, with center $c_I$ and side length $\ell_I$,
\begin{align}\label{eq-vol-varphiInvI}
|\varphi_p^{-1}(I)|= |\psi_p(I)| \simeq \left(|c_I|+\ell_I\right)^{-\frac{np}{p+1}} |I|.
\end{align}
where $\psi_p=\varphi_p^{-1}$ is as in \eqref{eq-psip}, and the implicit constants are independent of $I$.

Now, we prove the claim \eqref{eq-vol-varphiInvI}. If $|c_I|\ge 2n\ell_I$, then $I$ is far from the origin, so that
$|y|\simeq |c_I|$ for every $y\in I$,
with implicit constants independent of $I$. Consequently,
\begin{align*}
 |\psi_p(I)|
 =\int_I J_{\psi_p}(y)\,dy
  = \frac1{p+1} \int_I |y|^{-\frac{pn}{p+1}}  \,dy
 \simeq |c_I|^{-\frac{pn}{p+1}} |I|,
\end{align*}
which proves  \eqref{eq-vol-varphiInvI} in the case $|c_I|\ge 2n\ell_I$.

If $|c_I|<2n\ell_I$, then $I$ lies near the origin and $I\subset B(0,3n\ell_I)$. Thus, we obtain
\begin{align*}
 |\psi_p(I)|
  = \frac1{p+1} \int_I |y|^{-\frac{pn}{p+1}}  \,dy
  \ge \frac1{p+1} \int_I (3n\ell_I)^{-\frac{pn}{p+1}}  \,dy
  \simeq \ell_I^{-\frac{pn}{p+1}+n}
\end{align*}
and
\begin{align*}
 |\psi_p(I)|
 \le \frac1{p+1} \int_{B(0, 3n\ell_I)} |y|^{-\frac{pn}{p+1}}  \,dy
 \simeq \int_0^{3n\ell_I} \rho^{-\frac{pn}{p+1}+n-1}\,d\rho
 \simeq \ell_I^{-\frac{pn}{p+1}+n}.
\end{align*}
Combining these two estimates yields \eqref{eq-vol-varphiInvI} in the case $|c_I|<2n\ell_I$.
This proves the claim \eqref{eq-vol-varphiInvI}.

With \eqref{eq-vol-varphiInvI} at hand, we now apply Definition \ref{def-gauge}
to prove that $\|\varphi_p\|_{\mathscr G_\beta}=\infty$ for $p\in(0,\infty)$ and $\beta\in(-\infty,0)$.
By \eqref{eq0-gauge}, we take an arbitrary sequence
$\{\lambda_I\}_{I\in\mathcal D(Q)}\subset[0,\infty)$ satisfying
$$
\sum_{\gfz{I\in \cd(Q)}{I\supset J}} \lambda_I \left(
\frac{|I|}{|J|}\cdot\frac{|\varphi_p^{-1}(J)|}{|\varphi_p^{-1}(I)|}\right)^{\beta }
\ge 1 \quad\,\text{for all }\, J\in\mathcal D(Q).
$$ Consider the cube $Q_0=[0,1)^n$ and its dyadic subcubes $I_k=[0,2^{-k})^n$. Note that the center $c_{I_k}$ of $I_k$  satisfies
$$
|c_{I_k}|=\frac{\sqrt n}{2} \ell_{I_k}= \sqrt n 2^{-k-1}.
$$
For $J_j:=[0,2^{-j})^n$ with $j\in\mathbb N$,
using the above identity together with \eqref{eq-vol-varphiInvI}, we obtain
 \begin{align*}
   1 &\le  \sum_{\gfz{I\in \cd(Q_0)}{I\supset J_j}} \lambda_I \left(
\frac{|I|}{|J_j|}\cdot\frac{|\varphi_p^{-1}(J_j)|}{|\varphi_p^{-1}(I)|}\right)^{\beta }\\
&= \sum_{k=0}^j \lambda_{I_k} \left(
\frac{|I_k|}{|J_j|}\cdot\frac{|\varphi_p^{-1}(J_j)|}{|\varphi_p^{-1}(I_k)|}\right)^{\beta }\\
   &\simeq \sum_{k=0}^j \lambda_{I_k}
   \left(
\frac{|c_{J_j}|+\ell_{J_j}}{|c_{I_k}|+\ell_{I_k}}\right)^{-\frac{\beta np}{p+1}} \\
 &\simeq \sum_{k=0}^j \lambda_{I_k}
   \left(
\frac{\ell_{J_j}}{\ell_{I_k}}\right)^{-\frac{\beta n p}{p+1}} \\
&\simeq  \sum_{k=0}^j \lambda_{I_k} 2^{-(k-j)\frac{\beta n p}{p+1}},
 \end{align*}
 which implies
 $$
 2^{-j\frac{\beta n p}{p+1}} \ls \sum_{k=0}^j \lambda_{I_k} 2^{-k\frac{\beta n p}{p+1}}.
 $$
  Consequently, for any $j\in\nn$, we have
\begin{align*}
  \sum_{I\in \cd(Q_0)} \lambda_I \left(\frac{|I|}{|Q_0|}\right)^{\beta }
  &\ge \sum_{k=0}^\infty \sum_{I=I_k}\lambda_I \left(\frac{|I|}{|Q_0|}\right)^{\beta }\\
  &= \sum_{k=0}^\infty \lambda_{I_k}
  2^{-kn\beta}\\
  &\ge \sum_{k=0}^\infty \lambda_{I_k}
  2^{-kn\beta \cdot \frac{p}{p+1}}\\
  &\gs 2^{-j\frac{\beta n p}{p+1}},
 \end{align*}
where in the penultimate step the inequality
$2^{-kn\beta}\ge 2^{-kn\beta p/(p+1)}$ holds because $\beta<0$ and $p>0$.
Then, letting $j\to\infty$ and using $\beta<0$, we conclude that
$$
\sum_{I\in\mathcal D(Q_0)} \lambda_I \left(\frac{|I|}{|Q_0|}\right)^\beta = \infty.
$$
 From this and \eqref{eq1-gauge}, it follows that
 $ \|\varphi_p\|_{\mathscr G_\beta}
\ge \|\varphi_p\|_{\mathscr G_\beta(Q_0)}=\infty,$
which proves \eqref{eq-varphigauge-infinite}.
\end{proof}

\begin{remark}\label{rem-A1-gauge}
Now, we examine the finiteness of $\|\varphi_p\|_{\mathscr G_\beta}$
in the range $\beta\in[0,\infty)$.

\begin{enumerate}[label=\textup{(\roman*)}]
  \item If $\beta=0$ or $\beta\in(1,\infty)$, then Proposition \ref{prop-beta=0>1} yields
 $\|\varphi_p\|_{\mathscr G_\beta}<\infty$ for every $p\in(-1,\infty)$.

  \item If $\beta\in(0,1)$ and $p\in[0,\infty)$, then by \eqref{eq-power-A1} we have
  $J_{\varphi_p^{-1}}\in A_1(\mathbb R^n)$, which implies from Proposition \ref{prop-A1toGauge} that
   $\|\varphi_p\|_{\mathscr G_\beta}<\infty$.

  \item If $\beta\in(0,1)$ and $p\in(-1,0)$, then by \eqref{eq-power-Afz} and \eqref{eq-power-A1},
we have $J_{\varphi_p^{-1}}\in A_\infty(\mathbb R^n)$ but $J_{\varphi_p^{-1}}\notin A_1(\mathbb R^n)$.
Thus Proposition \ref{prop-A1toGauge} cannot be applied.
Instead, we apply Corollary \ref{cor-JEMS-gauge}
(or equivalently, Theorems \ref{thm-JEMS} and \ref{thm-trace-gauge})
to conclude that
$$
\|\varphi_p\|_{\mathscr G_\beta}<\infty
\quad\text{when } n\ge 2 \text{ and } \beta\in(1-2/n, 1).
$$
To see this, note that the proof of \eqref{eq-vol-varphiInvI}
---particularly in the case where the cube $I$ is far from the origin---shows that
$J_{\varphi_p^{-1}}$ is a local $A_1(\mathbb R^n;\{0\})$ weight.
The degeneracy set is the single point $\{0\}$, whose local self-similar Minkowski dimension is $0$.
Consequently, the finiteness of $\|\varphi_p\|_{\mathscr G_\beta}$
follows from the boundedness of
$\mathcal C_{\varphi_p}$ on $\mathcal Q_\alpha(\mathbb R^n)$, where $\beta=1-2\alpha/n$.
\end{enumerate}

\end{remark}

\subsection{Explicit quantitative upper bound for the capacity gauge}\label{ss3.7}

Let $Q\subset\rn$ be a cube. Suppose that $J\in\mathcal D_k(Q)$ for some $k\in\mathbb Z_+$.
Then there exists a unique chain of dyadic cubes
\begin{align}\label{eq-chainJQ}
J=I_k^J \subset I_{k-1}^J\subset\cdots\subset I_1^J \subset I_0^J=Q
\quad\, \text{with}\ \, I_i^J\in\mathcal D_i(Q)\ \, \text{for  each}\ \, 0\le i\le k.
\end{align}
In the following theorem,
we propose an allocation of weights $\lambda_I$ on dyadic subcubes $I\in\mathcal D(Q)$,
which derives an explicit quantitative upper bound for the capacity gauge on $Q$.

\begin{theorem}\label{thm-gauge-calc}
Let $\beta\in\rr$ and $\varphi:\ \mathbb{R}^n \to \mathbb{R}^n$
be a homeomorphism such that $|\varphi^{-1}(I)|\in(0,\infty)$ for all cubes $I\subset\mathbb R^n$.
Given a cube $Q\subset\mathbb R^n$,  define
\begin{align}\label{eq-lambdaJ-optimal}
\lambda_J^*:=
\begin{cases}
1 &\quad \text{if}\ \, J=Q;\\
\displaystyle \left(1-\max_{0\le i\le k-1}
\left(\frac{|I_i^J|}{|J|}\cdot\frac{|\varphi^{-1}(J)|}{|\varphi^{-1}(I_i^J)|}\right)^\beta \right)_+
&\quad \text{if}\ \, J\in\mathcal D_k(Q)\ \, \text{for some}\ k\in\nn,
\end{cases}
\end{align}
where $\{I_i^J\}_{i=0}^k$ is the unique dyadic chain from $J$ to $Q$ as in  \eqref{eq-chainJQ}.
Then $\{\lambda_J^*\}_{J\in\mathcal D(Q)}\in \mathscr T_\beta(Q)$
and
\begin{align}\label{eq-gauge-optimal}
\|\varphi\|_{\mathscr G_\beta(Q)}
&\le 
1+\sum_{k\in\nn} 2^{-kn\beta} \left(\sum_{J\in\cd_k(Q)}\left(1-\max_{0\le i\le k-1}
\left(\frac{|I_i^J|}{|J|}\cdot\frac{|\varphi^{-1}(J)|}{|\varphi^{-1}(I_i^J)|}\right)^\beta \right)_+ \right).
\end{align}
\end{theorem}

The geometric interpretation of \eqref{eq-gauge-optimal} is given in
Remark \ref{rem-gauge-nearly-optimal} below.  However,
we remark  that the test sequence $\{\lambda_J^*\}_{J\in\cd(Q)}$
in \eqref{eq-lambdaJ-optimal} may not be optimal
and the series in \eqref{eq-gauge-optimal} may diverge to $\infty$.
To prove Theorem \ref{thm-gauge-calc}, we need the following elementary lemma.

\begin{lemma}\label{lem-max+}
Let  $k\in\mathbb N$ and $x_0,x_1,\dots,x_k\in(0,\infty)$. Then,
\begin{align*}
1-x_0-\sum_{i=1}^k (x_i-x_0\vee x_1\vee\cdots\vee x_{i-1})_+
=1-x_0\vee x_1\vee\cdots\vee x_k.
\end{align*}
\end{lemma}

\begin{proof}
For each $i=0,\dots,k$, set $M_i:=x_0\vee x_1\vee\cdots\vee x_i$.
We claim that for every $i=1,\dots,k$,
$$
(x_i-M_{i-1})_+ = M_i-M_{i-1}.
$$
Indeed, if $x_i\le M_{i-1}$, then $M_i=M_{i-1}$ and both sides are $0$; if $x_i>M_{i-1}$, then $M_i=x_i$ and both sides equal $x_i-M_{i-1}$.
Summing this identity over $i=1,\dots,k$, we obtain
$$
\sum_{i=1}^k (x_i-M_{i-1})_+
=
\sum_{i=1}^k (M_i-M_{i-1})
=
M_k-M_0
=
M_k-x_0.
$$
Therefore,
$$
1-x_0-\sum_{i=1}^k (x_i-x_0\vee x_1\vee\cdots\vee x_{i-1})_+
=
1-x_0-(M_k-x_0)
=
1-M_k
=
1-x_0\vee x_1\vee\cdots\vee x_k,
$$
as desired.
\end{proof}

\begin{proof}[Proof of Theorem \ref{thm-gauge-calc}]
Let $\{\lambda_J^*\}_{J\in\mathcal D(Q)}$ be defined as in \eqref{eq-lambdaJ-optimal}.
By \eqref{eq0-gauge}, we have $\{\lambda_J^*\}_{J\in\mathcal D(Q)}\in\mathscr T_\beta(Q)$ if and only if
\begin{align}\label{eq0-TQ}
\sum_{\gfz{I\in \cd(Q)}{I\supset J}} \lambda_I^* \left(
\frac{|I|}{|J|}\cdot\frac{|\varphi^{-1}(J)|}{|\varphi^{-1}(I)|}\right)^{\beta}
\ge 1 \quad \,
\text{for all}\ \, J\in \cd(Q).
\end{align}
To prove this, we proceed by induction on the scale $k\in\mathbb Z_+$, showing that for each $k$,
the family $\{\lambda_J^*\}_{J\in\mathcal D(Q)}$ satisfies \eqref{eq0-TQ} for all $J\in\mathcal D_k(Q)$.

For $k=0$, we have $\mathcal D_0(Q)=\{Q\}$.
For $J=Q$, the inequality in \eqref{eq0-TQ} reduces to $\lambda_Q^*\ge1$, which clearly holds.

For $k=1$ and $J\in\mathcal D_1(Q)$, the sum in \eqref{eq0-TQ} contains exactly two terms,
corresponding to $I=J$ and $I=Q$. Since $\lambda_Q^*=1$, the inequality in \eqref{eq0-TQ} becomes
$$
\lambda_J^* + \left(
\frac{|Q|}{|J|}\cdot\frac{|\varphi^{-1}(J)|}{|\varphi^{-1}(Q)|}
\right)^\beta \ge 1,
$$
which is precisely satisfied provided that
\begin{align*}
\lambda_J^*=\left(1- \left( \frac{|I|}{|J|}\cdot
\frac{|\varphi^{-1}(J)|}{|\varphi^{-1}(Q)|}\right)^{\beta}\right)_+.
\end{align*}
This matches the value given in \eqref{eq-lambdaJ-optimal} for $J\in\mathcal D_1(Q)$.

Now let $k\ge 2$. Assume that we have proved that
the family $\{\lambda_J^*\}_{J\in\mathcal D(Q)}$ satisfies \eqref{eq0-TQ}
for all $J\in\mathcal D_m(Q)$ with $m=0,1,\dots,k-1$.
We will show that it also satisfies \eqref{eq0-TQ} for all $J\in\mathcal D_k(Q)$.

Fix $J\in\mathcal D_k(Q)$ with $k\ge 2$.
Let $\{I_i^J\}_{i=0}^k$ be the unique dyadic chain from $J$ to $Q$
as in \eqref{eq-chainJQ}. Then the inequality in \eqref{eq0-TQ} becomes
$$
\sum_{i=0}^k \lambda_{I_i^J}^*
\left(
\frac{|I_i^J|}{|J|}\cdot\frac{|\varphi^{-1}(J)|}{|\varphi^{-1}(I_i^J)|}
\right)^{\beta}
\ge 1.
$$
Singling out the terms $i=0$ and $i=k$ (i.e., $I_0^J=Q$ and $I_k^J=J$),
and using $\lambda_Q^*=1$, this is equivalent to
\begin{align}\label{eq1-TQ}
\lambda_{J}^*
\ge 1-
\left(
\frac{|Q|}{|J|}\cdot\frac{|\varphi^{-1}(J)|}{|\varphi^{-1}(Q)|}
\right)^{\beta}
-
\sum_{i=1}^{k-1} \lambda_{I_i^J}^*
\left(
\frac{|I_i^J|}{|J|}\cdot\frac{|\varphi^{-1}(J)|}{|\varphi^{-1}(I_i^J)|}
\right)^{\beta}.
\end{align}
By the expression of $\lambda_{I_i^J}^*$  in \eqref{eq-lambdaJ-optimal}, we have
\begin{align*}
\lambda_{I_i^J}^*
=\left(1-\max_{0\le l\le i-1}
\left(\frac{|I_l^J|}{|I_i^J|}\cdot\frac{|\varphi^{-1}(I_i^J)|}{|\varphi^{-1}(I_l^J)|}\right)^\beta \right)_+.
\end{align*}
Hence,
\begin{align*}
\sum_{i=1}^{k-1} \lambda_{I_i^J}^*
\left(
\frac{|I_i^J|}{|J|}\cdot\frac{|\varphi^{-1}(J)|}{|\varphi^{-1}(I_i^J)|}
\right)^{\beta}
&= \sum_{i=1}^{k-1} \left(\left(
\frac{|I_i^J|}{|J|}\cdot\frac{|\varphi^{-1}(J)|}{|\varphi^{-1}(I_i^J)|}
\right)^{\beta}-\max_{0\le l\le i-1}
\left(\frac{|I_l^J|}{|J|}\cdot\frac{|\varphi^{-1}(J)|}{|\varphi^{-1}(I_l^J)|}\right)^\beta \right)_+.
\end{align*}
For $i=0,1,\dots, k-1$, set
$$
X_i:=\left(
\frac{|I_i^J|}{|J|}\cdot\frac{|\varphi^{-1}(J)|}{|\varphi^{-1}(I_i^J)|}
\right)^{\beta}.
$$
Then, \eqref{eq1-TQ} is equivalent to
\begin{align}\label{eq11-TQ}
\lambda_{J}^* \ge 1- X_0-
\sum_{i=1}^{k-1} \left(X_i-\max_{0\le l\le i-1}X_l \right)_+.
\end{align}
By Lemma \ref{lem-max+}, the right-hand side equals
$$
\left(1-\max_{0\le i\le k-1}X_i \right)_+.
$$
Therefore, inequality \eqref{eq11-TQ} is satisfied if
$$
\lambda_J^* = \left(1-\max_{0\le i\le k-1} X_i \right)_+.
$$
In view of the definition of $X_i$, we see that this expression coincides
with the definition of $\lambda_J^*$ in \eqref{eq-lambdaJ-optimal}.
This proves that \eqref{eq0-TQ} holds for all $J\in\mathcal D_k(Q)$.

Thus, by induction, $\{\lambda_J^*\}_{J\in\mathcal D(Q)}$ satisfies \eqref{eq0-TQ}
for all $J\in\mathcal D_k(Q)$ and all $k\in\mathbb Z_+$.
In other words, we have proved that $\{\lambda_J^*\}_{J\in\mathcal D(Q)}\in\mathscr T_\beta(Q)$.

Finally, since  we have obtained
$\{\lambda_J^*\}_{J\in\mathcal D(Q)}\in\mathscr T_\beta(Q)$,
it follows from  \eqref{eq1-gauge} that
$$
\|\varphi\|_{\mathscr G_\beta(Q)}
\le \sum_{J\in\mathcal D(Q)} \lambda_J^* \left(\frac{|J|}{|Q|}\right)^\beta
= \sum_{k\in\zz_+} 2^{-kn\beta} \left(\sum_{J\in\cd_k(Q)} \lambda_J^\ast\right).
$$
Substituting the explicit expression for $\lambda_J^*$ yields \eqref{eq-gauge-optimal}.
\end{proof}

\begin{remark}\label{rem-gauge-nearly-optimal}
Let all the assumptions be as in Theorem \ref{thm-gauge-calc}. We make the following observations.

\begin{enumerate}[label=\textup{(\roman*)}]
  \item Considering only the term $i=0$ (that is, $I_0^J=Q$) in the maximal quantity of \eqref{eq-gauge-optimal}, we obtain
  \begin{align}\label{eq-gauge-optimal-est}
1\le \|\varphi\|_{\mathscr G_\beta(Q)}\le
1+\sum_{k\in\nn} 2^{-kn\beta} \left(\sum_{J\in\cd_k(Q)}\left(1-
2^{kn\beta}\left(\frac{|\varphi^{-1}(J)|}{|\varphi^{-1}(Q)|}\right)^\beta \right)_+ \right).
\end{align}

  \item If $\beta=0$, then the series in \eqref{eq-gauge-optimal-est} vanishes, so $\|\varphi\|_{\mathscr G_0(Q)}=1$.
If $\beta\in(1,\infty)$, then \eqref{eq-gauge-optimal-est} implies
$$
1\le \|\varphi\|_{\mathscr G_\beta(Q)}
\le 1+\sum_{k=1}^{\infty} 2^{-kn\beta} \left(\sum_{J\in\mathcal D_k(Q)} 1 \right)
= 1+\sum_{k=1}^{\infty} 2^{-kn(\beta-1)}
= \frac{1}{1-2^{-n(\beta-1)}}.
$$
These two facts recover Proposition \ref{prop-beta=0>1}(iii) and (iv), respectively.

  \item Let $\beta\in(0,1)$. In this case, the following elementary inequality is known:
$$
\beta(1-t)\le 1-t^\beta \le 1-t \qquad \text{for all }\, t\in[0,1].
$$
Thus, by \eqref{eq-gauge-optimal}, we have
\begin{align}\label{eq-gauge-geom1}
\|\varphi\|_{\mathscr G_\beta(Q)}
\ls
1+\sum_{k\in\nn} 2^{-kn\beta} \left(\sum_{J\in\cd_k(Q)}\left(1-\max_{0\le i\le k-1}
\frac{|I_i^J|}{|J|}\cdot\frac{|\varphi^{-1}(J)|}{|\varphi^{-1}(I_i^J)|} \right)_+ \right).
\end{align}
To interpret this bound  geometrically,
we define the average density of $\varphi^{-1}$ on a dyadic cube $I$ by
$$
\rho(I) := \frac{|\varphi^{-1}(I)|}{|I|}.
$$
Then \eqref{eq-gauge-geom1} is equivalent to
\begin{align}\label{eq-gauge-geom2}
\|\varphi\|_{\mathscr G_\beta(Q)}
\ls
1+\sum_{k\in\nn} 2^{-kn\beta} \left(\sum_{J\in\cd_k(Q)}\left(1-\max_{0\le i\le k-1}
\frac{|\rho(J)|}{|\rho(I_i^J)|} \right)_+ \right).
\end{align}
Geometrically, the quantity
$$
\max_{0\le i\le k-1} \frac{\rho(J)}{\rho(I_i^J)}
$$
compares the local density of the preimage measure on $J$
with the densities on all its dyadic ancestors $I_i^J$.
The positive part $(1-\cdot)_+$ in  \eqref{eq-gauge-geom2} is nonzero exactly when
$\rho(J) < \rho(I_i^J)$ for every $i=0,\dots,k-1$, that is, when $J$ has strictly
lower density than all of its dyadic ancestors. This corresponds precisely to
\emph{stretching} (density decrease) of the map $\varphi$.
In contrast, if $\rho(J) \ge \rho(I_i^J)$ for some ancestor $I_i^J$, then the ratio exceeds $1$
and the positive part vanishes, so
 \emph{compression} (density increase) contributes nothing to the gauge.
Thus, although \eqref{eq-gauge-geom2} is only an upper bound,
it reveals that the gauge is sensitive to stretching regions; the bound is controlled by the cumulative
stretching defect of $\varphi$ across all dyadic scales, with each scale $k$ weighted by $2^{-kn\beta}$.

%\item
%If the summation in the upper bounds \eqref{eq-gauge-optimal} or \eqref{eq-gauge-geom1} at scale $k$
%were taken over all $J\in\mathcal D_k(Q)$, that is,
%\begin{align*}
%\sum_{J\in\mathcal D_k(Q)}
%\left(
%1-\max_{0\le i\le k-1}
%\frac{|I_i^J|}{|J|}\cdot\frac{|\varphi^{-1}(J)|}{|\varphi^{-1}(I_i^J)|}
%\right)_+
%=
%\sum_{J\in\mathcal D_k(Q)}
%\left(
%1-\max_{0\le i\le k-1}
%\frac{|I_i^J|}{|J|}\cdot\frac{|\varphi^{-1}(J)|}{|\varphi^{-1}(I_i^J)|}
%\right),
%\end{align*}
%then by retaining
%only the term $i=0$ in the maximum, we would obtain
%\begin{align*}
%\sum_{J\in\mathcal D_k(Q)}
%\left(
%1-\max_{0\le i\le k-1}
%\frac{|I_i^J|}{|J|}\cdot\frac{|\varphi^{-1}(J)|}{|\varphi^{-1}(I_i^J)|}
%\right)_+
%&\le
%\sum_{J\in\mathcal D_k(Q)}
%\left(
%1-
%2^{kn}\frac{|\varphi^{-1}(J)|}{|\varphi^{-1}(Q)|}
%\right)\\
%&=
%2^{kn}-2^{kn}
%\frac{|\varphi^{-1}\left(\cup_{J\in\mathcal D_k(Q)}J\right)|}
%{|\varphi^{-1}(Q)|}\\
%&=0,
%\end{align*}
%which forces
%$$
%|\varphi^{-1}(J)| = 2^{-kn}|\varphi^{-1}(Q)|
%\quad\, \text{for all }\, J\in\mathcal D_k(Q),
%$$
%meaning that the preimage measure is uniformly distributed at scale $k$,
%with no stretching or compression relative to the global average.
%In general, however, the summation in the upper bound \eqref{eq-gauge-geom1} at scale $k$
%cannot simply range over all $J\in\mathcal D_k(Q)$;
%stretching and compression may coexist among different cubes of the same scale.
\end{enumerate}
\end{remark}

\section{From finite capacity gauge to boundedness of composition operators}\label{sec4}

\subsection{Two mean oscillation
characterizations of $\mathcal Q_\alpha(\mathbb R^n)$}\label{ss4.1}

The following characterization of the space \(\mathcal Q_\alpha(\mathbb R^n)\), in terms of the square mean oscillation, is from  \cite[Theorem~5.5]{EssenJansonPengXiao2000IUMJ}.

\begin{proposition}[\cite{EssenJansonPengXiao2000IUMJ}]
\label{prop-EJPX}
Let \(\alpha \in \mathbb R\). Then there exists a positive constant \(C = C(n,\alpha)\) such that for any \(f \in L_{\mathrm{loc}}^2(\mathbb R^n)\),
\begin{align}
\label{eq-MeanOsi}
C^{-1} \| f \|_{\mathcal Q_\alpha(\mathbb R^n)}^2
\le \sup_{\textup{cubes } Q \subset \mathbb R^n}
\sum_{I \in \mathcal D(Q)}
\left( \frac{|I|}{|Q|} \right)^{1-\frac{2\alpha}{n}} a_I(f)
\le C \| f \|_{\mathcal Q_\alpha(\mathbb R^n)}^2,
\end{align}
where
\begin{align}
\label{eq-MeanOsi-aIf}
a_I(f) := \frac{1}{|I|} \int_I |f - f_I|^2 \, dx,
\end{align}
and \(f_I := \frac{1}{|I|} \int_I f(x) \, dx\) denotes the integral average mean of \(f\) over \(I\).
\end{proposition}

To prove the sufficiency part of Theorem \ref{thm-trace-gauge}, we shall refine Proposition \ref{prop-EJPX} and adapt it to the setting of quasisymmetric mappings. This adaptation will be carried out in Subsection \ref{ss4.3} (see Theorems \ref{thm2-varphi-EJPX} and \ref{thm1-varphi-EJPX} below). As a key ingredient in the proof of Theorem \ref{thm2-varphi-EJPX}, we will use the following equivalent characterization of \(\mathcal Q_\alpha(\mathbb R^n)\) spaces, established in \cite[Proposition~3.1]{KoskelaXiaoZhangZhou2017JEMS}.

\begin{proposition}[\cite{KoskelaXiaoZhangZhou2017JEMS}] \label{prop-Q-equiChar}
 Let $\alpha \in (0,\infty)$ and $q\in (0,2]$. For any $f\in L_\loc^2(\rn)$ and any ball $B = B(x_0, r) \subset \mathbb{R}^n$, set
\begin{align}\label{eq-PsifB}
\Psi_{\alpha,\,q}(f,\,B) := \sum_{k\geq 0} 2^{2k\alpha}
\fint_{B(x_0,\, r)}
\inf_{c\in\mathbb{R}}
\left(
\fint_{B(x, \, 2^{-k}r)} |f(z) - c|^q \, dz
\right)^{\frac 2 q}
dx
\end{align}
and
$$
\Phi_{\alpha}(f,\,B) := |B|^{\frac{2\alpha}{n}-1}
\int_{B}
\int_{B}
\frac{|f(x)-f(y)|^2}{|x-y|^{n+2\alpha}}
\,dy\,
dx.
$$
Then, there exists a positive constant $C=C(\alpha,n,q)$ such that, for any $f\in L_{\rm loc}^2(\rn)$ and any ball $B \subset \mathbb{R}^n$,
\begin{align*}
C^{-1} \Phi_\alpha\left(f,\, \frac 1 {16}B\right) \leq \Psi_{\alpha,\,q}(f, B) \leq C \Phi_\alpha\left(f,\, 16B\right)
\end{align*}
and
\begin{align}\label{eq2-KXZZ}
C^{-1}\| f \|_{\mathcal Q_\alpha(\mathbb{R}^n)}^2\le  \sup_{\textup{balls}\, B\subset\rn} \Psi_{\alpha,\,q}(f,\,B) \le C \| f \|_{\mathcal Q_\alpha(\mathbb{R}^n)}^2.
\end{align}

\end{proposition}

\begin{remark}
Indeed, \cite[Proposition~1.3]{KoskelaXiaoZhangZhou2017JEMS} stated Proposition \ref{prop-Q-equiChar} only for \(n\ge 2\) and \(\alpha\in(0,1)\), but their proof actually works for all \(n\in\mathbb N\) and \(\alpha\in(0,\infty)\). Thus, we state Proposition \ref{prop-Q-equiChar} in this full generality.
\end{remark}

\subsection{A technical estimate for quasisymmetric mappings}\label{ss4.2}

The following lemma shows that the image under $\varphi$ of an annulus centered at $\varphi^{-1}(c_I)$ is contained in a dyadic annulus centered at $c_I$, with inner and outer radii controlled by the side length of $I$. This lemma will be used in the proof of Theorem \ref{thm1-varphi-EJPX} below.

\begin{lemma}\label{lem-varphiQI}
Let \(\varphi :\ \mathbb{R}^n \to \mathbb{R}^n\) be \(\eta\)-quasisymmetric and \(\varphi^{-1}\) be \(\eta'\)-quasisymmetric,  where \(\eta'\) is given in terms of \(\eta\) as in \eqref{eq-eta'}.   Choose arbitrary constant \(C_0 \in (10^{10}n, \infty)\) and  integers \(N_1, N_2 \ge 10^{10}n\) satisfying
\begin{align}\label{eq-C0N1N2}
C_0 > \eta(2^{N_1+2}) \quad\text{and}\quad 2^{N_2}>n \eta'\left(16C_0\right).
\end{align}
For any cube \(Q \subset \mathbb{R}^n\), any \(k \in \mathbb{Z}_+\), and any \(I \in \mathcal{D}_k(Q)\), if
\begin{align}\label{eq1-AnnualsAI}
z\in \mathcal A(I):=\left\{z\in \rn:\, C_0 \operatorname{diam}(\varphi^{-1}(2I)) \le
|z - \varphi^{-1}(c_I)| \le 4 C_0 \operatorname{diam}(\varphi^{-1}(2I))\right\}
\end{align}
then
\begin{align}\label{eq2-Annuals-varphiAI}
2^{-k+N_1}\ell(Q) \le |\varphi(z) - c_I| < 2^{-k+N_2}\ell(Q),
\end{align}
where \(c_I\) denotes the center of  \(I\).
\end{lemma}

\begin{proof}
Since \(I \in \mathcal{D}_k(Q)\), we have \(\ell(I) = 2^{-k}\ell(Q)\). Observe that there exists some \(u \in 2I\) such that
\begin{align}\label{eq-u}
|\varphi^{-1}(u) - \varphi^{-1}(c_I)| > \frac{1}{4}\operatorname{diam}(\varphi^{-1}(2I)).
\end{align}
Indeed, if no such point existed, then for all \(u, v \in 2I\),
$$
|\varphi^{-1}(u) - \varphi^{-1}(v)|
\le |\varphi^{-1}(u) - \varphi^{-1}(c_I)| + |\varphi^{-1}(c_I) - \varphi^{-1}(v)|
< \frac{1}{2}\operatorname{diam}(\varphi^{-1}(2I)).
$$
Taking the supremum over all \(u, v \in 2I\) would yield
$$
\operatorname{diam}(\varphi^{-1}(2I)) < \frac{1}{2}\operatorname{diam}(\varphi^{-1}(2I)),
$$
which is a contradiction. Thus, there must exist some \(u \in 2I\) satisfying \eqref{eq-u}.

Now suppose \(z \in \mathcal A(I)\). For the point \(u \in 2I\) satisfying \eqref{eq-u}, the monotonicity of \(\eta'\) gives
$$
\frac{|\varphi(z) - c_I|}{|u - c_I|}
\le \eta'\left(
\frac{|z - \varphi^{-1}(c_I)|}{|\varphi^{-1}(u) - \varphi^{-1}(c_I)|}
\right)
\le \eta'(16C_0).
$$
Consequently,
$$
|\varphi(z) - c_I|
\le \eta'(16C_0)|u - c_I|
< n\,\eta'(16C_0)\ell(I)
= n\,\eta'(16C_0)2^{-k}\ell(Q)
< 2^{-k+N_2}\ell(Q),
$$
provided that \(2^{N_2} > n\,\eta'(16C_0)\). This establishes the upper bound in \eqref{eq2-Annuals-varphiAI}.

For the lower bound in \eqref{eq2-Annuals-varphiAI}, we choose a point \(w \in 2I\) such that
$
\frac{1}{4}\ell(I) \le |w - c_I| \le \frac{1}{3}\ell(I).
$
Clearly,
$$
|\varphi^{-1}(w) - \varphi^{-1}(c_I)|
\le \operatorname{diam}(\varphi^{-1}(2I)).
$$
Then, for any \(z \in \mathcal A(I)\), we have
$$
C_0 \le \frac{|z - \varphi^{-1}(c_I)|}{|\varphi^{-1}(w) - \varphi^{-1}(c_I)|}
\le \eta\left(
\frac{|\varphi(z) - c_I|}{|w - c_I|}
\right),
$$
which implies
$$
\eta^{-1}(C_0) \le \frac{|\varphi(z) - c_I|}{|w - c_I|}
\le \frac{4|\varphi(z) - c_I|}{\ell(I)}
= \frac{4|\varphi(z) - c_I|}{2^{-k}\ell(Q)}.
$$
Hence,
$$
|\varphi(z) - c_I| \ge \eta^{-1}(C_0) 2^{-k-2}\ell(Q)
> 2^{-k+N_1}\ell(Q),
$$
provided that \(\eta^{-1}(C_0) > 2^{N_1+2}\), that is, \(C_0 > \eta(2^{N_1+2})\). This proves the lower bound in \eqref{eq2-Annuals-varphiAI}.
\end{proof}

\subsection{Mean oscillation characterizations
of ${\mathcal Q}_\alpha(\rn)$ adapted to quasisymmetric mappings}\label{ss4.3}

Applying Proposition \ref{prop-Q-equiChar}, we establish the following theorem, which sharpens the lower estimate in \eqref{eq-MeanOsi} by adapting it to the quasisymmetric setting. In the special case \(\varphi=\mathrm{id}\), it reduces exactly to the lower bound in \eqref{eq-MeanOsi}.

\begin{theorem}
 \label{thm2-varphi-EJPX}
 Let \(\alpha \in (0,\infty)\) and \(\varphi : \mathbb R^n \to \mathbb R^n\) be $\eta$-quasisymmetric, with the additional assumption that \(J_{\varphi},\,J_{\varphi^{-1}} \in A_\infty(\mathbb R)\) when \(n = 1\).
 Then, for any $f\in L_\loc^2(\rn)$,
 \begin{align}\label{eq-lowerfQ}
 \|f\circ \varphi^{-1}\|_{\mathcal Q_\alpha(\rn)}^2
 \le C \sup_{\textup{cubes}\,Q\subset\rn}\sum_{I\in\cd(Q)} \left(\frac{|I|}{|Q|}\right)^{1-\frac{2\alpha}{n}}
 a_{\varphi^{-1}(2I)}(f),
 \end{align}
 where \(a_{\varphi^{-1}(2I)}(f)\) is defined as in \eqref{eq-MeanOsi-aIf}, and \(C=C(n,\alpha,\eta, [J_\varphi]_{A_\infty(\rn)}, [J_{\varphi^{-1}}]_{A_\infty(\rn)})\) is a positive constant independent of \(f\).

\end{theorem}

\begin{proof}
If \(n\ge 2\), then \(J_\varphi,\, J_{\varphi^{-1}}\in A_\infty(\mathbb R^n)\) follows from quasisymmetry via Proposition \ref{prop-varphi-weight}; if \(n=1\), it is assumed. Denote by \(r_\varphi\in(1,\infty)\) the reverse H\"older exponent of \(J_\varphi\) from \eqref{eq-RHn} and, set
$$
q := 2\left(1 - \frac{1}{r_\varphi}\right).
$$
Let \(\Psi_{\alpha,q}\) be  as in \eqref{eq-PsifB}.
Our aim is to prove that, for any \(f \in  L_\loc^2(\mathbb R^n)\) and any ball \(B=B(x_0, r)\),
\begin{align}\label{eq1-aim}
\Psi_{\alpha,\,q}(f\circ \varphi^{-1},\,B)
\ls  \sum_{I\in\cd(Q_B)} \left(\frac{|I|}{|Q_B|}\right)^{1-\frac{2\alpha}{n}}
 a_{\varphi^{-1}(2I)}(f),
\end{align}
where
\(Q_B := Q(x_0, 2r)\) is the cube centered at \(x_0\) and of side length \(2r\).
Having established \eqref{eq1-aim}, we apply \eqref{eq2-KXZZ} and immediately obtain \eqref{eq-lowerfQ}.
The proof of \eqref{eq1-aim} is divided into two steps.

\medskip

{\bf Step 1:\, A preliminary estimate for $\Psi_{\alpha,\,q}(f\circ \varphi^{-1},\,B)$.\,}
For any $c\in\rr$, $k\in\zz_+$ and $x\in B=B(x_0, r)$,   by a change of variables $z=\varphi(y)$ and the H\"older inequality, we obtain
\begin{align}\label{eq-PsiBf-1}
  &\left(\fint_{B(x, \, 2^{-k}r)} |f\circ \varphi^{-1}(z) - c|^q \, dz\right)^\frac 1q \\
  &\quad = \left(\frac1{|B(x, \, 2^{-k}r)|} \int_{\varphi^{-1}\left(B(x, \, 2^{-k}r)\right)} |f(y) - c|^q  J_\varphi(y)\, dy\right)^\frac 1q \notag\\
   &\quad  \le \left(\frac{|\varphi^{-1}(B(x, \, 2^{-k}r))|}{|B(x, \, 2^{-k}r)|}\right)^\frac 1q
   \left(\fint_{\varphi^{-1}(B(x, \, 2^{-k}r))} |f(y) - c|^{2} \, dy\right)^\frac 1{2}
   \left(\fint_{\varphi^{-1}(B(x, \, 2^{-k}r))}  J_\varphi(y)^{r_\varphi}\, dy\right)^\frac 1{q r_\varphi}.\notag
\end{align}

By \eqref{eq-ball-varphi}, there exists a constant \(c_\eta \in(0,\infty)\), independent of \(x\), \(k\) and \(r\), such that
\begin{align}\label{eq-ball-Q}
B\left(x_{k,r},\, c_\eta R_{k,r}\right)
\subset \varphi^{-1}\bigl(B(x,\, 2^{-k}r)\bigr)
\subset B\left(x_{k,r},\,  R_{k,r}\right),
\end{align}
where
$
x_{k,r} := \varphi^{-1}(x)$ and  $R_{k,r} := \operatorname{diam}\varphi^{-1}(B(x,\, 2^{-k}r)).$
From these inclusions and the reverse H\"older inequality in \eqref{eq-RHn}, we deduce that
 \begin{align*}
  \left(\fint_{\varphi^{-1}\left(B(x, \, 2^{-k}r)\right)}  J_\varphi(y)^{r_\varphi}\, dy\right)^\frac 1{r_\varphi}
  &\le \left((c_\eta^{-1})^n\fint_{B(x_{k,r},\, R_{k,r})}  J_\varphi(y)^{r_\varphi}\, dy\right)^\frac 1{r_\varphi}\\
  &\le  c_\eta^{-\frac n {r_\varphi}} [J_\varphi]_{\mathrm{RH}_{r_\varphi}}\fint_{B(x_{k,r},\,  R_{k,r})}  J_\varphi(y)\, dy.
\end{align*}
By \(J_{\varphi}\in A_\infty(\mathbb R^n)\) and the volume doubling property in Proposition \ref{prop-varphi-weight}(iv), together with \eqref{eq-ball-Q}, we obtain
$$
\int_{B(x_{k,r},\,  R_{k,r})} J_\varphi(y)\, dy
=
\left|\varphi\left(B(x_{k,r},\,  R_{k,r})\right)\right|
\lesssim
\left|\varphi\left(B(x_{k,r},\, c_\eta R_{k,r})\right)\right|
\lesssim
\left|\varphi\left(\varphi^{-1}\bigl(B(x,\, 2^{-k}r)\bigr)\right)\right|
\simeq
|B(x,\, 2^{-k}r)|.
$$
Moreover, by \eqref{eq3-ball-varphi} and the definition of \(R_{k,r}\), we have
$$
|B(x_{k,r},\,  R_{k,r})| \simeq R_{k,r}^n \simeq
\left(\operatorname{diam}\varphi^{-1}(B(x,\, 2^{-k}r))\right)^n \simeq
|\varphi^{-1}(B(x,\, 2^{-k}r))|.
$$
Combining the preceding three estimates gives
\begin{align}\label{eq-RH-PsiBf}
\left(\fint_{\varphi^{-1}\left(B(x, \, 2^{-k}r)\right)}  J_\varphi(y)^{r_\varphi}\, dy\right)^\frac 1{r_\varphi} \ls \frac{|B(x, \, 2^{-k}r)|}{|\varphi^{-1}(B(x, \, 2^{-k}r))|}.
\end{align}

Next, substituting \eqref{eq-RH-PsiBf} into \eqref{eq-PsiBf-1} yields
\begin{align*}
 \left(\fint_{B(x, \, 2^{-k}r)} |f\circ \varphi^{-1}(z) - c|^q \, dz\right)^\frac 1q \ls
   \left(\fint_{\varphi^{-1}(B(x, \, 2^{-k}r))} |f(y) - c|^{2} \, dy\right)^\frac 1{2}.
\end{align*}
This immediately implies
\begin{align}\label{eq-PsiBf-2}
  \Psi_{\alpha,\,q}( f\circ \varphi^{-1},\, B) & \ls \sum_{k\ge 0}\fint_{B(x_0, r)}
  2^{2k\alpha}\left(\inf_{c\in\rr}\fint_{\varphi^{-1}\left(B(x, \, 2^{-k}r)\right)} |f(y) - c|^{2} \, dy\right)\,dx.
\end{align}
The next step is to replace the Euclidean balls \(B(x_0, r)\) and \(B(x, 2^{-k}r)\) by their cube counterparts.

\medskip

{\bf Step 2:\, Replacing balls by cubes.}
Fix $k\in\zz_+$. Note that \(B=B(x_0, r)\) is contained in the cube \(Q_B\), which is centered at \(x_0\) and has side length \(2r\). For each \(k \in \mathbb{Z}_+\), we consider the dyadic decomposition
$$
Q_B=\bigcup_{I\in\cd_k(Q_B)} I.
$$
Note that each  subcube \(I\) has side length \(\ell(I) = 2^{-k}\ell(Q_B) = 2^{-k+1}r\).
Let \(\nu_n\) denote the Lebesgue measure of the unit ball in \(\mathbb{R}^n\). Then,
\begin{align}\label{eq-PsiBf-3}
  &\fint_{B(x_0, r)}
  2^{2k\alpha}\left(\inf_{c\in\rr}\fint_{\varphi^{-1}\left(B(x, \, 2^{-k}r)\right)} |f(y) - c|^{2} \, dy\right)\,dx\\
  &\quad\le \sum_{I\in\cd_k(Q_B)}\frac1{\nu_n r^n} \int_{I}
  2^{2k\alpha}\left(\inf_{c\in\rr}\fint_{\varphi^{-1}\left(B(x, \, 2^{-k}r)\right)} |f(y) - c|^{2} \, dy\right)\,dx\notag\\
  &\quad = \frac{2^n}{\nu_n} \sum_{I\in\cd_k(Q_B)} \left(\frac{|I|}{|Q_B|}\right)^{1-\frac{2\alpha}{n}} \fint_{I}
\left(\inf_{c\in\rr}\fint_{\varphi^{-1}\left(B(x, \, 2^{-k}r)\right)} |f(y) - c|^{2} \, dy\right)\,dx.\notag
\end{align}

If $x \in I$ for some $I \in \mathcal{D}_k(Q_B)$, then using the fact that $\ell(I) = 2^{-k}\ell(Q_B) = 2^{-k+1}r$, we claim that
\begin{align}\label{eq-inclu}
B(x, 2^{-k}r)\subset 2I\subset B(x, \operatorname{diam}(2I))\subset
B(x, 2^{-k+2}nr).
\end{align}
Indeed, let $c_I = (c_1, \dots, c_n)$ denote  the center of $I$, and write $x = (x_1, \dots, x_n)$. For any $z = (z_1, \dots, z_n) \in B(x, 2^{-k}r)$ and each $j \in \{1, \dots, n\}$, we have
$$
|z_j-c_j|\le |z_j-x_j|+|x_j-c_j| \le |z-x|+|x_j-c_j| < 2^{-k}r+2^{-1}\ell(I)=2^{-k+1}r=\ell(I),
$$
which shows that $z\in 2I$. This proves the first inclusion in \eqref{eq-inclu}; the remaining inclusions are immediate.

By \(J_{\varphi^{-1}} \in A_\infty(\mathbb R^n)\), together with the doubling property in Proposition \ref{prop-varphi-weight}(iv) and  \eqref{eq-inclu}, we obtain
\begin{align}\label{eq-inclu-mea}
|\varphi^{-1}(2I)|
&\le \left|\varphi^{-1}\left(B(x,\, 2^{-k+2}nr)\right)\right|
\ls \left|\varphi^{-1}\left(B(x, \, 2^{-k}r)\right)\right|.
\end{align}
Combining \eqref{eq-inclu} with \eqref{eq-inclu-mea} then gives
\begin{align}\label{eq-PsiBf-4}
\inf_{c\in\rr}\fint_{\varphi^{-1}\left(B(x, \, 2^{-k}r)\right)} |f(y) - c|^{2} \, dy
&\ls \inf_{c\in\rr}\fint_{\varphi^{-1}(2I)} |f(y) - c|^{2} \, dy\\
&
\ls \fint_{\varphi^{-1}(2I)} |f(y) - f_{\varphi^{-1}(2I)}|^{2} \, dy. \notag
\end{align}

Now, inserting \eqref{eq-PsiBf-4} into \eqref{eq-PsiBf-3} yields
\begin{align*}
  &\fint_{B(x_0, r)}
  2^{2k\alpha}\left(\inf_{c\in\rr}\fint_{\varphi^{-1}\left(B(x, \, 2^{-k}r)\right)} |f(y) - c|^{2} \, dy\right)\,dx\\
  &\quad\ls \sum_{I\in\cd_k(Q_B)} \left(\frac{|I|}{|Q_B|}\right)^{1-\frac{2\alpha}{n}}
\left(\fint_{\varphi^{-1}(2I)} |f(y) - f_{\varphi^{-1}(2I)}|^{2} \, dy\right).
\end{align*}
Combining this with \eqref{eq-PsiBf-2}, we obtain
\begin{align*}
  \Psi_{\alpha,\,q}( f\circ \varphi^{-1},\, B)
  & \ls \sum_{k=0}^\infty\sum_{I\in\cd_k(Q_B)} \left(\frac{|I|}{|Q_B|}\right)^{1-\frac{2\alpha}{n}}
\left(\fint_{\varphi^{-1}(2I)} |f(y) - f_{\varphi^{-1}(2I)}|^{2} \, dy\right)\\
&\simeq \sum_{I\in\cd(Q_B)} \left(\frac{|I|}{|Q_B|}\right)^{1-\frac{2\alpha}{n}}
a_{\varphi^{-1}(2I)}(f).\notag
\end{align*}
This establishes the desired estimate in \eqref{eq1-aim}.
Consequently, \eqref{eq-lowerfQ} holds.
\end{proof}

In contrast to Theorem \ref{thm2-varphi-EJPX}, the following theorem provides a \(\varphi\)-adapted version of the upper estimate in \eqref{eq-MeanOsi}. As before, it reduces to the upper bound in \eqref{eq-MeanOsi} when \(\varphi=\mathrm{id}\). Its proof relies on Lemma \ref{lem-varphiQI}.

\begin{theorem}
 \label{thm1-varphi-EJPX}
 Let \(\alpha \in (0,\infty)\) and \(\varphi :\ \mathbb R^n \to \mathbb R^n\) be $\eta$-quasisymmetric, with the additional assumption that \(J_{\varphi^{-1}} \in A_\infty(\mathbb R)\) when \(n = 1\).
 Then, for any $f\in\mathcal Q_\alpha(\rn)$,
 \begin{align}\label{eq-upperfQ}
 \sup_{\textup{cubes}\,Q\subset\rn}\sum_{I\in\cd(Q)} \left(\frac{|\varphi^{-1}(I)|}{|\varphi^{-1}(Q)|}\right)^{1-\frac{2\alpha}{n}}
 a_{\varphi^{-1}(2I)}(f)\le C \|f\|_{\mathcal Q_\alpha(\rn)}^2,
 \end{align}
 where \(a_{\varphi^{-1}(2I)}(f)\) is defined as in \eqref{eq-MeanOsi-aIf}, and \(C=C(n,\alpha,\eta, [J_{\varphi^{-1}}]_{A_\infty(\rn)})\) is a positive constant independent of \(f\).
\end{theorem}

\begin{proof}
Fix large constants $C_0, N_1, N_2$ satisfying \eqref{eq-C0N1N2}. For any cube $Q\subset\mathbb R^n$ and $I\in\mathcal D_k(Q)$ with $k\in\mathbb Z_+$, define $\mathcal A(I)$ as in \eqref{eq1-AnnualsAI}. Then every $z\in \mathcal A(I)$ satisfies  \eqref{eq2-Annuals-varphiAI}. Further, applying \eqref{eq3-ball-varphi} and \eqref{eq-diamEF} yields
\begin{align}\label{eq-Annuals-varphiAI-meas}
|\mathcal A(I)|=\nu_n C_0^n (4^n-1) \bigl( \operatorname{diam}(\varphi^{-1}(2I))\bigr)^n \simeq |\varphi^{-1}(2I)|.
\end{align}
Using the elementary inequality
$$
(a+b)^2 \le 2(a^2+b^2) \quad \text{for all }\, a,\, b\in(0,\infty),
$$
we write
\begin{align*}
  a_{\varphi^{-1}(2I)}(f)
  &= \fint_{\varphi^{-1}(2I)} |f-f_{\varphi^{-1}(2I)}|^2\,dx\\
  &\le 2 \fint_{\varphi^{-1}(2I)} |f-f_{\mathcal A(I)}|^2\,dx
  + 2|f_{\mathcal A(I)}-f_{\varphi^{-1}(2I)}|^2.
\end{align*}
Note that the H\"older inequality implies
\begin{align*}
|f_{\mathcal A(I)}-f_{\varphi^{-1}(2I)}|^2
&\le \left( \fint_{\varphi^{-1}(2I)} |f-f_{\mathcal A(I)}|\,dx\right)^2
\le  \fint_{\varphi^{-1}(2I)} |f-f_{\mathcal A(I)}|^2\,dx
\end{align*}
and, hence,
\begin{align*}
  a_{\varphi^{-1}(2I)}(f)
  &\le 4 \fint_{\varphi^{-1}(2I)} |f-f_{\mathcal A(I)}|^2\,dx\\
  &\le 4  \fint_{\varphi^{-1}(2I)} \left|\fint_{\mathcal A(I)} (f(x)-f(y))\,dy\right|^2\,dx\notag\\
  &\le 4 \fint_{\varphi^{-1}(2I)} \fint_{\mathcal A(I)}\left| f(x)-f(y)\right|^2\,dy\,dx. \notag
\end{align*}
For any $x\in \varphi^{-1}(2I)$ and $y\in \mathcal A(I)$, it follows from \eqref{eq1-AnnualsAI}  that
$$
|x-y|\ge |y-\varphi^{-1}(c_I)|-|x-\varphi^{-1}(c_I)| \ge (C_0-1) \operatorname{diam}(\varphi^{-1}(2I)).
$$
Combining this with the measure estimate \eqref{eq-Annuals-varphiAI-meas}, we obtain
\begin{align*}
\fint_{\varphi^{-1}(2I)} \fint_{\mathcal A(I)}\left| f(x)-f(y)\right|^2\,dy\,dx
&\simeq \frac{\operatorname{diam}(\varphi^{-1}(2I))^{n+2\alpha}}{|\varphi^{-1}(2I)|^2}
  \int_{\varphi^{-1}(2I)} \int_{\mathcal A(I)}\frac{| f(x)-f(y)|^2}{|x-y|^{n+2\alpha}}\,dy\,dx\\
  &\simeq |\varphi^{-1}(2I)|^{\frac{2\alpha}{n}-1}
  \int_{\varphi^{-1}(2I)} \int_{\mathcal A(I)}\frac{| f(x)-f(y)|^2}{|x-y|^{n+2\alpha}}\,dy\,dx.
\end{align*}
Therefore, we arrive at the estimate
\begin{align}\label{eq-afAI}
  a_{\varphi^{-1}(2I)}(f)
  &\ls |\varphi^{-1}(2I)|^{\frac{2\alpha}{n}-1}
  \int_{\varphi^{-1}(2I)} \int_{\mathcal A(I)}\frac{| f(x)-f(y)|^2}{|x-y|^{n+2\alpha}}\,dy\,dx.
\end{align}
For any $y\in \mathcal A(I)$, by \eqref{eq2-Annuals-varphiAI}, we have
$$\varphi(y)\in B\left(c_I,\, 2^{N_2}\ell(I)\right)\subset 2^{N_2+2}I\subset 2^{N_2+2}Q,$$
which implies
$$\mathcal A(I)\subset\varphi^{-1}\left(2^{N_2+2}Q\right).$$
Now, we substitute \eqref{eq-afAI} into the left-hand side of \eqref{eq-upperfQ} and apply the volume doubling property of the Jacobian $J_{\varphi^{-1}}$ (see \eqref{eq-vd}). This gives
 \begin{align}\label{eq-afAI-sum}
& \sum_{I\in\cd(Q)} \left(\frac{|\varphi^{-1}(I)|}{|\varphi^{-1}(Q)|}\right)^{1-\frac{2\alpha}{n}}
 a_{\varphi^{-1}(2I)}(f)\\
 &\quad \ls
 \sum_{I\in\cd(Q)} |\varphi^{-1}(2Q)|^{\frac{2\alpha}{n}-1}
  \int_{\varphi^{-1}(2I)} \int_{\mathcal A(I)}\frac{| f(x)-f(y)|^2}{|x-y|^{n+2\alpha}}\,dy\,dx\notag\\
  &\quad\simeq
   |\varphi^{-1}(2Q)|^{\frac{2\alpha}{n}-1}
   \int_{\varphi^{-1}(2Q)} \int_{\varphi^{-1}(2^{N_2+2}Q)}
  \left(\sum_{I\in\cd(Q)}
  {\mathbf 1}_{\varphi^{-1}(2I)}(x) {\mathbf 1}_{\mathcal A(I)}(y)\right)\frac{| f(x)-f(y)|^2}{|x-y|^{n+2\alpha}}\,dy\,dx.\notag
 \end{align}
   For any $x\in \varphi^{-1}(2Q)$ and $y\in\varphi^{-1}(2^{N_2+2}Q)$, we introduce the kernel
$$K_Q(x,y):=  \sum_{I\in\cd(Q)}
  {\mathbf 1}_{\varphi^{-1}(2I)}(x) {\mathbf 1}_{\mathcal A(I)}(y)
  =\sum_{k\ge 0}\sum_{I\in\cd_k(Q)}
  {\mathbf 1}_{\varphi^{-1}(2I)}(x) {\mathbf 1}_{\mathcal A(I)}(y).$$
  Then, we claim the uniform bound
\begin{align}\label{eq-claimKQ}
K_Q(x,y) \lesssim 1,
\end{align}
with implicit constant independent of $x,y$, and $Q$.

To prove the claim \eqref{eq-claimKQ}, we suppose that $k\in\zz_+$ and $I\in\cd_k(Q)$ such that  $\varphi^{-1}(2I)\ni x$ and $\mathcal A(I )\ni y$. Then, by the definition of $\mathcal A(I )$ in \eqref{eq1-AnnualsAI} and the triangle inequality, we have
$$
(C_0-1) \operatorname{diam}(\varphi^{-1}(2I ))
\le |x-y|\le (4C_0+1) \operatorname{diam}(\varphi^{-1}(2I )).
$$
Combining this with \eqref{eq-diamEF}, the condition $C_0\ge 4$, and the monotonic increasing property of $\eta$, we obtain the two-sided estimate
$$
\frac{1}{2\eta(8C_0+2)}
\le
\frac{\operatorname{diam}(2I )}{\operatorname{diam}(\varphi(B(x, |y-x|)))}
\le\eta(1).
$$
Indeed, the upper bound follows from
$$
\frac{\operatorname{diam}(2I )}{\operatorname{diam}(\varphi(B(x, |y-x|)))}
= \frac{\operatorname{diam}\bigl(\varphi(\varphi^{-1}(2I ))\bigr)}{\operatorname{diam}(\varphi(B(x, |y-x|)))}
\le \eta\!\left(
\frac{2\operatorname{diam}(\varphi^{-1}(2I ))}{2|y-x|}
\right)
\le \eta(1),
$$
while the lower bound is a consequence of
$$
\frac{\operatorname{diam}(2I )}{\operatorname{diam}(\varphi(B(x, |y-x|)))}
= \frac{\operatorname{diam}\bigl(\varphi(\varphi^{-1}(2I ))\bigr)}{\operatorname{diam}(\varphi(B(x, |y-x|)))}
\ge \frac{1}{2\eta\!\left(
\frac{2|y-x|}{\operatorname{diam}(\varphi^{-1}(2I ))}
\right)}
\ge \frac{1}{2\eta(8C_0+2)}.
$$
Set $R_{x,y}:=\operatorname{diam}(\varphi(B(x, |y-x|)))$. Since
$$
\operatorname{diam}(2I )=2\sqrt n\,\ell(I )=2^{-k+1}\sqrt n\,\ell(Q),
$$
it follows that the index $k$ satisfying $\mathcal A(I )\ni y$ must satisfy
$$
\frac{R_{x,y}}{2\eta(8C_0+2)}\le 2^{-k+1}\sqrt n\,\ell(Q)\le \eta(1)R_{x,y}.
$$
Clearly, the number of integers \(k\) satisfying this two-sided inequality is at most
$$
\log_2\bigl(2\eta(1)\eta(8C_0+2)\bigr)+1.
$$
Moreover, for each \(k\in\mathbb Z_+\), we have
\begin{align*}
\#\left\{I\in\mathcal D_k(Q):\, \varphi^{-1}(2I)\ni x\,\right\}
= \#\left\{I\in\mathcal D_k(Q):\, 2I\ni \varphi(x)\,\right\} \le 2^n,
\end{align*}
since the cubes \(I\in\mathcal D_k(Q)\) partition \(Q\), and any point in \(\mathbb R^n\) lies in at most \(2^n\) of the dilated cubes \(2I\) with \(I\in\mathcal D_k(Q)\).
From these facts, it follows that
\begin{align*}
K_Q(x,y)
     \le 2^n\left( \log_2\bigl(2\eta(1)\eta(8C_0+2)\bigr)+1\right).
\end{align*}
This establishes the claim \eqref{eq-claimKQ}.

Now, with \eqref{eq-claimKQ} at hand, we return to \eqref{eq-afAI-sum} and obtain
\begin{align}\label{eq1-afAI-sum}
& \sum_{I\in\mathcal D(Q)} \left(\frac{|\varphi^{-1}(I)|}{|\varphi^{-1}(Q)|}\right)^{1-\frac{2\alpha}{n}}
 a_{\varphi^{-1}(2I)}(f)\\
 &\quad \lesssim
   |\varphi^{-1}(2Q)|^{\frac{2\alpha}{n}-1}
   \int_{\varphi^{-1}(2Q)} \int_{\varphi^{-1}(2^{N_2+2}Q)}
 \frac{|f(x)-f(y)|^2}{|x-y|^{n+2\alpha}}\,dy\,dx.\notag
 \end{align}
For each cube \(Q\), we define the ball
$$
B_{\varphi^{-1}}(Q):=B\Bigl(\varphi^{-1}(c_Q),\, \operatorname{diam}\left(\varphi^{-1}(2^{N_2+2}Q)\right)\Bigr),
$$
where $c_Q$ denotes the center of $Q$. Since \(N_2\) is chosen sufficiently large (see \eqref{eq-C0N1N2}), we have \(\varphi^{-1}(2Q)\subset \varphi^{-1}(2^{N_2+2}Q)\). Consequently, both \(\varphi^{-1}(2Q)\) and \(\varphi^{-1}(2^{N_2+2}Q)\) are contained in the ball \(B_{\varphi^{-1}}(Q)\). Moreover, by \eqref{eq-ball-varphi}--\eqref{eq-cube-varphi}--\eqref{eq3-ball-varphi} and the volume doubling property of \(J_{\varphi^{-1}}\) (see Proposition \ref{prop-varphi-weight}), we have
$$
|\varphi^{-1}(2Q)| \simeq |B_{\varphi^{-1}}(Q)|.
$$
Therefore, we obtain
\begin{align}\label{eq2-afAI-sum}
& |\varphi^{-1}(2Q)|^{\frac{2\alpha}{n}-1}
   \int_{\varphi^{-1}(2Q)} \int_{\varphi^{-1}(2^{N_2+2}Q)}
 \frac{|f(x)-f(y)|^2}{|x-y|^{n+2\alpha}}\,dy\,dx\\
 &\quad
 \lesssim
 |B_{\varphi^{-1}}(Q)|^{\frac{2\alpha}{n}-1}
   \int_{B_{\varphi^{-1}}(Q)} \int_{B_{\varphi^{-1}}(Q)}
 \frac{|f(x)-f(y)|^2}{|x-y|^{n+2\alpha}}\,dy\,dx\notag \\
 &\quad \lesssim \|f\|_{\mathcal Q_\alpha(\mathbb R^n)}^2. \notag
 \end{align}
Finally, combining \eqref{eq1-afAI-sum} and \eqref{eq2-afAI-sum} yields that for every cube $Q\subset\mathbb R^n$,
$$
\sum_{I\in\mathcal D(Q)} \left(\frac{|\varphi^{-1}(I)|}{|\varphi^{-1}(Q)|}\right)^{1-\frac{2\alpha}{n}}
 a_{\varphi^{-1}(2I)}(f)
 \lesssim \|f\|_{\mathcal Q_\alpha(\mathbb R^n)}^2,
$$
which proves the desired estimate in \eqref{eq-upperfQ}.
\end{proof}

\subsection{Proof of the sufficiency part of Theorem \ref{thm-trace-gauge}}\label{ss4.4}

The idea is to apply  Theorems \ref{thm2-varphi-EJPX} and \ref{thm1-varphi-EJPX} together with the definition of $\mathscr G_{1-\frac{2\alpha}{n}}$ given in Definition \ref{def-gauge}.

\begin{proof}[Proof of Theorem \ref{thm-trace-gauge}: the  sufficiency part]

Let $f\in\mathcal Q_\alpha(\rn)$.
Applying \eqref{eq-lowerfQ} in Theorem \ref{thm2-varphi-EJPX} yields
\begin{align}\label{eq1-tracebdd}
 \|f\circ \varphi^{-1}\|_{\mathcal Q_\alpha(\rn)}^2
 \ls \sup_{\text{cubes}\,Q\subset\rn}\sum_{I\in\cd(Q)} \left(\frac{|I|}{|Q|}\right)^{1-\frac{2\alpha}{n}}
 a_{\varphi^{-1}(2I)}(f),
 \end{align}
 where \(a_{\varphi^{-1}(2I)}(f)\) is defined as in \eqref{eq-MeanOsi-aIf}.

For any cube $Q\subset\rn$, take an arbitrary sequence \(\{\lambda_{I'}\}_{I'\in\mathcal D(Q)}\in \mathscr T_\beta(Q)\). Then, for any \(I\in\mathcal D(Q)\),
$$
\sum_{\gfz{I'\in\mathcal D(Q)}{I'\supset I}} \lambda_{I'} \left(
\frac{|I'|}{|I|}\cdot\frac{|\varphi^{-1}(I)|}{|\varphi^{-1}(I')|}\right)^{1-\frac{2\alpha}{n}}
\ge 1.
$$
Consequently, we have
\begin{align*}
&\sum_{I\in\cd(Q)} \left(\frac{|I|}{|Q|}\right)^{1-\frac{2\alpha}{n}}
a_{\varphi^{-1}(2I)}(f)\\
&\quad\le \sum_{I\in\cd(Q)}  \left(
\sum_{\gfz{I'\in \cd(Q)}{I'\supset I}} \lambda_{I'} \left(
\frac{|I'|}{|I|}\cdot\frac{|\varphi^{-1}(I)|}{|\varphi^{-1}(I')|}\right)^{1-\frac{2\alpha}{n}}\right)
\left(\frac{|I|}{|Q|}\right)^{1-\frac{2\alpha}{n}}
a_{\varphi^{-1}(2I)}(f)\\
&\quad = \sum_{I'\in \cd(Q)}
 \lambda_{I'}  \left(\frac{|I'|}{|Q|}\right)^{1-\frac{2\alpha}{n}}
 \left(\sum_{I\in\cd(I')}
\left(\frac{|\varphi^{-1}(I)|}{|\varphi^{-1}(I')|}\right)^{1-\frac{2\alpha}{n}}
a_{\varphi^{-1}(2I)}(f)\right),
\end{align*}
where we used the fact that, for \(I\subset I'\in\mathcal D(Q)\), the condition \(I\in\mathcal D(Q)\) is equivalent to \(I\in\mathcal D(I')\).
Applying \eqref{eq-upperfQ} from Theorem \ref{thm1-varphi-EJPX} gives the uniform estimate
\begin{align*}
 \sum_{I\in\cd(I')}
\left(\frac{|\varphi^{-1}(I)|}{|\varphi^{-1}(I')|}\right)^{1-\frac{2\alpha}{n}}
a_{\varphi^{-1}(2I)}(f)\ls \|f\|_{\mathcal Q_\alpha(\rn)}^2
 \end{align*}
with implicit constant independent of \(I'\). Substituting this into the preceding inequality, we deduce that
 \begin{align*}
\sum_{I\in\cd(Q)} \left(\frac{|I|}{|Q|}\right)^{1-\frac{2\alpha}{n}}
a_{\varphi^{-1}(2I)}(f)
\ls \|f\|_{\mathcal Q_\alpha(\rn)}^2 \sum_{I'\in \cd(Q)}
 \lambda_{I'}  \left(\frac{|I'|}{|Q|}\right)^{1-\frac{2\alpha}{n}}.
\end{align*}
Taking the infimum over all sequences \(\{\lambda_{I'}\}\) in \(\mathscr T_\beta(Q)\), we obtain
\begin{align*}
\sum_{I\in\cd(Q)} \left(\frac{|I|}{|Q|}\right)^{1-\frac{2\alpha}{n}}
a_{\varphi^{-1}(2I)}(f)
\lesssim \|f\|_{\mathcal Q_\alpha(\mathbb R^n)}^2 \|\varphi\|_{\mathscr G_{1-\frac{2\alpha}{n}}(Q)}.
\end{align*}
Further, taking the supremum over all cubes \(Q\subset\mathbb R^n\) on both sides and using \eqref{eq1-tracebdd} then gives
\begin{align}\label{eq-operNorm<gauge}
\|f\circ \varphi^{-1}\|_{\mathcal Q_\alpha(\rn)}^2
 \lesssim \|f\|_{\mathcal Q_\alpha(\mathbb R^n)}^2 \sup_{\text{cubes}\, Q\subset\rn}\|\varphi\|_{\mathscr G_{1-\frac{2\alpha}{n}}(Q)}
 \simeq \|f\|_{\mathcal Q_\alpha(\mathbb R^n)}^2 \|\varphi\|_{\mathscr G_{1-\frac{2\alpha}{n}}}.
\end{align}
Thus, \(\mathcal C_\varphi\) is bounded on \(\mathcal Q_\alpha(\rn)\), with operator norm bounded by a constant multiple of \(\|\varphi\|_{\mathscr G_{1-\frac{2\alpha}{n}}}^{1/2}\).
\end{proof}

\section{From boundedness of composition operators to finite capacity gauge}\label{sec5}

This section is devoted to the proof of the necessity part of Theorem \ref{thm-trace-gauge}.
The proof relies on constructing an explicit sequence $\{f_N\}_{N\in\nn}$ of functions in \(\mathcal Q_\alpha(\rn)\)
(see Theorem \ref{thm-upperf} below)
via the wavelet characterization (Proposition \ref{prop-waveletsQ}),
with the core difficulty of estimating $\|f_N\circ\varphi\|_{\mathcal Q_\alpha(\rn)}$.

\subsection{Three technical lemmas}\label{ss5.1}

The following lemma generalizes Lemma \ref{lem-sum-cube} from the identity map to arbitrary quasisymmetric mappings.

\begin{lemma}\label{lem-QS-est1}
  Let $p\in(1,\infty)$ and \(\varphi :\ \mathbb R^n \to \mathbb R^n\) be
  quasisymmetric, with the additional assumption that
  \(J_{\varphi} \in A_\infty(\mathbb R)\) when \(n = 1\).
  Then, there exists a positive constant $C=C(n, p,\, [J_{\varphi}]_{ A_\infty(\mathbb R^n)})$ such that for any cube $Q\subset\rn$,
  \begin{align}\label{eq-sum-varphiI}
  \sum_{I\in\cd(Q)} \left(\frac{|\varphi(I)|}{|\varphi(Q)|}\right)^p
  \le C.
  \end{align}
\end{lemma}

\begin{proof}
Let $r_{\varphi}$ be the reverse H\"older exponent of $J_{\varphi}$.
Since $I\subset Q$, we have $|\varphi(I)|\le |\varphi(Q)|$.
If $p>r_\varphi$, then
$$
\sum_{I\in\mathcal D(Q)} \left(\frac{|\varphi(I)|}{|\varphi(Q)|}\right)^p
\le
\sum_{I\in\mathcal D(Q)} \left(\frac{|\varphi(I)|}{|\varphi(Q)|}\right)^{r_\varphi}.
$$

Thus it suffices to prove the estimate \eqref{eq-sum-varphiI}  for $1<p\le r_{\varphi}$.
By the H\"older inequality, we have
$$
|\varphi(I)|^p
=\left(\int_I  J_{\varphi}(x)\,dx\right)^p
\le |I|^{p-1}\,\left(\int_I  J_{\varphi}(x)^p\,dx\right).
$$
From this, it follows that
    \begin{align*}
 \sum_{I\in\cd(Q)}  |\varphi(I)|^p
 &=  \sum_{k=0}^\infty \sum_{I\in\cd_k(Q)}  |\varphi(I)|^p\\
&\le \sum_{k=0}^\infty     \sum_{I\in\cd_k(Q)}  |I|^{p-1}\,
\left(\int_I  J_{\varphi}(x)^p\,dx\right)\\
&= \sum_{k=0}^\infty  2^{-kn(p-1)} |Q|^{p-1}
 \left(\sum_{I\in\cd_k(Q)} \int_I  J_{\varphi}(x)^p\,dx\right)\\
 &= \sum_{k=0}^\infty  2^{-kn({p-1})} |Q|^{p-1}
 \left(\int_Q  J_{\varphi}(x)^{p}\,dx\right)\\
 &=\frac1{1-2^{-n(p-1)}}|Q|^{p}
 \left(\fint_Q  J_{\varphi}(x)^p\,dx\right).
    \end{align*}
Applying the reverse H\"older inequality \eqref{eq-RHn}  and the fact $p\le r_{\varphi}$, we obtain
\begin{align*}
|Q|^{p}
 \left(\fint_Q  J_{\varphi}(x)^p\,dx\right)
 &\le |Q|^{p}
 \left(\fint_Q  J_{\varphi}(x)^{r_\varphi}\,dx\right)^{\frac p{r_\varphi}}\\
 &
 \le C |Q|^p
 \left(\fint_Q  J_{\varphi}(x)\,dx\right)^{p}\\
 &
 =C |\varphi(Q)|^p.
\end{align*}
Then, the desired estimate in \eqref{eq-sum-varphiI} is an immediate consequence of the last two formulae.
\end{proof}

The following lemma precisely quantifies the spatial localization of the pull-back geometry under the quasisymmetric mapping.

\begin{lemma}\label{lem-QS-est0}
Let \(\varphi :\ \mathbb R^n \to \mathbb R^n\) be $\eta$-quasisymmetric.
Suppose that $\tau\in(1,\infty)$ and $\varepsilon\in(0,\infty)$. Then, there exists a positive constant $C=C(n,\tau, \varepsilon, \eta, [J_{\varphi}]_{ A_\infty(\mathbb R^n)})$ such that for any cubes $Q\subset\rn$ and $J\subset\rn$,
\begin{align}\label{eq-upward-sum}
\sum_{\gfz{I\in\mathcal D(Q)}{E_\tau(I)\supset J}}
\left(\frac{|\varphi(J)|}{|I|}\right)^{\varepsilon}
\le C,
\end{align}
where
\begin{align}\label{eq-EtauI}
E_\tau(I):=B\Bigl(\varphi^{-1}(c_I),\ \tau\operatorname{diam}\varphi^{-1}(I)\Bigr),
\end{align}
with $c_I$ being the center of the cube $I$.
\end{lemma}

\begin{proof}
Write
\begin{align}\label{eq1-upward-sum}
\sum_{\gfz{I\in\mathcal D(Q)}{E_\tau(I)\supset J}}
\left(\frac{|\varphi(J)|}{|I|}\right)^{\varepsilon}
&=\sum_{k=0}^\infty\sum_{\gfz{I\in\mathcal D_k(Q)}{E_\tau(I)\supset J}}
\left(\frac{|\varphi(J)|}{2^{-kn}|Q|}\right)^{\varepsilon}\\
&= \left(\sum_{k=0}^\infty 2^{kn\varepsilon}
\#\left(\left\{I\in\cd_k(Q):\ E_\tau(I)\supset J\right\}\right)\right)
\left(\frac{|\varphi(J)|}{|Q|}\right)^{\varepsilon}. \notag
\end{align}

Fix $k\in\mathbb Z_+$ and $I\in\mathcal D_k(Q)$ such that $E_\tau(I)\supset J$. By the definition of $E_\tau(I)$ in \eqref{eq-EtauI}, we have
$$
\frac{\operatorname{diam}\varphi^{-1}(I)}{\operatorname{diam}E_\tau(I)} = \frac1{2\tau}\in(0,1).
$$
Applying the quasisymmetry control estimate \eqref{eq-diamEF}, we obtain
\begin{align}\label{eq-add-k}
 \frac1{2\eta(2\tau)}\le \frac{\operatorname{diam} I}
 {
  \operatorname{diam}\varphi(E_\tau(I))
  }\le \eta( \tau^{-1}).
\end{align}
Since $\operatorname{diam} I=2^{-k}  \operatorname{diam} Q$ and
$\operatorname{diam} J\le \operatorname{diam}(E_\tau(I))$, it follows that
$$
\frac1{2\eta(2\tau)}\le \frac{ 2^{-k} \operatorname{diam} Q}
 {
  \operatorname{diam}\varphi(J)
  }.
$$
In other words,  $k$ must satisfy
\begin{align}\label{eq-k-range}
2^k\le 2\eta(2\tau) \frac{
  \operatorname{diam} Q
  }{\operatorname{diam} \varphi(J)}=:M_{J,Q}.
\end{align}

Still fix $k\in\zz_+$.
Let $\{I_i\}_i$ be the family of dyadic cubes in $\cd_k(Q)$ such that
 $E_\tau(I_i)$ contains $J$.
Note that each $E_\tau(I_i)$ is a ball centered at $c_i:=\varphi^{-1}(c_{I_i})$ with radius
$r_i:=\tau\operatorname{diam}(\varphi^{-1}(I_i))$, where
$c_{I_i}$ denotes the center of the cube $I_i$.
By \eqref{eq-ball-varphi}, we have
$$\varphi(E_\tau(I_i))\subset B\Bigl(\varphi(c_i),\, \operatorname{diam} \varphi(E_\tau(I_i))\Bigr).$$
Observe that $\varphi(c_i)=c_{I_i}$ and, by \eqref{eq-add-k} and $I_i\in\cd_k(Q)$,
$$\operatorname{diam} \varphi(E_\tau(I_i))\le 2\eta(2\tau) \operatorname{diam} I_i=2^{-k+1}\eta(2\tau) \operatorname{diam} Q.$$
From  these discussions, it follows that
$$
\varphi(J)\subset \varphi(E_\tau(I_i))\subset B\left(c_{I_i},\ 2^{-k+1}\eta(2\tau) \operatorname{diam} Q\right).
$$
Denote by   $c_J$ the center of $J$. Then, we obtain
$$
|c_{I_i}-\varphi(c_J)|\le 2^{-k+1}\eta(2\tau) \operatorname{diam} Q.
$$
Consequently, every dyadic cube $I_i\in\mathcal D_k(Q)$ satisfying $E_\tau(I_i)\supset J$ must be contained in the larger ball
$$
B\left(\varphi(c_J),\ 2^{-k+2}\eta(2\tau)\operatorname{diam} Q\right).
$$
However, such cubes $\{I_i\}_i$ are mutually disjoint, each having side length $2^{-k}\ell(Q)$,
and all of them are contained in a single ball of radius $2^{-k+2}\eta(2\tau)\ell(Q)$.
Thus, the number of the cubes $\{I_i\}_i$ is bounded by a constant depending only on $n$, $\tau$, and $\eta$. So, we conclude that
\begin{align}\label{eq-k-number}
\#\left(\left\{I\in\cd_k(Q):\ E_\tau(I)\supset J\right\}\right)\le N(n,\tau,\eta).
\end{align}

Finally, substituting the bounds \eqref{eq-k-range} and \eqref{eq-k-number} into \eqref{eq1-upward-sum} yields
\begin{align*}
\sum_{\gfz{I\in\mathcal D(Q)}{E_\tau(I)\supset J}}
\left(\frac{|\varphi(J)|}{|I|}\right)^{\varepsilon}
&\le N(n,\tau,\eta) \sum_{\{k\in \zz_+:\  2^k\le M_{J,Q}\}} 2^{kn\varepsilon}
\left(\frac{|\varphi(J)|}{|Q|}\right)^{\varepsilon}\\
&\simeq \left(M_{J,Q}\right)^{n\varepsilon} \left(\frac{|\varphi(J)|}{|Q|}\right)^{\varepsilon}\\
&\simeq 1,
\end{align*}
where in the final step we used the definition of $M_{J,Q}$ and \eqref{eq3-ball-varphi}. This completes the proof of \eqref{eq-upward-sum}.
\end{proof}

\begin{remark}
Lemma \ref{lem-QS-est0} can be interpreted as a Carleson-type estimate for the dyadic tree $\mathcal D(Q)$.
Indeed,
define an atomic measure $\mu$ on $\mathcal D(Q)$ by setting $\mu(\{I\}):=|I|^{-\varepsilon}$
for each $I\in\mathcal D(Q)$.
With this definition,  Lemma \ref{lem-QS-est0}  asserts the \emph{Carleson condition}
$$
\mu\bigl(\{I\in\mathcal D(Q): B(\varphi^{-1}(c_I), \tau\operatorname{diam}\varphi^{-1}(I))\supset J\}\bigr)
\le C\,|\varphi(J)|^{-\varepsilon}
$$
for every cube $J\subset\mathbb R^n$, where the constant $C$ is independent of  $Q$ and $J$.
\end{remark}

The following technical lemma gives a scale-localized summation estimate under distortion, which is essential for dealing with scale-separation issues.

\begin{lemma}\label{lem-QS-est2}
  Let
  \(\varphi :\ \mathbb R^n \to \mathbb R^n\) be  $\eta$-quasisymmetric,
  with the additional assumption that \(J_{\varphi} \in A_\infty(\mathbb R)\) when \(n = 1\).
   Suppose that
  $$\gamma\in(0,1)\cup (1,\infty)\quad \text{and}\quad \beta>(1-\gamma)_+.$$
Let $\tau\in(1,\infty)$. Then there exists a positive constant \(C\), depending only on
\(\tau, n, \beta, \gamma, \eta, [J_\varphi]_{A_\infty(\mathbb R^n)}\),
such that for any \(R\in\mathcal D(\mathbb R^n)\) and \(I\in\mathcal D(\mathbb R^n)\),
\begin{align}\label{eq-sum-varphi-IR}
\sum_{\gfz{J\in\mathcal D(R)}{J\subset E_\tau(I)}}
\left(\frac{|J|}{|R|}\right)^\beta
\left(\frac{|\varphi(J)|}{|I|}\right)^\gamma
\le
C \min\left\{
\left(\frac{|\varphi^{-1}(I)|}{|R|}\right)^\beta,\
\left(\frac{|\varphi(R)|}{|I|}\right)^\gamma
\right\},
\end{align}
where \(E_\tau(I)\) is the ball defined in \eqref{eq-EtauI}.
\end{lemma}

\begin{proof}
Choose \(k_0\in\zz\) such that
$2^{-k_0-1}\le 2\tau\operatorname{diam}\varphi^{-1}(I)< 2^{-k_0}.$
Note that the set \(E_\tau(I)\) can be covered by finitely many dyadic cubes \(Q_1,\dots, Q_N\) in \(\cd_{k_0}(\rn)\), where \(N\le 2^n\).
Using the equivalence
\begin{align}\label{eq-mIR00}
\operatorname{diam}\varphi^{-1}(I)\simeq \operatorname{diam}(Q_i),
\end{align}
 we first apply \eqref{eq-diamEF} and then \eqref{eq3-ball-varphi}, thereby obtaining
$$
\operatorname{diam}(I) \simeq  \operatorname{diam} \varphi(Q_i)
$$
and
\begin{align}\label{eq-mIR0}
|I|\simeq |\varphi(Q_i)|.
\end{align}
Using \eqref{eq-mIR00} and \eqref{eq-mIR0}, we now prove \eqref{eq-sum-varphi-IR}
by treating the cases \(\gamma\in(1,\infty)\) and \(\gamma\in(0,1)\) separately.

\medskip

{\bf Case 1: \(\gamma\in(1,\infty)\) and \(\beta\in(0,\infty)\).\,}
Since \(\beta\in(0,\infty)\) and  $|J|\le |R|$ for all $J\in\cd(R)$,
it follows that
\begin{align}\label{eq1-n=1}
 \sum_{\gfz{J\in\cd(R)}{J\subset E_\tau(I)}}
  \left(\frac{|J|}{|R|}\right)^{\beta}
  \left(\frac{|\varphi(J)|}{|I|}\right)^{\gamma}
  \le \sum_{J\in\cd(R)}
  \left(\frac{|\varphi(J)|}{|I|}\right)^{\gamma}
  \ls
  \left(\frac{|\varphi(R)|}{|I|}\right)^{\gamma},
  \end{align}
where the last step uses Lemma \ref{lem-QS-est1} and $\gamma\in(1,\infty)$.

Next, for any $J\subset E_\tau(I)$, applying \eqref{eq3-ball-varphi} gives
$$
|J|\le | E_\tau(I)| \simeq \left[\operatorname{diam}\varphi^{-1}(I)\right]^n \simeq
|\varphi^{-1}(I)|.
$$
Consequently,
\begin{align*}
 \sum_{\gfz{J\in\cd(R)}{J\subset E_\tau(I)}}
  \left(\frac{|J|}{|R|}\right)^{\beta}
  \left(\frac{|\varphi(J)|}{|I|}\right)^{\gamma}
  \ls \left(\frac{|\varphi^{-1}(I)|}{|R|}\right)^{\beta}
  \sum_{\gfz{J\in\cd(R)}{J\subset E_\tau(I)}}
  \left(\frac{|\varphi(J)|}{|I|}\right)^{\gamma}.
  \end{align*}
For $J\subset E_\tau(I)$, we have $J\subset \cup_{i=1}^N Q_i$.
This, along with \eqref{eq-mIR0}, Lemma \ref{lem-QS-est1} and \(\gamma\in(1,\infty)\), yields
\begin{align*}
 \sum_{\gfz{J\in\cd(R)}{J\subset E_\tau(I)}}
  \left(\frac{|\varphi(J)|}{|I|}\right)^{\gamma}
  &\le \sum_{i=1}^N \sum_{\gfz{J\in\cd(R)}{J\subset Q_i}}
  \left(\frac{|\varphi(J)|}{|I|}\right)^{\gamma}
  \simeq \sum_{i=1}^N \sum_{J\in\cd(Q_i)}
  \left(\frac{|\varphi(J)|}{|\varphi(Q_i)|}\right)^{\gamma}\ls 1.
\end{align*}
 Combining the last two estimates gives
 \begin{align}\label{eq-n=1}
 \sum_{\gfz{J\in\cd(R)}{J\subset E_\tau(I)}}
  \left(\frac{|J|}{|R|}\right)^{\beta}
  \left(\frac{|\varphi(J)|}{|I|}\right)^{\gamma}
  \ls \left(\frac{|\varphi^{-1}(I)|}{|R|}\right)^{\beta}.
  \end{align}
It follows from  \eqref{eq1-n=1} and \eqref{eq-n=1} that \eqref{eq-sum-varphi-IR}
holds for the case \(\gamma\in(1,\infty)\) and \(\beta\in(0,\infty)\).

\medskip

{\bf Case 2: \(\gamma\in(0, 1)\) and \(\beta\in(1-\gamma,\infty)\).\,}
By \(\gamma\in(0, 1)\) and the H\"older inequality, we have
\begin{align}\label{eq-JphiJ}
  \sum_{\gfz{J\in\cd(R)}{J\subset E_\tau(I)}}
  |J|^{\beta }
 |\varphi(J)|^{\gamma}
  &=  \sum_{\gfz{J\in\cd(R)}{J\subset E_\tau(I)}}
 |J|^{\beta+\gamma}
  \left(\fint_J J_\varphi(x)\,dx\right)^{\gamma}\\
  &\le \left( \sum_{\gfz{J\in\cd(R)}{J\subset E_\tau(I)}}
 |J|^{\beta+\gamma}\right)^{1-\gamma}
 \left( \sum_{\gfz{J\in\cd(R)}{J\subset E_\tau(I)}}
 |J|^{\beta+\gamma}
\fint_J J_\varphi(x)\,dx\right)^{\gamma}. \notag
\end{align}

On the one hand, applying  Lemma \ref{lem-sum-cube} with the exponent
$p=\beta+\gamma\in(1,\infty)$, we obtain
\begin{align*}
\sum_{J\in\cd(R)}
 |J|^{\beta+\gamma}
 \simeq |R|^{\beta+\gamma}.
\end{align*}
On the other hand,
\begin{align*}
 \sum_{\gfz{J\in\cd(R)}{J\subset E_\tau(I)}}
 |J|^{\beta+\gamma}
\fint_J J_\varphi(x)\,dx
&\le \sum_{k=0}^\infty \sum_{J\in\cd_k(R)}
 |J|^{\beta+\gamma-1}
\int_J J_\varphi(x)\,dx\\
&= \sum_{k=0}^\infty (2^{-kn}|R|)^{\beta+\gamma-1}
\left( \sum_{J\in\cd_k(R)}
\int_J J_\varphi(x)\,dx\right)\\
&\le \sum_{k=0}^\infty (2^{-kn}|R|)^{\beta+\gamma-1}
\left(\int_{R} J_\varphi(x)\,dx\right)\\
&\simeq |R|^{\beta+\gamma-1}
|\varphi(R)|.
\end{align*}
Inserting the last two estimates into  \eqref{eq-JphiJ} yields
\begin{align*}
 \sum_{\gfz{J\in\cd(R)}{J\subset E_\tau(I)}}
  |J|^{\beta }
 |\varphi(J)|^{\gamma}
    &\ls |R|^{\beta }
 |\varphi(R)|^{\gamma}
\end{align*}
and, hence,
\begin{align}\label{eq1-n>=2}
  \sum_{\gfz{J\in\cd(R)}{J\subset E_\tau(I)}}
  \left(\frac{|J|}{|R|}\right)^{\beta }
  \left(\frac{|\varphi(J)|}{|I|}\right)^{\gamma}
  \ls \left(\frac{|\varphi(R)|}{|I|}\right)^{\gamma}.
  \end{align}

  Recall that $E_\tau(I)$ can be covered by finitely many dyadic cubes $Q_1,\dots, Q_N$ in $\cd_{k_0}(\rn)$ and, moreover, from \eqref{eq-mIR00} and \eqref{eq3-ball-varphi} it follows that
$$
|Q_i| \simeq |\varphi^{-1}(I)|
$$
for each $i=1,\dots, N$.
Consequently, applying Lemma \ref{lem-sum-cube} again gives
\begin{align*}
  \sum_{\gfz{J\in\cd(R)}{J\subset E_\tau(I)}}
 |J|^{\beta+\gamma}
 \le \sum_{i=1}^N \sum_{J\in\cd(Q_i)}
 |J|^{\beta+\gamma}
 \simeq \sum_{i=1}^N |Q_i|^{\beta+\gamma}
 \simeq |\varphi^{-1}(I)|^{\beta+\gamma}.
\end{align*}
Meanwhile, we have
\begin{align*}
 \sum_{\gfz{J\in\cd(R)}{J\subset E_\tau(I)}}
 |J|^{\beta+\gamma}
\fint_J J_\varphi(x)\,dx
&\le\sum_{i=1}^N  \sum_{\gfz{J\in\cd(R)}{J\subset Q_i}}
 |J|^{\beta+\gamma-1}
\int_J J_\varphi(x)\,dx\\
&\le \sum_{i=1}^N\sum_{k=0}^\infty \sum_{J\in\cd_k(Q_i)}
 |J|^{\beta+\gamma-1}
\int_J J_\varphi(x)\,dx\\
&= \sum_{i=1}^N\sum_{k=0}^\infty (2^{-kn}|Q_i|)^{\beta+\gamma-1}
\left( \sum_{J\in\cd_k(Q_i)}
\int_J J_\varphi(x)\,dx\right)\\
&=\sum_{i=1}^N\sum_{k=0}^\infty
(2^{-kn}|Q_i|)^{\beta+\gamma-1} \left(
\int_{Q_i} J_\varphi(x)\,dx\right)\\
&\simeq \sum_{i=1}^N |Q_i|^{\beta+\gamma-1} |\varphi(Q_i)|\\
&\simeq |\varphi^{-1}(I)|^{\beta+\gamma-1} |I|,
\end{align*}
where  the last step is due to \eqref{eq-mIR0}.
Using the last two estimates, we now continue \eqref{eq-JphiJ} as follows:
\begin{align*}
 \sum_{\gfz{J\in\cd(R)}{J\subset E_\tau(I)}}
  |J|^{\beta }
 |\varphi(J)|^{\gamma}
    &\ls |\varphi^{-1}(I)|^{\beta }
 |I|^{\gamma}.
\end{align*}
Consequently, we obtain
\begin{align}\label{eq2-n>=2}
  \sum_{\gfz{J\in\cd(R)}{J\subset E_\tau(I)}}
  \left(\frac{|J|}{|R|}\right)^{\beta }
  \left(\frac{|\varphi(J)|}{|I|}\right)^{\gamma}
  \ls \left(\frac{|\varphi^{-1}(I)|}{|R|}\right)^{\beta }.
\end{align}
Combining \eqref{eq1-n>=2} and \eqref{eq2-n>=2} yields \eqref{eq-sum-varphi-IR}  in the case \(\gamma\in(0, 1)\) and \(\beta\in(1-\gamma,\infty)\).
\end{proof}

\subsection{Wavelet characterizations of $\mathcal Q_\alpha(\rn)$}\label{ss5.2}

Let $\mathcal D(\rn)$ denote the family of standard dyadic cubes on $\rn$ (see \eqref{eq-dyadic-rn}), and set
$$
E_\ast:=\{0,1\}^n\setminus\{(\underbrace{0,\ldots,0}_{n\text{ times}})\}.
$$
 According to \cite{Daubechies1988CPAM} or \cite[Sections~3.8 and 3.9]{Meyer1992book},
there exists a family of \emph{wavelets} $\{\psi_I^{\nu}:\ I\in\mathcal D(\rn),\,\nu\in E_\ast\}$ satisfying the following properties:
\begin{enumerate}[label=(P\arabic*), ref=(P\arabic*)]
\item \label{P1} The family $\{\psi_I^{\nu}:\ I\in\mathcal D(\rn),\,\nu\in E_\ast \}$ forms an orthonormal basis of $L^2(\rn)$;

\item \label{P2} There exists a constant $m_\ast\in [1,\infty)$ such that for every $I\in \mathcal D(\rn)$ and $\nu\in E_\ast $,
$$
\operatorname{supp} \psi^\nu_I\subset m_\ast I,
$$
where $m_\ast I$ denotes the $m_\ast $-dilation of $I$ about its center;

\item \label{P3} There exists a constant $C\in(0,\infty)$ such that for all $I\in \mathcal D(\rn)$, $\nu\in E_\ast $ and $x\in\rn$,
$$
|\psi_I^\nu(x)|+\ell(I)\,|\nabla \psi^\nu_I(x)|\le C |I|^{-\frac 12};
$$

\item \label{P4} For any $I\in \mathcal D(\rn)$ and $\nu\in E_\ast$,
$$
\int_{\rn}\psi_I^\nu(x)\,dx=0.
$$

\end{enumerate}
Since the family $\{\psi_I^{\nu}:\ I\in\mathcal D(\rn),\,\nu\in E_\ast\}$ forms an orthonormal basis of $L^2(\rn)$, every function $f\in L^2(\rn)$ enjoys the following \emph{wavelet expansion}
\begin{align}\label{eq-crf}
f=\sum_{I\in \mathcal D(\rn)}\sum_{\nu\in E_\ast} \langle f,\,\psi_I^\nu \rangle \psi_I^\nu
\quad\,\text{in}\ \, L^2(\rn),
\end{align}
where the bracket $\langle \cdot,\cdot \rangle$ denotes the inner product in $L^2(\rn)$; that is, for any $f,g\in L^2(\rn)$,
$$
\langle f,\, g\rangle=\int_{\mathbb R^n} f(x)\overline{g(x)}\,dx.
$$

Now, we recall the wavelet characterizations of  $\mathcal Q_\alpha(\rn)$
established in \cite[Theorem 6.2]{EssenJansonPengXiao2000IUMJ}.
For any dyadic cube $Q\in \mathcal D(\rn)$ and any sequence
$s=\{s_I^\nu:\ I\in\mathcal D(Q),\,\nu\in E_\ast\}$, define
\begin{align}\label{eq-wavelets-T}
T_{s,\,\alpha}(Q) := \frac1{|Q|}\sum_{I\in\mathcal D(Q)}\sum_{\nu\in E_\ast} \left( \frac{|Q|}{|I|} \right)^{\frac{2\alpha}{n}} |s_{I}^\nu|^2.
\end{align}

\begin{proposition}[\cite{EssenJansonPengXiao2000IUMJ}]\label{prop-waveletsQ}
Let $\alpha\in(0,1)$. If $f\in \mathcal Q_{\alpha}(\mathbb R^n)$, then the sequence of its wavelet coefficients
$$
s_I^\nu:= \langle f,\psi_I^\nu\rangle=\int_{\mathbb R^n} f(x)\overline{\psi_I^\nu(x)}\,dx
$$
satisfies
\begin{equation}\label{eq:T-finite}
\sup_{Q\in\mathcal D(\rn)} T_{s,\,\alpha}(Q)<\infty.
\end{equation}
Conversely, every sequence $s=\{s_I^\nu:\ I\in\mathcal D(\rn),\,\nu\in E_\ast\}$ satisfying \eqref{eq:T-finite} is the sequence of wavelet coefficients of a unique (modulo constants) $f\in
\mathcal Q_{\alpha}(\mathbb R^n)$; moreover,
$$
\|f\|_{\mathcal Q_{\alpha}(\mathbb R^n)}^2
\simeq \sup_{Q\in\mathcal D(\rn)} T_{s,\,\alpha}(Q),
$$
with implicit constants depending only on $\alpha$ and $n$.
\end{proposition}

\subsection{Construction of $\mathcal Q_\alpha(\rn)$-functions
and norm estimates for their compositions with $\varphi$}\label{ss5.3}

The main aim of this section is the following wavelet-based construction of a $\mathcal Q_\alpha(\mathbb R^n)$ function, which will be crucial in the proof of Theorem \ref{thm-trace-gauge}.

\begin{theorem}\label{thm-upperf}
   Let \(\alpha \in (0, \,\min\{1,\, \frac n2\})\), $\beta_\alpha:=\frac{n-2\alpha}{n}$ and \(\varphi :\ \mathbb R^n \to \mathbb R^n\) be $\eta$-quasisymmetric, with the additional assumption that \(J_\varphi,\,J_{\varphi^{-1}} \in A_\infty(\mathbb R)\) when \(n = 1\).
  Fix $Q_0\in\mathcal D(\rn)$ and assume that $\{d_I\}_{I\in\mathcal D(Q_0)}$ satisfies the $(\beta_\alpha,\,\varphi)$-packing condition on $Q_0$.
For any $N\in\nn$, define the truncation
\begin{align}\label{eq-fN-dI}
f_N:=\sum_{\gfz{I\in\mathcal D(Q_0)}{\ell(I)\ge 2^{-N}\ell(Q_0)}}\sum_{\nu\in E_\ast} \sqrt{|I|\,d_I}\, \psi_I^\nu.
\end{align}
Then,  there exists a positive constant $C=C(\eta, n,\alpha, [J_{\varphi^{-1}}]_{A_\infty(\rn)})$
such that
 \begin{align}\label{eq-fNinQ}
\sum_{\gfz{I\in\mathcal D(Q_0)}{\ell(I)\ge 2^{-N}\ell(Q_0)}} d_I \left(\frac{|I|}{|Q_0|}\right)^{1-\frac{2\alpha}{n}}
\le C \|f_N\|_{\mathcal Q_\alpha(\rn)}^2
\end{align}
and
   \begin{align}\label{eq-fvarphi}
   \|f_N\circ\varphi\|_{\mathcal Q_\alpha(\rn)}\le C.
   \end{align}
 In particular, the constant \(C\) is  independent of \(N\), \(Q_0\) and the sequence \(\{d_I\}_{I\in\mathcal D(Q_0)}\).
\end{theorem}

In order to prove Theorem \ref{thm-upperf}, we start with the following remark and then present two lemmas.

\begin{remark}\label{rem-0extension}
Fix a dyadic cube $Q_0\in\mathcal D(\mathbb R^n)$ and suppose that $\{d_I\}_{I\in\mathcal D(Q_0)}$ satisfies the $(\beta_\alpha,\,\varphi)$-packing condition on $Q_0$. By \eqref{eq-packing}, this means
$$
\sup_{R\in\mathcal D(Q_0)}\sum_{I\in\mathcal D(R)} d_I\left(\frac{|\varphi^{-1}(I)|}{|\varphi^{-1}(R)|}\right)^{\beta_\alpha} \le 1.
$$
Note that for $R\in\mathcal D(Q_0)$, the condition $I\in\mathcal D(R)$ is equivalent to $I\in\mathcal D(Q_0)$ and $I\subset R$.
We extend the sequence $\{d_I\}_{I\in\mathcal D(Q_0)}$ to a sequence indexed by $\mathcal D(\mathbb R^n)$ via setting
$$d_I:=0\quad \text{as}\quad I\notin\mathcal D(Q_0).$$
It is straightforward to verify that this zero-extended sequence satisfies the same $(\beta_\alpha,\,\varphi)$-packing condition on every dyadic cube $Q \in \mathcal D(\mathbb R^n)$.
Indeed, since any two dyadic cubes are either disjoint or nested, there are only the following three cases: $Q \cap Q_0 = \emptyset$, $Q \subset Q_0$, or $Q_0 \subset Q$.
\begin{enumerate}[label=\textup{(\roman*)}]
  \item  If $Q \cap Q_0 = \emptyset$, then for any $R \in \mathcal D(Q)$ and $I \in \mathcal D(R)$, we have $d_I = 0$; hence all terms in the $(\beta_\alpha,\,\varphi)$-packing condition vanish.

  \item If $Q \subset Q_0$, then for any $R \in \mathcal D(Q)$, the assumption on $Q_0$ implies that
      $$
      \sum_{I \in \mathcal D(R)} d_I \left(\frac{|\varphi^{-1}(I)|}{|\varphi^{-1}(R)|}\right)^{\beta_\alpha} \le 1.
      $$

  \item Now suppose $Q_0 \subset Q$. For any $R \in \mathcal D(Q)$ with $R \cap Q_0 = \emptyset$, case (i) shows that all terms in the $(\beta_\alpha,\,\varphi)$-packing condition vanish. For any $R \in \mathcal D(Q)$ with $R \subset Q_0$, the desired estimate follows immediately from the assumption on $Q_0$, just as in case (ii).
         For those dyadic cubes $R \in \mathcal D(Q)$ such that $R \supset Q_0$, we have $|\varphi^{-1}(R)| \ge |\varphi^{-1}(Q_0)|$ and, hence,
    \begin{align*}\sum_{I\in\cd(R)} d_I\left(\frac{|\varphi^{-1}(I)|}{|\varphi^{-1}(R)|}\right)^{\beta_\alpha}
    &= \sum_{I\in\cd(R)\cap \cd(Q_0)} d_I\left(\frac{|\varphi^{-1}(I)|}{|\varphi^{-1}(R)|}\right)^{\beta_\alpha}\\
    &
    \le
    \sum_{I\in\cd(Q_0)} d_I\left(\frac{|\varphi^{-1}(I)|}{|\varphi^{-1}(Q_0)|}\right)^{\beta_\alpha}\\
    &
    \le
    1,
    \end{align*}
which consequently yields the packing condition on $Q$.
\end{enumerate}
\end{remark}

The following off-diagonal estimate for wavelets, which is adapted to the quasisymmetric  mapping $\varphi$, will be used in the next subsection.

\begin{lemma}\label{lem-off-diag}
Let $\varphi:\ \rn\to\rn$ be $\eta$-quasisymmetric,  with the additional assumption that \(J_{\varphi},\, J_{\varphi^{-1}} \in A_\infty(\mathbb R)\) when \(n = 1\).
Then, there exists a positive constant $C$, such that for any $I,J\in\cd(\rn)$
and $\nu,\lambda\in E_\ast$,
\begin{align}\label{eq-off-diag}
\left|\langle \psi_I^\nu \circ\varphi,\ \psi_J^\lambda\rangle \right|
&\le C\left(\frac{|J|}{|I|}\right)^\frac 12 \ \min\left\{
\frac{|\varphi^{-1}(I)|}{|J|},\ \ \left(\frac{|\varphi(J)|}{|I|}\right)^\frac1n
\right\}.
\end{align}
\end{lemma}

\begin{proof}
On the one hand, by \ref{P2} and \ref{P3}, we have
\begin{align}\label{eq1-off-diag}
 \left|\langle \psi_I^\nu \circ\varphi,\ \psi_J^\lambda\rangle \right|
 &\le  \int_{\varphi^{-1}(m_\ast I)} |\psi_I^\nu(\varphi(x))|\, |\psi_J^\lambda(x) |\,dx\\
 &\ls\frac{|\varphi^{-1}(m_\ast I)|}{\sqrt{|I||J|}}\notag \\
 &\ls \frac{|\varphi^{-1}(I)|}{\sqrt{|I||J|}}, \notag
\end{align}
where in the last step we used the doubling property of $J_{\varphi^{-1}}$ from Proposition \ref{prop-varphi-weight}.

On the other hand, denote by $c_J$ the center of the dyadic cube $J$. Moreover, by \eqref{eq3-ball-varphi} and Proposition \ref{prop-varphi-weight}(iv), we have
$$
\left(\frac{|\varphi(J)|}{|I|}\right)^{\frac 1 n} \simeq \frac{\operatorname{diam}(\varphi(J))}{\ell(I)}
\simeq \frac{\operatorname{diam}(\varphi(m_\ast J))}{\ell(I)}.
$$
From these facts, together with \ref{P4} and \ref{P3}, it follows that
\begin{align}\label{eq2-off-diag}
 \left|\langle \psi_I^\nu \circ\varphi,\ \psi_J^\lambda\rangle \right|
 &=\left|\int_\rn
 \left(\psi_I^\nu\big(\varphi(x)\big)-\psi_I^\nu\big(\varphi(c_J)\big)\right)
 \, \overline{\psi_J^\lambda(x)}\, dx\right|\\
 &\le  \int_{m_\ast J} \|\nabla \psi_I^\nu\|_{L^\infty(\rn)} |\varphi(x)-\varphi(c_J)|\, |\psi_J^\lambda(x) |\,dx \notag\\
 &\ls\frac{\operatorname{diam}(\varphi(m_\ast J))}{\ell(I)}\sqrt{\frac {|J|}{|I|}} \notag\\
 &\ls \left(\frac{|\varphi(J)|}{|I|}\right)^\frac1n \sqrt{\frac {|J|}{|I|}}. \notag
\end{align}
Combining \eqref{eq1-off-diag} and \eqref{eq2-off-diag} yields the desired estimate in \eqref{eq-off-diag}.
\end{proof}

\begin{lemma}\label{lem-coefN}
Let all the assumptions be as in Theorem \ref{thm-upperf}. Suppose that
\begin{align}\label{eq-ez-range}
0<\varepsilon <\min\left\{\frac 12, \ 1-\alpha,\  \frac{2\alpha}{n}\right\}.
\end{align}
Then, there exist sufficiently large positive constants $\kappa,\tau$ and $C$,
depending only on $n$, $\alpha$, $\eta$, $\varepsilon$, $[J_\varphi]_{A_\infty(\rn)}$ and $ [J_{\varphi^{-1}}]_{A_\infty(\rn)}$, such that
for any $N\in\nn$, $J\in\cd(\rn)$ and $\lambda\in E_\ast$,
\begin{align}\label{eq-coefN}
\left| \langle f_N\circ\varphi,\, \psi_J^\lambda\rangle \right|^2
\le C|J|
\left(
\sum_{\gfz{I\in\mathcal D_{\le N}(Q_0)}{I\subset \varphi(\kappa J)}} d_I \left(\frac{|\varphi^{-1}(I)|}{|J|}\right)^{1-\varepsilon}
+\sum_{\gfz{I\in\mathcal D_{\le N}(Q_0)}{E_\tau(I)\supset J}}
d_I \left(\frac{|\varphi(J)|}{|I|}\right)^{\frac {2-2\varepsilon} n} \right),
\end{align}
where $E_\tau(I)$ is the $\varphi^{-1}$-ball defined  in \eqref{eq-EtauI}, namely,
\begin{align*}
E_\tau(I):=B\Bigl(\varphi^{-1}(c_I),\ \tau\operatorname{diam}\varphi^{-1}(I)\Bigr),
\end{align*}
with $c_I$ being the center of the cube $I$.
\end{lemma}

\begin{proof}
  For any $N\in\nn$, set
\begin{align}\label{eq-cdN}
\cd_{\le N}(Q_0):=\left\{I\in\mathcal D(Q_0):\ \ell(I)\ge 2^{-N}\ell(Q_0)\right\} =\bigcup_{k=0}^N \cd_k(Q_0).
\end{align}
 Since (see \eqref{eq-fN-dI})
$$
f_N=\sum_{I\in\mathcal D_{\le N}(Q_0)}\sum_{\nu\in E_\ast} \sqrt{|I|\,d_I}\, \psi_I^\nu,$$
it follows that
\begin{align}\label{eq-fNpsi}
\langle f_N \circ\varphi,\, \psi_J^\lambda\rangle
=\sum_{I\in\mathcal D_{\le N}(Q_0)}\sum_{\nu\in E_\ast} \sqrt{|I|\,d_I}\, \langle \psi_I^\nu  \circ\varphi,\, \psi_J^\lambda\rangle.
\end{align}
By the support conditions of wavelets in \ref{P2}, we observe that the bracket \(\langle \psi_I^\nu,\, \psi_J^\lambda\rangle\) is nonzero only if
$$
\left(\varphi^{-1}(m_\ast I)\right)\cap (m_\ast J)\neq \emptyset.
$$
In view of \eqref{eq-ball-varphi} and \eqref{eq3-ball-varphi},
one of the following two cases must occur:
\begin{enumerate}
  \item[\rm (a)] If \(|\varphi^{-1}(I)| \le |J|\), then there exist constants \(\lambda\in(1,\infty)\) and \(\kappa\in(1,\infty)\) such that
    $$
    \operatorname{diam}\varphi^{-1}(I)\le \lambda\,\ell(J)
    \quad\text{and}\quad
    \varphi^{-1}(I)\subset \kappa J.
    $$
    Note that \(\varphi^{-1}(I)\subset \kappa J\) if and only if \(I\subset \varphi(\kappa J)\).

  \item[\rm (b)] If \(|\varphi^{-1}(I)| > |J|\), then there exist constants \(\gamma\in(1,\infty)\) and \(\tau\in(1,\infty)\) such that
    $$
    \ell(J)\le \gamma\,\operatorname{diam}\varphi^{-1}(I)
    $$
    and
    $$
    J\subset  B\left(\varphi^{-1}(c_I),\ \tau \operatorname{diam}\varphi^{-1}(I)\right)
    =E_\tau(I).
    $$
\end{enumerate}
Applying these facts and Lemma \ref{lem-off-diag} to \eqref{eq-fNpsi}, we obtain
\begin{align}\label{eq-fNpsi-sum}
|\langle f_N \circ\varphi,\ \psi_J^\lambda\rangle|
&\lesssim \sum_{\gfz{I\in\mathcal D_{\le N}(Q_0)}{I\subset \varphi(\kappa J)}}
\sqrt{|J|\,d_I}\, \left(\frac{|\varphi^{-1}(I)|}{|J|}\right)
+
\sum_{\gfz{I\in\mathcal D_{\le N}(Q_0)}{E_\tau(I)\supset J}}
\sqrt{|J|\,d_I} \,\left(\frac{|\varphi(J)|}{|I|}\right)^{\frac 1n}\\
&=: {\rm I}+{\rm II},\notag
\end{align}
where the first sum follows from case (a), and the second from case (b).
We next turn to estimating the squares of the terms \({\rm I}\) and \({\rm II}\), respectively.

For term ${\rm I}$, by the H\"older inequality, we write
\begin{align*}
  {\rm I}^2
  &= |J|
  \left(\sum_{\gfz{I\in\mathcal D_{\le N}(Q_0)}{I\subset \varphi(\kappa J)}}
\sqrt{d_I}\,\left(\frac{|\varphi^{-1}(I)|}{|J|}\right)^{\varepsilon}\, \left(\frac{|\varphi^{-1}(I)|}{|J|}\right)^{1-\varepsilon}
\right)^2\\
&\le |J|
\left(\sum_{\gfz{I\in\mathcal D_{\le N}(Q_0)}{I\subset \varphi(\kappa J)}} d_I \left(\frac{|\varphi^{-1}(I)|}{|J|}\right)^{1-\varepsilon}\right)
\left(
\sum_{\gfz{I\in\mathcal D_{\le N}(Q_0)}{I\subset \varphi(\kappa J)}}
\left(\frac{|\varphi^{-1}(I)|}{|J|}\right)^{2\varepsilon}
\left(\frac{|\varphi^{-1}(I)|}{|J|}\right)^{1-\varepsilon}\right)\\
&= |J|
\left(\sum_{\gfz{I\in\mathcal D_{\le N}(Q_0)}{I\subset \varphi(\kappa J)}} d_I \left(\frac{|\varphi^{-1}(I)|}{|J|}\right)^{1-\varepsilon}\right)
\left(
\sum_{\gfz{I\in\mathcal D_{\le N}(Q_0)}{I\subset \varphi(\kappa J)}}
\left(\frac{|\varphi^{-1}(I)|}{|J|}\right)^{1+\varepsilon}\right).
\end{align*}
For the second term on the right-hand side, choose \(k_0\in\zz\) such that
$$
2^{-k_0-1}\le \operatorname{diam}\varphi(\kappa J)< 2^{-k_0}.
$$
Then the set \(\varphi(\kappa J)\) can be covered by finitely many dyadic cubes \(Q_1,\dots, Q_N\) in \(\cd_{k_0}(\rn)\), with \(N\le 2^n\). By \eqref{eq-diamEF} and the fact that
$$
\operatorname{diam}\varphi(\kappa J) \simeq 2^{-k_0}\simeq \operatorname{diam} Q_i,
$$
we deduce that
$$
\operatorname{diam} J \simeq \operatorname{diam} (\kappa J)\simeq  \operatorname{diam}\varphi^{-1}(Q_i)
$$
and, hence, by \eqref{eq3-ball-varphi},
$$
|J| \simeq |\varphi^{-1}(Q_i)|.
$$
Combining these facts with Lemma \ref{lem-QS-est1} yields
\begin{align*}
  \sum_{\gfz{I\in\mathcal D_{\le N}(Q_0)}{I\subset \varphi(\kappa J)}}
\left(\frac{|\varphi^{-1}(I)|}{|J|}\right)^{1+\varepsilon}
\le \sum_{i=1}^N \sum_{I\in\mathcal D(Q_i)}
\left(\frac{|\varphi^{-1}(I)|}{|J|}\right)^{1+\varepsilon}
\lesssim \sum_{i=1}^N \left(\frac{|\varphi^{-1}(Q_i)|}{|J|}\right)^{1+\varepsilon} \simeq 1.
\end{align*}
Therefore, we conclude that
\begin{align}\label{eq-termI}
  {\rm I}^2
  \ls |J|
\sum_{\gfz{I\in\mathcal D_{\le N}(Q_0)}{I\subset \varphi(\kappa J)}} d_I \left(\frac{|\varphi^{-1}(I)|}{|J|}\right)^{1-\varepsilon}.
\end{align}

Now, we turn to term ${\rm II}$. Again, applying the H\"older inequality gives
\begin{align*}
  {\rm II}^2 &
  =|J|\left(\sum_{\gfz{I\in\mathcal D_{\le N}(Q_0)}{E_\tau(I)\supset J}}
\sqrt{\,d_I} \,\left(\frac{|\varphi(J)|}{|I|}\right)^{\frac 1n}\right)^2 \\
  &=|J|\left(\sum_{\gfz{I\in\mathcal D_{\le N}(Q_0)}{E_\tau(I)\supset J}}
\sqrt{\,d_I} \left(\frac{|\varphi(J)|}{|I|}\right)^{\frac {1-\varepsilon} n} \,\left(\frac{|\varphi(J)|}{|I|}\right)^{\frac \varepsilon n}\right)^2 \\
&\le |J|\left(\sum_{\gfz{I\in\mathcal D_{\le N}(Q_0)}{E_\tau(I)\supset J}}
d_I \left(\frac{|\varphi(J)|}{|I|}\right)^{\frac {2-2\varepsilon} n} \right)
\left(
\sum_{\gfz{I\in\mathcal D_{\le N}(Q_0)}{E_\tau(I)\supset J}}
\left(\frac{|\varphi(J)|}{|I|}\right)^{\frac {2\varepsilon} n}
\right).
\end{align*}
Applying Lemma \ref{lem-QS-est0} yields
$$
\sum_{\gfz{I\in\mathcal D_{\le N}(Q_0)}{E_\tau(I)\supset J}}
\left(\frac{|\varphi(J)|}{|I|}\right)^{\frac {2\varepsilon} n}
\le \sum_{\gfz{I\in\mathcal D(Q_0)}{E_\tau(I)\supset J}}
\left(\frac{|\varphi(J)|}{|I|}\right)^{\frac {2\varepsilon} n} \ls 1.
$$
Thus, we obtain
\begin{align}\label{eq-termII}
  {\rm II}^2 \ls |J|\sum_{\gfz{I\in\mathcal D_{\le N}(Q_0)}{E_\tau(I)\supset J}}
d_I \left(\frac{|\varphi(J)|}{|I|}\right)^{\frac {2-2\varepsilon} n}.
\end{align}

Finally, taking the square of both sides in \eqref{eq-fNpsi-sum}, and then applying the estimates \eqref{eq-termI} and \eqref{eq-termII}, we obtain \eqref{eq-coefN}.
\end{proof}

Now, applying Lemmas \ref{lem-off-diag} and \ref{lem-coefN}, together with Proposition \ref{prop-waveletsQ}, we show Theorem \ref{thm-upperf}.

\begin{proof}[Proof of Theorem \ref{thm-upperf}]
For any $N\in\nn$, let $\cd_{\le N}(Q_0)$ be as defined in \eqref{eq-cdN}.
To prove \eqref{eq-fNinQ}, using the orthogonality of the wavelets
$\{\psi_I^\nu:\ I\in\mathcal D(\rn),\,\nu\in E_\ast\}$, as stated in \ref{P1}, we obtain
\begin{align*}
 \langle f_N,\,\psi_{I}^{\nu}\rangle=
 \begin{cases}
  \sqrt{|I|\,d_{I}} \qquad & \text{if } I\in\mathcal D_{\le N}(Q_0)\ \, \text{and}\ \, \nu\in E_\ast;\\
  0 \qquad  & \text{otherwise}.
 \end{cases}
\end{align*}
Combining this with the first part of Proposition \ref{prop-waveletsQ}
and noting that $\# E_\ast =2^n-1$, we conclude that
\begin{align}\label{eq-lower}
(2^n-1)\sum_{I\in\mathcal D_{\le N}(Q_0)} d_I \left(\frac{|I|}{|Q_0|}\right)^{1-\frac{2\alpha}{n}}
&=
\frac1{|Q_0|}\sum_{I\in\mathcal D_{\le N}(Q_0)}\sum_{\nu\in E_\ast} \left( \frac{|I|}{|Q_0|} \right)^{-\frac{2\alpha}{n}} |\langle f_N,\,\psi_I^\nu\rangle|^2\\
&\ls \|f_N\|_{\mathcal Q_\alpha(\rn)}^2. \notag
\end{align}
This proves \eqref{eq-fNinQ}.

It remains to show \eqref{eq-fvarphi}. In view of \eqref{eq-wavelets-T}, for any dyadic cube $R\in\cd(\rn)$, we define
$$
T_{f_N\circ\varphi}(R):= \frac1{|R|}\sum_{J\in\mathcal D(R)}\sum_{\lambda\in E_\ast} \left( \frac{|J|}{|R|} \right)^{-\frac{2\alpha}{n}} |\langle f_N\circ\varphi,\ \psi_J^\lambda\rangle|^2.
$$
By the second part of Proposition \ref{prop-waveletsQ}, the proof of \eqref{eq-fvarphi} reduces to showing that for every  $R\in\cd(\rn)$,
   \begin{align}\label{eq1-fvarphi}
   T_{f_N\circ\varphi}(R) \ls 1.
   \end{align}
To prove \eqref{eq1-fvarphi}, we fix $R\in\cd(\rn)$. Then, applying \eqref{eq-coefN} and $\# E_\ast=2^n-1$ yields
\begin{align}\label{eq2-fvarphi}
T_{f_N\circ\varphi}(R)
&\ls\sum_{J\in\mathcal D(R)} \left( \frac{|J|}{|R|} \right)^{1-\frac{2\alpha}{n}} \left(\sum_{\gfz{I\in\mathcal D(Q_0)}{I\subset \varphi(\kappa J)}} d_I \left(\frac{|\varphi^{-1}(I)|}{|J|}\right)^{1-\varepsilon}\right)\\
&\qquad
+\sum_{J\in\mathcal D(R)} \left( \frac{|J|}{|R|} \right)^{1-\frac{2\alpha}{n}}
\left(\sum_{\gfz{I\in\mathcal D(Q_0)}{E_\tau(I)\supset J}}
d_I \left(\frac{|\varphi(J)|}{|I|}\right)^{\frac {2-2\varepsilon} n} \right)\notag\\
&=: {\rm Z}_1+{\rm Z}_2, \notag
\end{align}
where $\varepsilon\in(0,1/2)$ is a small number satisfying \eqref{eq-ez-range}. Thus, it remains to show that ${\rm Z}_i\ls 1$ for $i=1,2$.

To estimate ${\rm Z}_1$, applying the Fubini theorem gives
\begin{align*}
  {\rm Z}_1 &=  \sum_{\gfz{I\in\mathcal D(Q_0)}{I\subset \varphi(\kappa R)}}
  d_I \left(\frac{|\varphi^{-1}(I)|}{|R|} \right)^{1-\frac{2\alpha}{n}}
  \left(\sum_{\gfz{J\in\mathcal D(R)}{\varphi(\kappa J)\supset I}} \left(\frac{|\varphi^{-1}(I)|}{|J|}\right)^{\frac{2\alpha}{n}-\varepsilon}\right).
\end{align*}
By \eqref{eq-ball-varphi} and \eqref{eq-diamEF}, we easily see that $$\varphi(\kappa J)\subset B\Big(\varphi(c_J), \  \gamma \operatorname{diam}(\varphi(J))\Big)$$
for some positive constant $ \gamma=\gamma(n,\kappa,\eta)$.
Applying this, together with the fact that $\frac{2\alpha}{n}-\varepsilon >0$ (see \eqref{eq-ez-range}),
and using \eqref{eq-upward-sum} (with $\varphi^{-1}$ replaced by $\varphi$), we get
$$
\sum_{\gfz{J\in\mathcal D(R)}{\varphi(\kappa J)\supset I}} \left(\frac{|\varphi^{-1}(I)|}{|J|}\right)^{\frac{2\alpha}{n}-\varepsilon}
\ls 1.
$$
Consequently, we have
\begin{align}\label{eq-Z11}
  {\rm Z}_1 \ls   \sum_{\gfz{I\in\mathcal D(Q_0)}{I\subset \varphi(\kappa R)}}
  d_I \left(\frac{|\varphi^{-1}(I)|}{|R|} \right)^{1-\frac{2\alpha}{n}}.
\end{align}
Following the same arguments as in the proof of \eqref{eq-termI},
we cover $\varphi(\kappa R)$ by finitely many dyadic cubes \(Q_1,\dots, Q_N\) with \(N\le 2^n\) such that
$$
\operatorname{diam}\varphi(\kappa R) \simeq  \operatorname{diam} Q_i
$$
and
$$
|R| \simeq |\varphi^{-1}(Q_i)|.
$$
Since the zero-extended sequence $\{d_I\}_{I\in\cd(\rn)}$ satisfies the $(\beta_\alpha,\,\varphi)$-packing condition on every dyadic cube $Q\in\cd(\rn)$
(see Remark \ref{rem-0extension}), it follows that
 \begin{align}\label{eq-Z1}
  {\rm Z}_1
  &\ls  \sum_{i=1}^N \sum_{\gfz{I\in\mathcal D(Q_0)}{I\subset Q_i}}
  d_I \left(\frac{|\varphi^{-1}(I)|}{|R|} \right)^{1-\frac{2\alpha}{n}}\\
  &\simeq \sum_{i=1}^N \sum_{\gfz{I\in\mathcal D(Q_0)}{I\subset Q_i}}
  d_I \left(\frac{|\varphi^{-1}(I)|}{|\varphi^{-1}(Q_i)|} \right)^{1-\frac{2\alpha}{n}}\notag\\
  &\ls 1. \notag
\end{align}

Next, we estimate ${\rm Z}_2$.  Applying the Fubini theorem, we write
\begin{align}\label{eq-Z20}
{\rm Z}_2 & =\sum_{I\in\cd(Q_0)}
d_I  \left(
\sum_{\gfz{J\in\mathcal D(R)}{J\subset E_\tau(I)}} \left( \frac{|J|}{|R|} \right)^{1-\frac{2\alpha}{n}}
\left(\frac{|\varphi(J)|}{|I|}\right)^{\frac {2-2\varepsilon} n}
\right).
\end{align}
Note  that all the cubes $Q_0, I, J$ in \eqref{eq-Z20} belong to $\mathcal D(\rn)$.
For any \(I\in\mathcal D(Q_0)\),
the inner sum over \(J\) in \eqref{eq-Z20} is nonzero only if there exists some \(J\in\mathcal D(R)\) with \(J\subset E_\tau(I)\);
otherwise, \(I\) contributes nothing.
Thus, the summation over \(I\) in \eqref{eq-Z20} is restricted to those \(I\) for which
\begin{align*}
R\cap E_\tau(I)=R\cap B\Bigl(\varphi^{-1}(c_I),\
\tau\operatorname{diam}\varphi^{-1}(I)\Bigr)\neq\emptyset.
\end{align*}
We then estimate the inner sum over \(J\) by applying Lemma \ref{lem-QS-est2} with
with $\beta=1-\frac{2\alpha}{n}$ and $\gamma=\frac {2-2\varepsilon} n$.
Clearly, \(\beta\in(0,1)\).
From the choice of \(\varepsilon\) in \eqref{eq-ez-range},
it follows that $\gamma>1$ when $n=1$, and $\gamma\in(1-\beta,1)$ when
 $n\ge 2$.
Thus Lemma \ref{lem-QS-est2} is applicable for such \(\beta\) and \(\gamma\).
Consequently, we obtain
\begin{align}\label{eq-Z21}
{\rm Z}_2 \ls \sum_{\gfz{I\in\cd(\rn)}{R\cap E_\tau(I) \neq\emptyset}}
d_I
\min\left\{ \left(\frac{|\varphi^{-1}(I)|}{|R|}\right)^{1-\frac{2\alpha}{n}},\ \
   \left(\frac{|\varphi(R)|}{|I|}\right)^{\frac{2-2\varepsilon}n}
  \right\}.
  \end{align}
For $I\in\cd(\rn)$ satisfying $R\cap E_\tau(I) \neq\emptyset$, we examine the cases \(|\varphi^{-1}(I)|\le |R|\) and \(|\varphi^{-1}(I)|> |R|\) separately.

\begin{itemize}
  \item {\bf The case $R\cap E_\tau(I) \neq\emptyset$ and $|\varphi^{-1}(I)|\le |R|$.}
  In this case, applying \eqref{eq3-ball-varphi} gives
$$\operatorname{diam} \varphi^{-1}(I) \ls \ell(R).$$
Note that $E_\tau(I)$ is ball of diameter $2\tau\operatorname{diam} \varphi^{-1}(I)$.
This, together with $R\cap E_\tau(I) \neq\emptyset$, implies that
$$
E_\tau(I)\subset \lambda_1 R
$$
for some large constant $\lambda_1 \in(1,\infty)$.
  We may assume that  \(\tau>1\), so \(\varphi^{-1}(I)\subset E_\tau(I)\) by \eqref{eq-cube-varphi}. Thus,
\begin{align}\label{eq-case1}
\varphi^{-1}(I)\subset E_\tau(I)\subset \lambda_1 R.
\end{align}

  \item {\bf The case $R\cap E_\tau(I) \neq\emptyset$ and $|\varphi^{-1}(I)|> |R|$.}
    In this case, in view of \eqref{eq3-ball-varphi}, we have
$$\operatorname{diam} \varphi^{-1}(I) \gs \ell(R).$$
Again, using the fact that $R\cap E_\tau(I)\neq\emptyset$, we can find a constant $C_\ast\in(0,\infty)$ satisfying
\begin{align}\label{eq-R-inclu1}
R\subset C_\ast E_\tau(I)=B\Bigl(\varphi^{-1}(c_I),\ C_\ast \tau \operatorname{diam}\varphi^{-1}(I)\Bigr).
\end{align}
Let \(c_{\eta,n}\) be the constant determined in \eqref{eq-cube-varphi}. Choose $\lambda_2\in(1,\infty)$  sufficiently large so that
$$
\frac 2{\lambda_2}<\eta^{-1}\left(\frac{c_{\eta,n}}{C_\ast \tau}\right).
$$
For such $\lambda_2$, applying  \eqref{eq-diamEF} and the monotonicity of \(\eta\) yields
$$
\frac{\operatorname{diam} \varphi^{-1}(I)}{\operatorname{diam} \varphi^{-1}(\lambda_2I)} \le \eta\left(\frac 2{\lambda_2}\right)<\frac{c_{\eta,n}}{C_\ast \tau}.
$$
Combining this with \eqref{eq-R-inclu1} and \eqref{eq-ball-varphi} gives
\begin{align}\label{eq-R-inclu2}
R\subset  B\Bigl(\varphi^{-1}(c_I),\ c_{\eta, n} \operatorname{diam}\varphi^{-1}(\lambda_2I)\Bigr)
\subset \varphi^{-1}(\lambda_2I).
\end{align}
\end{itemize}
Based on the above two cases, we now  the summation in \eqref{eq-Z21} accordingly and write
\begin{align*}
{\rm Z}_2
&\ls  \sum_{\gfz{I\in\cd(\rn),\,R\cap E_\tau(I) \neq\emptyset}{|\varphi^{-1}(I)|\le |R|}}
d_I  \left(\frac{|\varphi^{-1}(I)|}{|R|}\right)^{1-\frac{2\alpha}{n}}
+\sum_{\gfz{I\in\cd(\rn),\,R\cap E_\tau(I) \neq\emptyset}{|\varphi^{-1}(I)|> |R|}}
d_I
   \left(\frac{|\varphi(R)|}{|I|}\right)^{\frac{2-2\varepsilon}n}
\\
&=:{\rm Z}_{21}+{\rm Z}_{22}. \notag
\end{align*}
Similar to the treatment of the term ${\rm Z}_1$ in \eqref{eq-Z11}, applying \eqref{eq-case1} yields
$$
{\rm Z}_{21}
\le \sum_{\gfz{I\in\cd(\rn)}{I\subset\varphi(\lambda_1 R)}}
d_I\left( \frac{|\varphi^{-1}(I)|}{|R|} \right)^{1-\frac{2\alpha}{n}}
\ls 1.
$$
For the term ${\rm Z}_{22}$, using the fact $d_I\le 1$ from Remark \ref{rem-packing}, together with \eqref{eq-R-inclu2},  we obtain
$$
{\rm Z}_{22}
\le \sum_{\gfz{I\in\cd(\rn)}{R\subset \varphi^{-1}(\lambda_2I)}}
\left(\frac{|\varphi(R)|}{|I|}\right)^{\frac{2-2\varepsilon}n}
= \sum_{k\in\zz} \sum_{\gfz{I\in\cd_k(\rn)}{R\subset \varphi^{-1}(\lambda_2I)}}
\left(\frac{|\varphi(R)|}{2^{-kn}}\right)^{\frac{2-2\varepsilon}n}.
$$
By the same argument as in the proof of \eqref{eq-k-number}, we have
\begin{align*}
\#\left(\left\{I\in\cd_k(\rn):\ \varphi^{-1}(\lambda_2I)\supset R\right\}\right)\le N(n,\lambda_2,\eta).
\end{align*}
For \(I\in\mathcal D_k(\mathbb R^n)\), the condition \(R\subset \varphi^{-1}(\lambda_2I)\) is equivalent to \(\varphi(R)\subset \lambda_2I\), which implies
$$|\varphi(R)|\le (\lambda_2 2^{-k})^n.$$
From these facts, it follows that
\begin{align*}
  {\rm Z}_{22}
  \le N(n,\lambda_2,\eta) \sum_{\{k\in\zz:\  2^{kn}\le \lambda_2^n /|\varphi(R)|\}}
\left(\frac{|\varphi(R)|}{2^{-kn}}\right)^{\frac{2-2\varepsilon}n}
\ls 1.
\end{align*}
Therefore, combining the estimates for \(Z_{21}\) and \(Z_{22}\) yields
\begin{align}\label{eq-Z2}
{\rm Z}_2 \ls 1.
\end{align}

Finally, inserting \eqref{eq-Z1} and \eqref{eq-Z2} into \eqref{eq2-fvarphi} gives
$T_{f_N\circ\varphi}(R) \lesssim 1.$
This proves \eqref{eq1-fvarphi}, and hence completes the proof of \eqref{eq-fvarphi}.
\end{proof}

\subsection{Proof of the necessity part of Theorem \ref{thm-trace-gauge}}\label{ss5.4}

\begin{proof}[Proof of Theorem \ref{thm-trace-gauge}: the necessity part]
For simplicity, denote
$$
\|\mathcal C_\varphi\|:=\|\mathcal C_\varphi\|_{\mathcal Q_\alpha(\mathbb R^n)\to \mathcal Q_\alpha(\mathbb R^n)}.
$$
By the assumption of Theorem \ref{thm-trace-gauge}, for every $g\in\mathcal Q_\alpha(\rn)$, we have
\begin{align}\label{eq-bdd}
\|g\circ\varphi^{-1}\|_{\mathcal Q_\alpha(\rn)}\le \|\mathcal C_\varphi\|\, \|g\|_{\mathcal Q_\alpha(\rn)}.
\end{align}

Fix a dyadic cube $Q_0\in\mathcal D(\rn)$ and suppose that $\{d_I\}_{I\in\mathcal D(Q_0)}$
satisfies the $(\beta_\alpha,\,\varphi)$-packing condition on $Q_0$,
where $\beta_\alpha:=(n-2\alpha)/n$.
Associated to this sequence, define the function $f_N$ as in \eqref{eq-fN-dI}.
Since each $f_N$ in \eqref{eq-fN-dI} is a finite linear combination of compactly supported smooth functions, we have $f_N\in \mathcal Q_\alpha(\rn)$.

Now, we apply Theorem \ref{thm-upperf} to the function
$g:=f_N\circ\varphi.$
On the one hand, applying \eqref{eq-fvarphi} yields
$$
\|g\|_{\mathcal Q_\alpha(\rn)}=\|f_N\circ\varphi\|_{\mathcal Q_\alpha(\rn)}\ls 1.
$$
On the other hand, applying \eqref{eq-fNinQ}  yields
$$
\|g\circ\varphi^{-1}\|_{\mathcal Q_\alpha(\rn)}^2
=\|f_N\|_{\mathcal Q_\alpha(\rn)}^2
\gs \sum_{\gfz{I\in\mathcal D(Q_0)}{\ell(I)\ge 2^{-N}\ell(Q_0)}}
d_I \left(\frac{|I|}{|Q_0|}\right)^{1-\frac{2\alpha}{n}}.
$$
In the last two estimates, the implicit constants are independent of $Q_0$, $N$,
and the sequence $\{d_I\}_{I\in\mathcal D(Q_0)}$.
Combining these two estimates with \eqref{eq-bdd}, we obtain
\begin{align}\label{eq1-bdd}
\sum_{\gfz{I\in\mathcal D(Q_0)}{\ell(I)\ge 2^{-N}\ell(Q_0)}}
d_I \left(\frac{|I|}{|Q_0|}\right)^{1-\frac{2\alpha}{n}}
\lesssim \|\mathcal C_\varphi\|^2.
\end{align}
Recall that $\beta_\alpha=(n-2\alpha)/{n}$.
Letting $N\to\infty$ in \eqref{eq1-bdd} and taking the supremum over
all sequences $\{d_I\}_{I\in\mathcal D(Q_0)}$ satisfying the $(\beta_\alpha,\,\varphi)$-packing condition on $Q_0$, we deduce$$
\|\varphi\|_{\mathscr G_{1-\frac{2\alpha}n}(Q_0)} \ls \|\mathcal C_\varphi\|^2.
$$
This, together with Theorem \ref{thm-dyadic-gauge}, yields
\begin{align}\label{eq-operNorm>gauge}
\|\varphi\|_{\mathscr G_{1-\frac{2\alpha}n}}
&\simeq \sup_{Q_0\in\mathcal D(\rn)}\|\varphi\|_{\mathscr G_{1-\frac{2\alpha}n}(Q_0)}
\lesssim \|\mathcal C_\varphi\|^2.
\end{align}
Thus, we complete the proof of the necessity part of Theorem \ref{thm-trace-gauge}.
\end{proof}

\begin{remark}
  From \eqref{eq-operNorm<gauge} and \eqref{eq-operNorm>gauge}, it follows  the equivalence
  in \eqref{eq-norm-gauge}.
\end{remark}

\section{Applications}\label{sec6}

\subsection{Composition stability of finite capacity gauge}\label{ss6.1}

As a consequence of Theorem \ref{thm-trace-gauge}, under suitable conditions, the finiteness of the $\beta$-capacity gauge $\|\cdot\|_{\mathscr G_\beta}$ is closed under composition.

\begin{corollary}\label{cor1-gauge}
Let $\beta\in(1-\frac 2n,\, 1)$ for $n\ge 2$, or $\beta\in(0,1)$ for $n=1$. Assume that $\varphi$ and $\psi$ are quasisymmetric mappings on $\rn$, with the additional assumption that $J_{\varphi}, J_\psi, J_{\varphi^{-1}}, J_{\psi^{-1}}\in A_\infty(\mathbb R)$ when $n=1$. If
$\|\varphi\|_{\mathscr G_\beta}<\infty$ and $\|\psi\|_{\mathscr G_\beta}<\infty$,
then
$
\|\varphi\circ\psi\|_{\mathscr G_\beta}<\infty.
$
\end{corollary}

\begin{proof}
Observe that the composition of two quasisymmetric mappings remains quasisymmetric (see, for  example, \cite[Proposition~10.6]{Heinonen2001book}). Consequently, \(\varphi\circ\psi\) is quasisymmetric on \(\mathbb R^n\).

Let \(n=1\). Recall that \(A_\infty(\mathbb R)\) weights are characterized by the estimate in \eqref{eq-Apweight}.
Since \(J_\varphi, J_\psi\in A_\infty(\mathbb R)\), it follows from \eqref{eq-Apweight} that there exist constants $C_\varphi,\,C_\psi\in(0,\infty)$
and $\delta_\varphi,\,\delta_{\psi}\in(0,1)$ such that,
for any interval \(I\subset\rr\) and any measurable  set \(S\subset I\),
$$
\frac{|\varphi(S)|}{|\varphi(I)|} \le C_\varphi \left(\frac{|S|}{|I|}\right)^{\delta_{\varphi}}
\quad\, \text{and}\quad \,
\frac{|\psi(S)|}{|\psi(I)|} \le C_\psi \left(\frac{|S|}{|I|}\right)^{\delta_{\psi}}.
$$
Consequently, for any interval \(I\subset\mathbb R\) and any measurable set \(S\subset I\),
\begin{align*}
 \frac{|\varphi\circ\psi(S)|}{|\varphi\circ\psi(I)|} \le C_\varphi
 \left(\frac{|\psi(S)|}{|\psi(I)|}\right)^{\delta_{\varphi}}
 \le C_\varphi C_\psi^{\delta_\varphi}  \left(\frac{|S|}{|I|}\right)^{\delta_{\psi}\delta_{\varphi}}.
\end{align*}
Thus \(J_{\varphi\circ\psi}\) satisfies the estimate in \eqref{eq-Apweight}, and hence \(J_{\varphi\circ\psi}\in A_\infty(\mathbb R)\).
The same argument, applied to \(\varphi^{-1}\) and \(\psi^{-1}\) under the additional assumptions \(J_{\varphi^{-1}},\, J_{\psi^{-1}}\in A_\infty(\mathbb R)\) when \(n=1\), yields \(J_{(\varphi\circ\psi)^{-1}}\in A_\infty(\mathbb R)\).

Now, by the assumption on \(\beta\), there exists some \(\alpha\in(0,\min\{1,n/2\})\) such that \(\beta=1-\frac{2\alpha}{n}\). Applying Theorem \ref{thm-trace-gauge} yields that both \(\mathcal C_\varphi\) and \(\mathcal C_\psi\) are bounded on \(\mathcal Q_\alpha(\mathbb R^n)\). Since
$$
\mathcal C_{\varphi\circ\psi}(f)
= f\circ (\varphi\circ\psi)^{-1}
= f\circ \psi^{-1}\circ \varphi^{-1}
= \mathcal C_\varphi(f\circ \psi^{-1})
= (\mathcal C_\varphi \circ \mathcal C_\psi)(f),
$$
it follows that \(\mathcal C_{\varphi\circ\psi}\) is also bounded on \(\mathcal Q_\alpha(\mathbb R^n)\). Applying Theorem \ref{thm-trace-gauge} again, we obtain
$$
\|\varphi\circ\psi\|_{\mathscr G_{1-\frac{2\alpha}{n}}}<\infty,
$$
as desired.
\end{proof}

\subsection{Invariance of removability under finite capacity gauge}\label{ss6.2}

\begin{definition}\label{def-capQ}
Let \(\alpha \in (0,\,\min\{1,\,n/2\})\). For any precompact open set $\Omega\subset\rn$ and any compact set $K\subset\Omega$, define
\begin{align}\label{def-cap-compact}
    \operatorname{Cap}_{\mathcal Q_\alpha}(K;\,\Omega)
    =
    \inf\left\{\|f\|_{\mathcal  Q_\alpha(\rn)}^2:\ f\in C_c(\Omega),\ f\ge 1 \text{ on } K\right\}.
\end{align}
    Further, for any open set $O\Subset\Omega$,  define
\begin{align*}
\operatorname{Cap}_{\mathcal Q_\alpha}(O;\,\Omega)
:= \sup_{\text{compact}\,K\subset O}  \operatorname{Cap}_{\mathcal Q_\alpha}(K;\,\Omega).
\end{align*}
Moreover, for an arbitrary set $E\Subset\Omega$, define
\begin{align}\label{def-cap-general}
\operatorname{Cap}_{\mathcal Q_\alpha}(E;\,\Omega)
:= \inf_{\gfz{\text{open}\,O\supset E}{O\Subset\Omega}} \operatorname{Cap}_{\mathcal Q_\alpha}(O;\,\Omega).
\end{align}
\end{definition}

A standard argument shows that, for a compact set \(K\subset\mathbb R^n\), the definition in \eqref{def-cap-general} coincides with the original definition in \eqref{def-cap-compact}.

\begin{definition}\label{def-removable}
  Let \(\alpha \in (0,\,\min\{1,\,n/2\})\) and  $\Omega\subset\rn$ be a precompact open set. A subset $E\Subset\Omega$ is called
  \emph{$\mathcal Q_\alpha$-removable in $\Omega$} if
  $\operatorname{Cap}_{\mathcal Q_\alpha}(E;\,\Omega)=0.$
\end{definition}

The following theorem is a geometric application of Theorem \ref{thm-trace-gauge}, which proves that finiteness of the capacity gauge suffices for transferring removability from Euclidean sets to their distorted images.

\begin{theorem}\label{thm-remov}
  Let \(\alpha \in (0,\,\min\{1,\,n/2\})\) and \(\varphi:\ \mathbb R^n \to \mathbb R^n\)
 be  $\eta$-quasisymmetric,
 with the additional assumption that \(J_\varphi, J_{\varphi^{-1}} \in A_\infty(\mathbb R)\)
 when \(n = 1\).
 Then, for any precompact open set $\Omega\subset\rn$ and any subset $E\Subset\Omega$,
\begin{align}\label{eq-cap-gauge}
\operatorname{Cap}_{\mathcal  Q_\alpha}(\varphi(E);\,\varphi(\Omega)) \le C \|\varphi\|_{\mathscr G_{1-\frac{2\alpha}{n}}}\, \operatorname{Cap}_{\mathcal Q_\alpha}(E;\,\Omega),
\end{align}
where \(C=C(n,\alpha,\eta, [J_{\varphi}]_{A_\infty(\rn)}, [J_{\varphi^{-1}}]_{A_\infty(\rn)})\) is a positive constant.
 Consequently, if
 $\|\varphi\|_{\mathscr G_{1-\frac{2\alpha}{n}}}<\infty,$
 then
 \begin{align}\label{eq-cap-revom}
 \text{\(E\) is $\mathcal Q_\alpha$-removable in $\Omega$}
 \ \ \Rightarrow\ \
 \text{\(\varphi(E)\) is $\mathcal Q_\alpha$-removable in $\varphi(\Omega)$}.
 \end{align}
\end{theorem}

\begin{proof}
Since \(\varphi\) is continuous on \(\mathbb R^n\), it maps the precompact open set \(\Omega\)
to the precompact open set \(\varphi(\Omega)\). Thus, the capacity
$\operatorname{Cap}_{\mathcal  Q_\alpha}(\varphi(E);\,\varphi(\Omega))$ is well-defined.

The implication in \eqref{eq-cap-revom} follows immediately from \eqref{eq-cap-gauge}.
To prove \eqref{eq-cap-gauge}, we may as well assume  that
$\|\varphi\|_{\mathscr G_{1-\frac{2\alpha}{n}}}<\infty;$
otherwise the right-hand side of \eqref{eq-cap-gauge} is infinite and the estimate is trivial.

By Definition \ref{def-capQ}, it suffices to show \eqref{eq-cap-gauge} for compact \(E\Subset\Omega\). Let \(f\in C_c(\Omega)\) with \(f\ge 1\) on \(K\).   Set
$$
h:=f\circ\varphi^{-1}.
$$
Since \(\varphi\) is a homeomorphism on $\rn$, it follows that $\varphi(\Omega)$ is a precompact open set, $\varphi(E)$ is a compact set, and $\varphi(E)\Subset\varphi(\Omega)$. Moreover, we have \(h\in C_c(\mathbb R^n)\) with \(h\ge 1\) on \(\varphi(E)\)
and
$$
\supp h = \varphi(\supp f)\subset \varphi(\Omega).
$$
This, together with Theorem \ref{thm-trace-gauge}, yields
$$
\operatorname{Cap}_{\mathcal Q_\alpha}(\varphi(E);\,\varphi(\Omega))
\le \|h\|_{\mathcal Q_\alpha(\rn)}^2
= \|f\circ\varphi^{-1}\|_{\mathcal Q_\alpha(\rn)}^2
\ls
\|\varphi\|_{\mathscr G_{1-\frac{2\alpha}{n}}} \,\|f\|_{\mathcal Q_\alpha(\rn)}^2.
$$
Taking the infimum over all \(f\in C_c(\Omega)\) with \(f\ge 1\) on \(E\), we conclude that
$$
\operatorname{Cap}_{\mathcal Q_\alpha}(\varphi(E);\,\varphi(\Omega))
\ls
\|\varphi\|_{\mathscr G_{1-\frac{2\alpha}{n}}} \, \operatorname{Cap}_{\mathcal Q_\alpha}(E;\,\Omega).
$$
Consequently, \eqref{eq-cap-gauge} holds for compact sets; the general case follows from the definition of capacity.
\end{proof}

\subsection{Applications to transport equations}\label{ss6.3}

In this subsection, we always let $n\ge 2$ and $b:\ [0,T]\times\mathbb R^n\to\mathbb R^n$ be a vector field.
Consider the transport equation
\begin{equation}\label{eq-transport-main}
\begin{cases}
\partial_t u(t,x) + b(t,x)\cdot \nabla_x u(t,x) = 0,\\
u(0,x)=u_0(x).
\end{cases}
\end{equation}
A function $u\in L^1(0,T;L^1_{\mathrm{loc}}(\mathbb R^n))$ is called a
\emph{weak solution} to \eqref{eq-transport-main} if
\begin{equation*}
\int_0^T\int_{\mathbb R^n} u\bigl(\partial_t\phi+\operatorname{div}(b\phi)\bigr)\,dx\,dt
+ \int_{\mathbb R^n} u_0(x)\phi(0,x)\,dx = 0
\end{equation*}
for every test function $\phi\in C_c^\infty([0,T)\times\mathbb R^n)$.

For initial data $u_0\in \mathrm{BMO}(\mathbb R^n)$,
under suitable regularity assumptions on $b$ (see Remark \ref{rem-flow} below),
the weak solution $u$ to \eqref{eq-transport-main}
admits the \emph{characteristic representation}
\begin{equation}\label{eq-sol}
u(t,x)=u_0\bigl(\Phi_t^{-1}(x)\bigr),
\end{equation}
where $\Phi_t$ is the flow generated by $b$, i.e.,
\begin{equation}\label{eq-flow-ode}
\begin{cases}
\partial_t \Phi_t(x)=b(t,\Phi_t(x)),\\
\Phi_0(x)=x.
\end{cases}
\end{equation}

\begin{definition}\label{def-QSflow}
A family of mappings $\{\Phi_t\}_{t\in[0,T]}$ with $\Phi_t:\ \mathbb R^n\to\mathbb R^n$
is called a \emph{quasisymmetric flow} generated by $b$ if it satisfies \eqref{eq-flow-ode} for a.e. $t\in[0,T]$ and $x\in\mathbb R^n$,
and if each $\Phi_t$ is quasisymmetric.
\end{definition}

\begin{remark}\label{rem-flow}
Definition \ref{def-QSflow} is natural in view of the known regularity theory for transport equations.
For a time-independent vector field $v:\ \mathbb R^n\to\mathbb R^n$, Reimann \cite{Reimann1976IM} introduced the condition
$$
\|v\|_Q := \sup_{\gfz{x\in\mathbb R^n}{|a|=|b|\neq 0}}
\left|
\frac{\bigl(a,\, v(x+a)-v(x)\bigr)}{|a|^2}
-
\frac{\bigl(b,\, v(x+b)-v(x)\bigr)}{|b|^2}
\right|
< \infty,
$$
which is now commonly referred to as \emph{condition $Q$}.
According to \cite[Theorem~1]{Reimann1976IM}, this condition is sufficient
to ensure that the flow generated by $v$
is quasisymmetric.
Further, by \cite[Theorem~3]{Reimann1976IM},
when the distributional derivatives of $v$
are locally integrable,
condition $Q$ is equivalent to the growth condition $v(x)=O(|x|\log |x|)$,
together with the essential boundedness of the \emph{anticonformal part} $Sv$, where
$$
S v:=\frac12\bigl(Dv+(Dv)^t\bigr)-\frac{\operatorname{div} v}{n}I_{n\times n}.
$$
As a time-dependent analogue of Reimann's criterion for $n\ge2$,
the flow $\{\Phi_t\}_{t\in[0,T]}$ generated by the vector field $b$ is quasisymmetric
if $b$ satisfies the following weaker assumptions (see \cite[Theorem~5]{ClopJiangMateuOrobitg2018ACV}):
\begin{align}\label{eq-reg}
\begin{cases}
b\in L^1(0,T;W^{1,1}_{\mathrm{loc}}(\mathbb R^n));\\
b(t,x)(1+|x|\log_+|x|)^{-1}\in L^1(0,T;L^\infty(\mathbb R^n));\\
S b\in L^1(0,T;L^\infty(\mathbb R^n)).
\end{cases}
\end{align}
Under the above regularity conditions,
the abstract framework of Definition \ref{def-QSflow} applies.
\end{remark}

\begin{theorem}\label{thm-transport-Qalpha}
Let $n\ge 2$ and $\alpha \in (0,\,1)$.
Suppose that  $u$
admits the characteristic representation \eqref{eq-sol}, with $\{\Phi_t\}_{t\in[0,T]}$
therein being a quasisymmetric flow generated by the vector field $b$. Assume that
\begin{equation}\label{eq-gauge-flow}
\|\Phi_t\|_{\mathscr G_{1 - \frac{2\alpha}{n}}}+
\|\Phi_t^{-1}\|_{\mathscr G_{1 - \frac{2\alpha}{n}}}  < \infty\quad\,
\text{for any}\ \, t\in[0,T].
\end{equation}
Then the following properties hold:
\begin{enumerate}[label=\textup{(\roman*)}]
\item \textbf{(Propagation of $\mathcal Q_\alpha$-regularity)}
For every initial datum $u_0\in\mathcal Q_\alpha(\mathbb R^n)$,
it holds that $u(t,\,\cdot)\in\mathcal Q_\alpha(\mathbb R^n)$ for all $t\in[0,T]$, with
the norm equivalence
\begin{equation}\label{eq-equiv-TP}
C^{-1} \left(\|\Phi_t^{-1}\|_{\mathscr G_{1 - \frac{2\alpha}{n}}}\right) ^{-\frac 12} \|u_0\|_{\mathcal Q_\alpha(\mathbb R^n)}
\le \|u(t,\,\cdot)\|_{\mathcal Q_\alpha(\mathbb R^n)}
\le C \left(\|\Phi_t\|_{\mathscr G_{1 - \frac{2\alpha}{n}}}\right) ^\frac 12 \|u_0\|_{\mathcal Q_\alpha(\mathbb R^n)}.
\end{equation}

\item \textbf{(Initial-data stability)}
If $u$ and $v$ are associated with initial data $u_0,v_0\in\mathcal Q_\alpha(\mathbb R^n)$,
respectively, then
\begin{align}\label{eq-stab-TP}
C^{-1} 
\left(\|\Phi_t^{-1}\|_{\mathscr G_{1 - \frac{2\alpha}{n}}}\right)^{-\frac 12}
\|u_0-v_0\|_{\mathcal Q_\alpha(\mathbb R^n)}.
& \le\|u(t,\,\cdot)-v(t,\,\cdot)\|_{\mathcal Q_\alpha(\mathbb R^n)}\\
&
\le C \left(\|\Phi_t\|_{\mathscr G_{1 - \frac{2\alpha}{n}}}\right) ^\frac 12
\|u_0-v_0\|_{\mathcal Q_\alpha(\mathbb R^n)}. \notag
\end{align}
\end{enumerate}
Here, the constants $C$ in \eqref{eq-equiv-TP} and \eqref{eq-stab-TP}
depend only on $n,\alpha$ and the distortion $\eta_t$ of $\Phi_t$.
If the family $\{\Phi_t\}_{t\in[0,T]}$
has a uniform distortion $\eta$ for all $t$,
then $C$ is also independent of $t$.
\end{theorem}

\begin{proof}
Fix $t\in[0,T]$ and $u_t(\cdot):=u(t,\cdot)$. To prove(i), by \eqref{eq-sol}, we have
$$
u_t=u_0\circ \Phi_t^{-1} = \mathcal C_{\Phi_t}(u_0)
$$
and
$$
u_0= u_t\circ \Phi_t = \mathcal C_{\Phi_t^{-1}}\left(u_t\right).
$$
Since $\Phi_t:\mathbb R^n\to\mathbb R^n$ is quasisymmetric, so is its inverse $\Phi_t^{-1}$.
Then, it follows from \eqref{eq-gauge-flow}  and
Theorem \ref{thm-trace-gauge} that \eqref{eq-equiv-TP} holds,
with implicit constants depend only on \(n,\alpha\) and the distortion  $\eta_t$ of each $\Phi_t$.

For (ii), the linearity of the composition operator gives
$$
u_t-v_t = (u_0-v_0)\circ \Phi_t^{-1} = \mathcal C_{\Phi_t}(u_0-v_0)
$$
and
$$
u_0-v_0 = (u_t-v_t)\circ \Phi_t = \mathcal C_{\Phi_t^{-1}}(u_t-v_t).
$$
As in the proof of (i), the above identities,
together with Theorem \ref{thm-trace-gauge}
applied to $u_0-v_0$,
yield the desired equivalence \eqref{eq-stab-TP}.
\end{proof}

Theorem~\ref{thm-transport-Qalpha} is stated within the abstract framework
that the vector field $b$ generates a quasisymmetric flow
with finite capacity gauge.
We now provide sufficient conditions on $b$ that ensure these assumptions.

\begin{corollary}\label{cor-Jacobian-A1}
Let $n\ge 2$ and $\alpha\in(0,1)$.
Suppose that the vector field $b$ satisfies \eqref{eq-reg}
and $\operatorname{div} b \in L^1(0,T;L^\infty(\mathbb R^n))$.
Then, for any initial datum $u_0\in\mathcal Q_\alpha(\mathbb R^n)$, the equation \eqref{eq-transport-main} has a unique weak solution,
which is given by \eqref{eq-sol} and
\begin{equation}\label{eq-ut-est}
\|u(t,\,\cdot)\|_{\mathcal Q_\alpha(\mathbb R^n)}
\simeq\|u_0\|_{\mathcal Q_\alpha(\mathbb R^n)}.
\end{equation}
Moreover, if $u$ and $v$ are two solutions with initial data $u_0,v_0\in\mathcal Q_\alpha(\mathbb R^n)$, respectively, then
\begin{equation}\label{eq-stab-TP-cor}
\|u(t,\cdot)-v(t,\cdot)\|_{\mathcal Q_\alpha(\mathbb R^n)}
\simeq \|u_0-v_0\|_{\mathcal Q_\alpha(\mathbb R^n)}.
\end{equation}
The implicit constants in \eqref{eq-ut-est} and \eqref{eq-stab-TP-cor} depend only on $n$, $\alpha$,
$\|S b\|_{L^1(0,T;L^\infty(\mathbb R^n))}$,
and $\|\operatorname{div} b\|_{L^1(0,T;L^\infty(\mathbb R^n))}$.
\end{corollary}

\begin{proof}
Given a vector field \(b\) satisfying \eqref{eq-reg},
it was proved in \cite[Theorem~5]{ClopJiangMateuOrobitg2018ACV}
that there exists a unique flow \(\{\Phi_t\}_{t\in[0,T]}\)
of uniformly \(K\)-quasiconformal mappings satisfying \eqref{eq-flow-ode}, with
$$
K = \exp\left(2(n-1)\|S b\|_{L^1(0,T;L^\infty(\mathbb R^n))}\right).
$$
Consequently, the family \(\{\Phi_t\}_{t\in[0,T]}\) has a
uniform distortion function \(\eta\), depending only on \(K\) and \(n\), for all \(t\in[0,T]\).

Since $b$ satisfies \eqref{eq-reg} and $\operatorname{div} b\in L^1(0,T;L^\infty(\mathbb R^n))$,
the DiPerna--Lions theory (see \cite[Section~III]{DiPernaLions1989IM}) guarantees
that $\{\Phi_t\}_{t\in[0,T]}$ is a regular Lagrangian flow.
In this framework, the Liouville equation
$$
\partial_t J_{\Phi_t}(x) = (\operatorname{div} b)(t,\Phi_t(x))\,J_{\Phi_t}(x)
$$
holds for a.e. $x\in\mathbb R^n$ and $t\in(0,T)$. Solving this ODE with the initial condition \(J_{\Phi_0}(x)=1\) gives
$$
J_{\Phi_t}(x) = \exp\left( \int_0^t (\operatorname{div} b)(s,\Phi_s(x))\,ds \right).
$$

Thus, upon setting
$$
D:=\|\operatorname{div} b \|_{L^1(0,T;L^\infty(\mathbb R^n))}
=\int_0^T \|(\operatorname{div} b)(s,\cdot)\|_{L^\infty(\mathbb R^n)}\,ds,
$$
we obtain
\begin{equation}\label{eq-JPhit}
e^{-D} \le J_{\Phi_t}(x) \le e^{D}
\quad\, \text{for  a.e. }\, x\in\mathbb R^n.
\end{equation}
For the inverse flow $\Phi_t^{-1}$, setting $y=\Phi_t(x)$
and using the identity
$J_{\Phi_t^{-1}}(y) = (J_{\Phi_t}(x))^{-1}$
gives
\begin{equation}\label{eq-JPhit-inver}
e^{-D} \le J_{\Phi_t^{-1}}(y) \le e^{D}
\quad\, \text{for  a.e. }\, x\in\mathbb R^n.
\end{equation}
Next, for any cube $Q\subset\mathbb R^n$,
the two-sided bounds in \eqref{eq-JPhit} imply
$$
\frac{1}{|Q|}\int_Q J_{\Phi_t}(x)\,dx \le e^{D}
\le e^{2D} \cdot e^{-D}
\le e^{2D} \operatorname*{ess\,inf}_{x\in Q} J_{\Phi_t}(x).
$$

Thus $J_{\Phi_t}$ satisfies the $A_1(\mathbb R^n)$ condition with constant
$$
[J_{\Phi_t}]_{A_1(\mathbb R^n)}\le e^{2D}.
$$
The same argument applies to $J_{\Phi_t^{-1}}$ by using \eqref{eq-JPhit-inver}.
Consequently, applying \eqref{eq-A1gauge} gives
\begin{align}\label{eq-flow-uniform-gauge}
\sup_{t\in[0,T]} \left(\|\Phi_t\|_{\mathscr G_{1 - \frac{2\alpha}{n}}}
+ \|\Phi_t^{-1}\|_{\mathscr G_{1 - \frac{2\alpha}{n}}}\right)
\le C_{n,\alpha} e^{2D(1 - \frac{2\alpha}{n})}.
\end{align}
In particular, the capacity gauge condition \eqref{eq-gauge-flow} holds uniformly for all $t\in[0,T]$.

Since $J_{\Phi_t}$ and $J_{\Phi_t^{-1}}$
belong to $A_1(\mathbb R^n)$, we follow the arguments in the proof of
\cite[Theorem~1.1(ii)]{Xiao2019JDE} to obtain that \eqref{eq-transport-main}
has a unique weak solution admitting the characteristic representation \eqref{eq-sol}.
Thus, all the assumptions in Theorem \ref{thm-transport-Qalpha} are satisfied.
Applying Theorem \ref{thm-transport-Qalpha} yields the two-sided uniform estimates
\eqref{eq-ut-est} and \eqref{eq-stab-TP-cor}.
\end{proof}

\begin{remark}
The restriction \(n\ge 2\) in Corollary \ref{cor-Jacobian-A1}
comes from \cite[Theorem~5]{ClopJiangMateuOrobitg2018ACV}.
The existence and uniqueness for initial data
in \(\mathcal Q_\alpha(\mathbb R^n)\) were
obtained in \cite{Xiao2019JDE} for all \(n\in\nn\),
but under the stronger Lipschitz assumption
$$\|b\|_{L^1(0,T;\mathrm{Lip}(\mathbb R^n))}<\infty.$$
Assuming the slightly weaker conditions \eqref{eq-reg} and
\(\operatorname{div} b\in L^1(0,T;L^\infty(\mathbb R^n))\),
Corollary \ref{cor-Jacobian-A1} establishes the uniform finite capacity gauge
condition \eqref{eq-gauge-flow} (see \eqref{eq-flow-uniform-gauge})
and then deduces the explicit quantitative estimates \eqref{eq-ut-est}--\eqref{eq-stab-TP-cor}.
It would be interesting to seek even weaker
conditions ensuring \eqref{eq-gauge-flow}
beyond the range of \(A_1(\mathbb R^n)\) weights of the Jacobian of the flow;
however, we do not pursue this here.
\end{remark}

\addcontentsline{toc}{section}{References}

%\bibliographystyle{amsplain}
%\bibliography{ref-RieszPotential}

\bigskip

\noindent Liguang Liu

\smallskip

\noindent School of Mathematics,
Renmin University of China,
Beijing 100872, People's Republic of China

\smallskip

\noindent{\it E-mail:} \texttt{liuliguang@ruc.edu.cn}

\bigskip

\noindent Jie Xiao
\smallskip

\noindent Department of Mathematics and Statistics,
Memorial University, St. John's,
NL A1C 5S7, Canada

\smallskip

\noindent{\it E-mail:} \texttt{jxiao@mun.ca}

\printindex

\end{document}